\documentclass{amsproc}

\usepackage[cmtip,all]{xy}

\newtheorem{theorem}{Theorem}[section]

\theoremstyle{definition}
\newtheorem{definition}[theorem]{Definition}

\theoremstyle{remark}

\numberwithin{equation}{section}
\usepackage[usenames,dvipsnames]{color}

\theoremstyle{plain}
\newtheorem{thm}{Theorem}[section]
\newtheorem{lem}[thm]{Lemma}

\newtheorem{cor}[thm]{Corollary}

\theoremstyle{definition}
\newtheorem{dfn}[thm]{Definition}
\newtheorem{ex}[thm]{Example}

\theoremstyle{remark}
\newtheorem{rmk}[thm]{Remark}

\newcommand{\CC}{\mathbf{C}}

\newcommand{\ZZ}{\mathbf{Z}}

\begin{document}

 \title[Vertex rings]{Vertex rings and their Pierce bundles}
\title{Vertex Rings and their Pierce Bundles}

\author{Geoffrey Mason}
\address{Department of Mathematics, University of California, Santa Cruz}
\curraddr{}
\email{gem@ucsc.edu}
\thanks{We thank the Simon Foundation, grant $\#427007$, for its support.}


\date{}

\begin{abstract} In part I we introduce \emph{vertex rings}, which bear the same relation to vertex algebras
(or VOAs) as commutative, associative rings do to commutative, associative algebras over $\mathbf{C}$.\ We show that vertex rings are characterized by \emph{Goddard axioms}.\ These include a generalization of the translation-covariance axiom
of VOA theory that involves a canonical \emph{Hasse-Schmidt derivation} naturally associated to any vertex ring.\
 We give several illustratory applications of these axioms, including the construction of vertex rings associated with the
 Virasoro algebra.\ We consider some categories of vertex rings, and the r\^{o}le played by the \emph{center} of a vertex ring.\ 
 In part II we extend the theory of \emph{Pierce bundles} associated to a commutative ring to the setting of vertex rings.\ This amounts  to the construction of certain \emph{reduced}  \'{e}tale bundles of vertex rings functorially associated to a vertex ring.\ 
 We introduce \emph{von Neumann regular} vertex rings as a generalization of  von Neumann regular commutative rings;  
we obtain a characterization of this class of vertex rings as those whose  Pierce bundles are bundles of \emph{simple} vertex rings.
\end{abstract}

\maketitle

\noindent
{Part I{:} Vertex Rings}

\medskip
\noindent
1.\ Introduction.\\
 2.\ Basic properties of vertex rings.\\
 2.1.\ Definition of vertex ring.\\
 2.2.\ Commutator, associator, and locality formulas.\\
 2.3.\ Vacuum vector.
 
 \medskip\noindent
 3.\ Derivations.\\
 3.1.\ Hasse-Schmidt derivations.\\
 3.2.\ Translation-covariance.
 
 \medskip\noindent
 4.\ Characterizations of vertex rings.\\
 4.1.\ Fields on an abelian group.\\
 4.2.\ Statement of the existence theorem.\\
 4.3.\ Residue products.\\
 4.4.\ The relation between residue products and translation-covariance.\\
 4.5.\ Completion of the proof of Theorem 4.3.\\
 4.6.\ Generators of a vertex ring and a refinement of Theorem 4.3.\\
 4.7.\ Formal Taylor expansion and an alternate existence Theorem.
 
 \medskip\noindent
 5.\ Categories of vertex rings.\\
 5.1.\ The category  of vertex rings.\\
 5.2.\ Commutative rings with HS derivation.

 \medskip\noindent
 6.\ The center of a vertex ring.\\
 6.1.\ Basic properties.\\
 6.2.\ Vertex $k$-algebras.\\
 6.3.\ Idempotents.\\
 6.4.\ Units.\\
 6.5.\ Tensor product of vertex rings.

 \medskip\noindent
 7.\ Virasoro vertex $k$-algebras.\\
 7.1.\ The Lie algebra $Vir_k$.\\
 7.2.\ The Virasoro vertex ring $M_k(c', 0)$.\\
 7.3.\ Virasoro vectors.\\
 7.4.\ Graded vertex rings.\\
 7.5.\ Vertex operator $k$-algebras.

\medskip\noindent
{Part II{:} Pierce bundles of vertex rings}
 
 \medskip\noindent
 8.\ {\'{E}tale bundles of vertex rings}.\\
 8.1.\ Basic definitions.\\
 8.2.\ Nonassociative vertex rings and sections.

 \medskip\noindent
 9.\ Pierce bundles of vertex rings.\\
 9.1.\ The Stone space of a vertex ring.\\
 9.2.\ The Pierce bundle of a vertex ring.\\
 9.3.\  Some local sections.
 
 \medskip\noindent
 10.\ Von Neumann regular vertex rings.\\
 10.1.\ Regular ideals.\\
 10.2. Von Neumann regular vertex rings.
 
 \medskip\noindent
 11.\ Equivalence of some categories of vertex rings.\\
 
 \medskip\noindent
 12.\ Appendix
 
  \begin{center}{\bf Part I{:} Vertex Rings}
 \end{center}

 \section{Introduction}
 The \textit{raison d'etre} for the present paper stems from the simple observation that the axioms for a vertex operator algebra
 (VOA) are \emph{integral}{:} there are no denominators.\ It is therefore meaningful to speak of a \emph{vertex ring} which, roughly speaking, is a VOA with an additive structure but not necessarily a  linear structure, and somewhat more precisely, it is an additive abelian group admitting a countable infinity of $\mathbf{Z}$-bilinear operations satisfying the same basic identity (sometimes called the \emph{Jacobi identity}) as a VOA.
 
  \medskip
  It is well-known that
 certain VOAs possess an \emph{integral structure}, i.e., a basis with respect to which the structure constants are integers, and the 
 $\mathbf{Z}$-span of such a basis is a vertex ring. (For example lattice theories have this property.)\ Dong and Griess have made a study of such integral forms invariant under a group action \cite{DG1}, \cite{DG2}.
 
 \medskip
  If the VOA $V$ has an integral structure
 and if $\tilde{V}$ is the $\mathbf{Z}$-span of an integral basis, the base-change $k\otimes\tilde{V}$ is a \emph{vertex $k$-algebra} for any commutative ring $k$, and the binary operations become $k$-linear.\  VOAs defined over base rings (or at least base fields) other than $\mathbf{C}$ occur frequently in the literature.\ One encounters base-changes such as
 $\mathbf{C}[t]\otimes V$ frequently, though they are often viewed as VOAs defined over $\mathbf{C}$.\ And in a slightly different direction, Dong and Ren \cite{DR} and Li and Mu \cite{LiM} have made interesting studies of Virasoro VOAs and Heisenberg VOAs respectively over base fields other than $\mathbf{C}$.\ 
 
 \medskip
 All of these examples point to the desirability of having available a general theory of vertex rings, and more generally 
 vertex $k$-algebras, and the purpose of the
 present paper is to make a start on such a theory.\ On the other hand, our original motivation for getting involved with
 such a project was quite different, and arose from the desire to extend some results in \cite{DM1} to a more general setting.\ There, Chongying Dong and I described the decomposition of a VOA into \emph{blocks} according to its (central) idempotents and I wanted to see what this theory would look if the VOA had a lot of idempotents.\ In order to even formulate precisely what this means one needs the general notion of a vertex ring.
 
 \medskip
 In the rest of this Introduction we will describe some of the content and main ideas of the present paper, which has two quite different parts.\ The first part deals with the axiomatics of vertex rings, the second with their so-called \emph{Pierce bundles}.\ As is well-known, one may obtain an important characterization of 
 vertex algebras (over $\mathbf{C}$) using the  \emph{Goddard axioms} \cite{G}, \cite{LL},  \cite{MN}.\ The 
 general idea is to show that the Jacobi identity for a  VOA is equivalent to a collection of \emph{mutually local, translation-covariant, creative fields}.\ Part I is mainly devoted to a generalization of this result to vertex rings and giving some applications.\ Most of the needed proofs dealing with \emph{locality} already exist in the literature and
carry over to the setting of vertex rings.\ However the  same is not true of
 translation-covariance.\ Translation-covariance for VOAs deals with a certain natural derivation, often denoted by $T$.\ For a vertex ring $V$, $T$ must be replaced by what we call the \emph{canonical Hasse-Schmidt derivation} of $V$, which is a certain sequence $\underline{D}{=}(D_0{:=}Id_V, D_1, D_2, ...)$ of endomorphisms $D_i$ of $V$
 satisfying \emph{Leibniz}, or \emph{divided power}, identities.\ We formulate a general translation-covariance axiom for vertex rings using the canonical HS derivation.\ This is carried out in Section 3.\ The introduction of the canonical HS derivation is very natural, and not without precedent.\  There is an extensive literature dealing with commutative rings with either a derivation or  HS-derivation
 \cite{EGA},\cite{Mats}.\ Indeed, pairs $(k, \underline{D})$ consisting of a (unital) commutative ring $k$ equipped with 
 HS-derivation $\underline{D}$ provide perhaps the easiest examples of  vertex rings that are \emph{not} VOAs.\

 \medskip
 Section 4 is taken up with the characterization of vertex rings \emph{a la} Goddard, using locality and our more general notion of translation-covariance.\ Here the exposition has been influenced by the presentation of
 Matsuo and Nagatomo \cite{MN}.\ We make several subsequent applications of this characterization.\ The first, in Section 4, deals with \emph{generating fields} for a vertex ring.\ This is the most transparent way to construct VOAs and vertex rings alike.
 
 \medskip
 The remainder of Part I is concerned with categories of vertex rings and related topics.\ Of paramount importance for everything that comes later is the idea of the \emph{center} $C(V)$ of a vertex ring $V$.\ This is concept is known in VOA theory \cite{DM1}, \cite{LL}, but its importance diminishes in the presence of denominators.\ One way to define $C(V)$, which is naturally a commutative ring and a vertex subring of
 $V$, is as the group of $\underline{D}$-constants of $V$.\ It is also the set of states with \emph{constant} vertex operator (cf.\ Theorem \ref{thmcenter1}).\ There is a categorical explanation for the importance of $C(V)$ that runs
 as follows:\ a unital, commutative ring $k$ is a vertex ring.\ Indeed, it corresponds to a pair $(k, \underline{D})$  where
 $\underline{D}{:=}(Id_k, 0, 0, ...)$ is the \emph{trivial} HS-derivation. Thus there is a functorial insertion 
 \begin{eqnarray}\label{Jinsertion}
K{:}\mathbf{Comm}{\rightarrow}\mathbf{Ver}
\end{eqnarray}
of the category $\mathbf{Comm}$ of unital, commutative rings into the category $\mathbf{Ver}$ of vertex rings.\ It is a basic fact that
in this way, $\mathbf{Comm}$ becomes a \emph{coreflective} subcategory of $\mathbf{Ver}$, i.e., the insertion $K$ has a \emph{right adjoint} (see \cite{Mac} for background).\ Indeed, the right adjoint is the center functor $C{:}\mathbf{Ver}{\rightarrow}\mathbf{Comm}$ that associates $C(V)$ to $V$.\
Section 6 is taken up with these issues and some other aspects of vertex rings that depend on the center functor.\ 
These include \emph{idempotents} and \emph{units} of  $V$, all of which turn out to  lie in $C(V)$.\ We also formally introduce vertex $k$-algebras.\ This could have been done from the outset in Section 1, but since we want to think of a commutative ring $k$ as a vertex ring it is more natural to define a vertex $k$-algebra as an object in the 
comma category $(k{\downarrow} \mathbf{Ver})$ of \emph{objects under $k$}.\ We treat \emph{tensor products} of vertex rings using our Goddard axioms; this is the \emph{coproduct} in $\mathbf{Ver}$.\ This construction, which can be awkward even in the setting of VOAs over  $\mathbf{C}$ (cf.\ \cite{FHL}), includes base changes
such as $R\otimes_k V$ ($R$ is a commutative $k$-algebra) that we mentioned before.

\medskip
Section 7 gives a more substantial application of the  Goddard axioms to the construction of \emph{Virasoro} vertex $k$-algebras
over an arbitrary base ring $k$.\ The Virasoro $k$-algebra ($k$ a commutative ring) is the Lie $k$-algebra $Vir$ with 
$k$-basis $L(n)\ (N{\in}\mathbf{Z})$ together with $K$, and  where the \emph{nontrivial} brackets are 
\begin{eqnarray}\label{Virrelns}
[L(m), L(n)]{=} (m-n)L(m+n){+}\frac{m^3-m}{6}\delta_{m{+}n, 0}c'K
\end{eqnarray}
($c'{\in}k$ is called the \emph{quasicentral charge} of $Vir$).\ Compared to the usual definition of the Virasoro algebra (\ref{Virrelns}) makes sense for \emph{any} $k$.\ It  amounts to a rescaling of the central element $K$ by a factor of $2$.\
To show that $\sum_n L(n)z^{-n-2}$ is a generating field for a vertex $k$-algebra, thanks to the Goddard axioms one only needs to demonstrate the existence of a
suitable HS derivation $\underline{D}{=}(Id, D_1, D_2, ...)$.\ We show that $\underline{D}$ exists and that
\begin{eqnarray}\label{Dmrelns}
L({-}1)^m{=}m!D_m\ (m{\in}\mathbf{Z}_{\geq 0}).
\end{eqnarray}

This construction allows us to define VOAs over an arbitrary commutative ring $k$ as
vertex $k$-algebras with a compatible $k$-grading and Virasoro vector whose modes satisfy
(\ref{Virrelns}).\ Vertex algebras over $\mathbf{C}$ and $Vir$ itself are the basic examples.\ We
include (\ref{Dmrelns}) as one of our VOA axioms.\ This seems natural, though experts may demur.\ In any case, we terminate our presentation of the axiomatics of vertex rings at this point.

\medskip
 The coreflective property of (\ref{Jinsertion}) suggests that $\mathbf{Ver}$ may be regarded as a \emph{natural extension} of 
 $\mathbf{Comm}$, and that certain theorems and/or theories
that hold in $\mathbf{Comm}$ might extend to $\mathbf{Ver}$.\ This point of view motivates Part II, where the main idea is to
demonstrate that some of the constructions and results in the remarkable paper \cite{RSP} of Pierce do indeed extend to 
$\mathbf{Ver}$.\ Pierce's paper concerns certain sheaves
of rings functorially associated to a commutative ring $k$.\ Actually, in keeping with standard practice at the time,  Pierce  did
not deal with sheaves \emph{per se} but rather with \emph{bundles}, and more precisely an equivalent category $\mathbf{redCommbun}$ whose objects are  \emph{reduced \'{e}tale bundles} of rings\footnote{In fact, the phrase `\'{e}tale bundle' never occurs in \cite{RSP}, where such things are called `sheaves'.}; `reduced' means
 that the bundles have a \emph{Boolean base space} (Hausdorff and totally disconnected) and \emph{indecomposable} fibers.\ One of the main 
 results of \cite{RSP} is an equivalence of categories $\mathbf{Comm}\stackrel{\sim}{\longrightarrow}\mathbf{redCommbun}$.\ Similarly, one of our main results in Part II is the extension of this result  to vertex rings{:}\ thus every vertex ring $V$ is canonically associated with
 a reduced \'{e}tale bundle $\mathcal{R}{\rightarrow}X$ of vertex rings.\ An important point is that the base $X$  is
 none other than the base of the Pierce bundle $E{\rightarrow}X$ associated to $C(V)$, namely $X{:=}Spec(B(C(V))$ where
 $B(k)$ for a commutative ring $k$ is a certain Boolean ring  whose elements comprise the idempotents of $k$.\ We call $X$  the \emph{Stone space} of $k$ (or $V$), being closely related to the duality theory of Marshall Stone.\ It is sometimes called the
 \emph{Boolean spectrum} of $k$.\
 The upshot is that there is a diagram of categories and functors that is discussed in Section \ref{Seqcat}{:}
\begin{eqnarray}\label{catdiag}
\xymatrix{
&\mathbf{Ver}\ar[r]^{\sim} &\mathbf{redVerbun} \\
&\mathbf{Comm}\ar[u]^K\ar[r]_{\sim} & \bf{redCommbun}\ar[u] 
 }
\end{eqnarray}
 where $\mathbf{redVerbun}$ is the category of reduced \'{e}tale bundles of vertex rings  and the horizontal functors are equivalences.
 
 \medskip
 For the purposes of the present paper it is crucial to deal with bundles  rather than the corresponding sheaf of sections.\ This is because the infinitely many operations in a vertex ring may lead to problems with
 sections over open sets, and one does not necessarily obtain a sheaf of vertex rings as the sheaf of sections but rather something weaker - what we call a sheaf of \emph{nonassociative vertex rings}.\ On the other hand the local sections over a \emph{closed} set \emph{do} carry the structure of a vertex ring.\ Since the base spaces we deal with are Boolean there is a basis of clopen sets,
 and this makes a sheaf perspective viable.\ Most of this is explained (with plenty of background) in the first two Sections of
 Part II.

\medskip
Pierce's theory works particularly well for commutative \emph{von Neumann regular rings}, and  
something similar  holds true for vertex rings.\ Thus in Section \ref{SvNr} we introduce \emph{von Neumann regular} (vNr) vertex rings.\ These are vertex rings $V$ such that every principal $2$-sided ideal has the form $e(-1)V$ for an idempotent $e$.\ We  establish various properties of vNr vertex rings.\ In particular
we show that in the upper equivalence of (\ref{catdiag}), the full subcategory of $\mathbf{Ver}$ whose objects are vNr vertex rings corresponds to the category of reduced \'{e}tale bundles of \emph{simple vertex rings}.\ This result is the main motivation for considering vNr vertex rings.\ Indeed, if $V$ is a simple vertex ring then $C(V)$ is a \emph{field}, so that simple vertex rings are 
the more familiar vertex algebras over a field, and vNr vertex rings are exactly those vertex rings whose Pierce bundle has 
such vertex algebras as stalks.\ As a special case, applying this when $V$ is a commutative vNr ring recovers Pierce's Theorem (\cite{RSP}, Theorem $10.3$).\ Pierce used this result to study \emph{modules} over a vNr ring, however we do not pursue the representation theory of vertex rings here.\
 
\medskip
The paper is \emph{expository} in nature, though proofs are almost always complete.\ (Section \ref{Seqcat} is an exception.)\  It  should be possible for nonexperts to follow the material, while experts will find much that is familiar.\ This approach is more-or-less forced upon us by the nature of the material:\ some of the existing proofs in the literature concerning VOAs work perfectly well in the setting of vertex rings, some work only with modification, some do not work at all.\ Under the circumstances, it seemed better to give a presentation starting from scratch.\ Part II is written assuming that the reader has no prior knowledge of bundleology.\ We have explained  the basic constructions in the context of vertex rings, though this is hardly different from bundles of commutative rings as discussed, for example, in \cite{MacM}.\ Pierce's original paper \cite{RSP} is also an excellent place to read about his construction (indeed, about bundles too), and 
we have borrowed shamelessly from this source, going so far as to use the same notation in some places.

\medskip
We thank Ken Goodearl for helpful conversations.

 \section{Basic properties of vertex rings}\label{Sbasic}
In this and the following few Sections we introduce vertex rings and show that
they consist of mutually local, creative, translation-covariant fields.

\subsection{Definition of vertex ring}

\begin{definition}\label{defvr} A \emph{vertex ring} is an additive abelian group $V$ equipped with biadditive products $(u, v) \mapsto u(n)v\ (u, v\in V)$ defined for \emph{all} $n\in\mathbf{Z}$, together with a distinguished
element $\mathbf{1}\in V$ (the \emph{vacuum element}).\ The following identities are required to hold for all $u, v, w\in V$:
\begin{eqnarray}
&&(a)\ \mbox{there is an integer}\ n_0(u, v)\geq 0\ \mbox{such that}\ u(n)v=0\ \mbox{for all}\ n\geq n_0. \notag\\
&&(b)\ u(-1)\mathbf{1}=u;\ u(n)\mathbf{1}=0\ \mbox{for}\ n\geq 0.\notag\\
&&(c)\ \forall r, s, t\in\mathbf{Z}, \notag\\
&&\ \ \ \ \ \ \ \ \ \ \ \ \ \ \ \ \ \ \ \  \sum_{i\geq 0} {r\choose i}(u(t+i)v)(r+s-i)w = \label{JI}\\
&&\ \ \ \ \ \sum_{i\geq 0}(-1)^i {t\choose i} \left\{u(r+t-i)v(s+i)w-(-1)^tv(s+t-i)u(r+i)w\right\}.\notag
\end{eqnarray}
\end{definition}

We refer to (\ref{JI}) as the \emph{Jacobi identity}.\ Thanks to (a),  the two sums in (c) make sense inasmuch as there are only finitely many nonzero terms.\ Similar comments will apply in numerous contexts in what follows, and we will generally not make this explicit.\ We call $u(n)v$ the \emph{$n^{th}$ product} of $V$.

\medskip
For an additive abelian group $V$, $End(V)$ denotes the set of endomorphisms of
$V$.\ It is an associative ring with respect to composition, and
a $\mathbf{Z}$-Lie algebra with respect to the usual bracket $[a, b] {:=} ab-ba$.

\medskip
 If $V$ is a vertex ring, we often refer to elements of $V$ \emph{states}, and call $u(n)$ the \emph{$n^{th}$ mode} of the state $u$.\ Because $u(n)v$ is additive in $v$, we may, and shall, regard
$u(n)$ as an \emph{endomorphism} in $End(V)$ for all $u\in V$ and $n\in\mathbf{Z}$.\ Then
(\ref{JI})(c) can be regarded as an identity in $End(V)$.

\medskip
The \emph{vertex operator} corresponding to $u\in V$ is the formal generating function defined by
\begin{eqnarray*}
Y(u, z) {:=} \sum_{n\mathbf{Z}} u(n)z^{-n-1}{\in}End(V)[[z, z^{-1}]]
\end{eqnarray*}
for an indeterminate $z$.\ Identities between endomorphisms of $V$ are conveniently written as identities involving vertex operators.\ To
facilitate this we use some `obvious' notations when dealing with vertex operators.\ For example, if
$u, v\in V$ then
\begin{eqnarray*}
Y(u, z)v{:=} \sum_{n\in\mathbb{Z}} u(n)vz^{-n-1}{\in}V[z^{-1}][z],
\end{eqnarray*}
the last containment being a consequence of (\ref{JI})(a).\ Similarly, (\ref{JI})(b) says that
\begin{eqnarray}\label{createform}
Y(u, z)\mathbf{1}{\in}u+ zV[[z]].
\end{eqnarray}
\begin{dfn}\label{dfncreate} We  paraphrase (\ref{createform}) by saying that \emph{$Y(u, z)$ is creative with respect to $\mathbf{1}$ and creates the state $u$}.
(A refinement  is discussed in Theorem \ref{thmHS}.)
\end{dfn}

\medskip
The additivity of $u(n)v$ in $u$ as well as $v$ permits us to promote $Y$
to a morphism of abelian groups
\begin{eqnarray*}
Y{:} V \longrightarrow End(V)[[z, z^{-1}]],\ u\mapsto Y(u, z).
\end{eqnarray*}
$Y$ is then called the \emph{state-field correspondence}.\ (Discussion of the word \emph{field} as it used here
is deferred until Section \ref{SSfields}.)

\begin{rmk} The state-field correspondence is \emph{injective}.
\end{rmk}
\begin{proof} Use the creativity of $Y(u, z)$.
\end{proof}

\subsection{Commutator, associator, and locality formulas}\label{SScal}
 We emphasize some particularly useful special cases of (\ref{JI}).\ The first two, 
 the \emph{commutator formula} and \emph{associator formula}, are obtained simply by setting $t{=}0$ and $r{=}0$ respectively.\ As identities in $End(V)$, they read as follows:
\begin{eqnarray}\label{commform}
[u(r), v(s)]{=} \sum_{i\geq 0} {r\choose i}(u(i)v)(r+s-i),
\end{eqnarray}
\begin{eqnarray}\label{assocform}
&&\ \ \ \  (u(t)v)(s) {=} \sum_{i\geq 0}(-1)^i {t\choose i} \left\{u(t-i)v(s+i)-(-1)^tv(s+t-i)u(i)\right\}.
\end{eqnarray}

\medskip
The third special case arises by choosing $t\geq 0$ large enough so that,
for a given pair of states $u, v$, we have $u(t+i)v=0$ for all $i\geq 0$.\ The existence of $t$
is guaranteed by (\ref{JI})(a), and with such a choice the left-hand-side of (\ref{JI})(c) vanishes.\
What pertains is the following formula, which holds for $t\gg0$:
\begin{eqnarray}\label{locform}
\sum_{i\geq 0}(-1)^i{t\choose i}\left\{u(r+t-i)v(s+i)-(-1)^tv(s+t-i)u(r+i)\right\}=0.
\end{eqnarray}

\medskip\noindent
This is more compelling when formulated in terms of vertex operators{:}
\begin{lem}\label{lemlocform}(Locality formula) If $u, v$ are states in a vertex ring then there is
an integer $t\gg 0$ (depending on $u$ and $v$) such that
\begin{eqnarray}\label{locform1}
(z-w)^t[Y(u, z), Y(v, w)]=0.
\end{eqnarray}
In other words, 
\begin{eqnarray}\label{weakcommform}
(z-w)^tY(u, z)Y(v, w)= (z-w)^tY(v, w)Y(u, z).
\end{eqnarray}
\end{lem}
\begin{proof} We have
\begin{eqnarray*}
(z-w)^t[Y(u, z), Y(v, w)]=\left\{\sum_{i=0}^t (-1)^i{t\choose i}z^iw^{t-i}\right\}\left\{\sum_{m, n} [u(m), v(n)]z^{-m-1}w^{-n-1}\right\},
\end{eqnarray*}
and the coefficient of $z^{-a-1}w^{-b-1}$ in this expression is
\begin{eqnarray*}
&& \sum_{i=0}^t (-1)^i{t\choose i}[u(i+a), v(t-i+b)]\\
&=&  \sum_{i=0}^t (-1)^{t-i}{t\choose i}u(t-i+a) v(i+b)-\sum_{i=0}^t (-1)^i{t\choose i}v(t-i+b)u(i+a).
\end{eqnarray*}
But this vanishes on account of (\ref{locform}) if $t$ is large enough.
\end{proof}

\begin{dfn}\label{dfnmlocal}
We paraphrase property (\ref{locform1}) by saying that $Y(u, z)$ and $Y(v, z)$ are \emph{mutually local
or order $t$}, or simply \emph{mutually local} if we do not wish to emphasize $t$.
\end{dfn}

\begin{rmk} (a)\  The identity (\ref{weakcommform}) is also sometimes called \emph{weak commutativity}.\
There is an analog, called \emph{weak associativity} or \emph{duality}, which can be proved by a similar argument that requires only slightly more effort.\ We will just state the result, which will not be used below{:}\
For fixed $u, v\in V$ and $t\gg 0$ we have
\begin{eqnarray}\label{wassocform}
(z+w)^tY(Y(u, z)v, w)=(z+w)^tY(u, z+w)Y(v, w),
\end{eqnarray}
where we are observing the convention (\ref{binexp}) for the binomial expansion of $(z+w)^n$. \\ 
\\
(b) As we will see, the Jacobi identity is more-or-less \emph{equivalent} to the conjunction
of weak commutativity and weak associativity.\ As in the classical theory of rings, it can be fruitful to consider
axiomatic set-ups  where weak associativity (but \emph{not} weak commutativity)
is assumed, leading to a theory of associative (but not commutative) vertex rings.\
 One might even consider
\emph{nonassociative vertex rings} where neither weak commutativity nor weak associativity pertain (but other relevant axioms hold).\ Such objects arise naturally as sheaves of sections of bundles of vertex rings.\ We discuss this further
in Section \ref{Setale}.\ The present work focuses almost exclusively on vertex rings (aka weak commutative, and associative vertex rings) as we have defined them.
\end{rmk}

We record a result which will be useful in several places. 
\begin{lem}\label{lemcommvector} The following are equivalent for a vertex ring $V$ and states $u, v{\in} V$.
\begin{eqnarray*}
&&(a)\ [u(r), v(s)]{=}0\ \mbox{for all}\ r, s{\in}\mathbf{Z},\\
&&(b)\ u(n)v {=} 0\ \mbox{for all}\ n\geq 0.
\end{eqnarray*}
\end{lem}
\begin{proof} This follows easily from  (\ref{commform}).
\end{proof}

\subsection{Vacuum vector}
We will prove
\begin{thm}\label{thmvacuum} For all $n{\in}\mathbf{Z}$ we have
$\mathbf{1}(n) {=} \delta_{n, -1}Id_V$.\ That is,
\begin{eqnarray*} 
 Y(\mathbf{1}, z){=}Id_V.
 \end{eqnarray*}
\end{thm}
\begin{proof} 
First take $v=w=\mathbf{1}, r=-1, t=0$ in (\ref{JI}) and use (b) to obtain 
\begin{eqnarray}\label{ideq}
u(-1)\mathbf{1}(s)\mathbf{1} = \mathbf{1}(s)u(-1)\mathbf{1} = \mathbf{1}(s)u
\end{eqnarray}
for all $u\in V$.\ Thus in order to prove the Theorem, it suffices to show that 
$\mathbf{1}(s)\mathbf{1}=\delta_{s, -1}\mathbf{1}$.\ By (\ref{JI})(b) we have $\mathbf{1}(s)u=0$ for $s \geq 0$ and any $u$,
so certainly $\mathbf{1}(s)=0$ for $s \geq 0$.\ Similarly, 
$\mathbf{1}(-1)\mathbf{1}=\mathbf{1}$.

\medskip
We prove that $\mathbf{1}(n)\mathbf{1}=0$ for $n\leq -2$ by induction on $-n$.\
First take $u=v=w=\mathbf{1}$ and $t=-1$ in (\ref{JI}) together with  (b) to see that
\begin{eqnarray}\label{calc1}
\mathbf{1}(r+s)\mathbf{1} = \sum_{i\geq 0} \left\{\mathbf{1}(r-1-i)\mathbf{1}(s+i)\mathbf{1}+\mathbf{1}(s-1-i)\mathbf{1}(r+i)\mathbf{1}\right\}
\end{eqnarray}
for all $r, s\in\mathbf{Z}$.\ If we first take $r=s=-1$ in (\ref{calc1}) we obtain
\begin{eqnarray*}
\mathbf{1}(-2)\mathbf{1}=\sum_{i\geq 0} \left\{\mathbf{1}(-2-i)\mathbf{1}(-1+i)\mathbf{1}+\mathbf{1}(-2-i)\mathbf{1}(-1+i)\mathbf{1}\right\} = 2\mathbf{1}(-2)\mathbf{1},
\end{eqnarray*}
whence $\mathbf{1}(-2)\mathbf{1}=0$.\ This begins the induction.\ Let $r+s=n\leq -2$ where we choose $0\leq r < -s-1$.\ (\ref{calc1}) then reads
\begin{eqnarray*}
\mathbf{1}(n)\mathbf{1} &=& \sum_{i\geq 0} \mathbf{1}(r-1-i)\mathbf{1}(s+i)\mathbf{1}.
\end{eqnarray*}

\medskip
Note that all of the modes $\mathbf{1}(r), \mathbf{1}(s)$ \emph{commute}
thanks to (\ref{commform}).\ By induction, it follows that in the previous display,
 the only possible nonzero terms on the right-hand-side come from
$i=r$ and $i=-s-1$, and in both cases these are equal to $\mathbf{1}(n)\mathbf{1}$.\
Thus we obtain $\mathbf{1}(n)\mathbf{1}=2\mathbf{1}(n)\mathbf{1}$ and therefore $\mathbf{1}(n)\mathbf{1}=0$.\ This
completes the proof of the Theorem.
\end{proof}

\section{Derivations}\label{SHS}
Derivations play a ubiquitous r\^{o}le in the theory of vertex rings. 

\subsection{Hasse-Schmidt derivations}
By \emph{nonassociative ring} we will always mean a not-necessarily associative ring $V$ that may not  have an identity, i.e.,
an additive abelian group equipped with a biadditive product $uv\ (u, v\in V)$.\
The main examples we use are commutative rings, which will \emph{always} mean commutative, associative rings with an identity; and vertex rings $V$ equipped with their $n^{th}$ product.

\begin{dfn}
(a)\ Let $V$ be a nonassociative ring.\ A \emph{derivation} of $V$ is an endomorphism $f{\in}End(V)$ such that
$f(uv){=}uf(v){+}f(u)v\ \ (u, v{\in}V).$\\
(b)\ Let $V$ be a vertex ring.\ A \emph{derivation} of $V$ is an endomorphisms $f{\in}End(V)$ such that
$f$ is a derivation of \emph{each} of the nonassociative rings defined by $V$ together with any of its $n^{th}$ products.\  In
other words, we have for all $n{\in}\mathbf{Z}$ and $u, v{\in}V$, 
\begin{eqnarray*}
f(u(n)v){=}u(n)f(v){+}f(u)(n)v
\end{eqnarray*}

\end{dfn}

In each case we let $Der(V)$ denote the set of all derivations of $V$.\ By a standard argument,
$Der(V)\subseteq End(V)$ is a Lie subalgebra.

\begin{dfn}\label{dfnHS}\ Let $V$ be an additive abelian group, and suppose that $\underline{D}{:=} (D_0, D_1, \hdots)$ is a sequence of endomorphisms $D_i\in End(V)$ with $D_0=Id_V$.\\
(a)\  If $V$ is a nonassociative ring, we call\footnote{To refer to $\underline{D}$ as a derivation is a convenient misnomer as only
$D_1$ is a true derivation.\ $\underline{D}$ is called a \emph{differentiation} in \cite{Mats}.}  $\underline{D}$ a \emph{Hasse-Schmidt} (HS) derivation of $V$ if, for all $u, v\in V$ and all $m\geq 0$, we have
\begin{eqnarray*}
D_m(uv)= \sum_{i+j=m} D_i(u)D_j(v).
\end{eqnarray*}
(b)\ If $V$ is a vertex ring, we call $\underline{D}$ a HS derivation of $V$ if, for every $n{\in}\ZZ$,
$\underline{D}$ is a HS derivation of the nonassociative ring defined by $V$ together with its $n^{th}$ product.\\
(c)\ $\underline{D}$ is called \emph{iterative} if, for all $i, j{\geq}0$, we have
\begin{eqnarray}\label{dfniter}
D_i\circ D_j {=} {i+j\choose i}D_{i+j}.
\end{eqnarray}
\end{dfn}

\begin{ex}\label{trivHS} (a)\ The \emph{trivial} HS derivation of $V$ is $\underline{D}{=}(Id_V, 0, 0, \hdots)$, in which
all higher $D_m\ (m\geq 1)$ are zero.\\
(b)\  If $\underline{D}$ is iterative then $D_1^m{=}m!D_m\ (m{\geq}0)$.
\end{ex}

 Using the binomial theorem, we obtain
\begin{lem}\label{rmkDinvert} Suppose that $\underline{D}=(Id_V, D_1, \hdots)$ is an iterative derivation.\ Then
\begin{eqnarray*}
\left\{\sum_{m=0}^{\infty}D_mz^m\right\}\left\{\sum_{m=0}^{\infty}D_m(-z)^m\right\}=Id_V.
\end{eqnarray*}
$\hfill \Box$
\end{lem}

For HS derivations in the theory of commutative rings, see \cite{Mats}.\ Iterative HS derivations arise naturally in vertex rings, as we now show.

\begin{thm}\label{thmHS} Let $V$ be a vertex ring, and for $m{\geq}0$ define  
$D_m{\in}End(V)$ by the formula $D_m(u){:=}u(-m-1)\mathbf{1}$, i.e.,
\begin{eqnarray*}
Y(u, z)\mathbf{1}=\sum_{m\geq 0} D_m(u)z^m.
\end{eqnarray*}
Then $\underline D{:=}(D_0, D_1, \hdots)$ is an iterative, HS derivation of $V$.
\end{thm}
\begin{proof} We use Theorem \ref{thmvacuum} repeatedly in what follows.\
The identification $D_0{=}Id_V$ follows from (\ref{JI})(b).\
For the iterative property, we have
\begin{eqnarray*}
D_{\ell}\circ D_m(u) 
&=&(u(-m-1)\mathbf{1})(-\ell-1)\mathbf{1} \\
&=&\sum_{i\geq 0}(-1)^i {-m-1\choose i} u(-m-1-i)\mathbf{1}(-\ell-1+i)\mathbf{1}\\
&=&(-1)^{\ell} {-m-1\choose \ell} u(-m-1-\ell)\mathbf{1} \\
&=& {\ell+m\choose \ell}D_{\ell+m}(u).
\end{eqnarray*}

As for the Hasse-Schmidt property, we first record
\begin{lem}\label{lemmaDiu} We have
\begin{eqnarray*}
D_i(u)(n) {=} (-1)^i{n\choose i}u(n-i).
\end{eqnarray*}
\end{lem}
\begin{proof} We have
\begin{eqnarray*}
(D_iu)(n) &=& (u(-i-1)\mathbf{1})(n) \\
&=&\sum_{j\geq 0}(-1)^j {-i-1\choose j} \left\{u(-i-1-j)\mathbf{1}(n+j)+(-1)^i\mathbf{1}(n-i-1-j)u(j)\right\}\\
&=&\sum_{j\geq 0}{i+j\choose j}\left\{u(-i-1-j)\mathbf{1}(n+j)+(-1)^i\mathbf{1}(n-i-1-j)u(j)\right\}\\
&=& (-1)^i{n\choose i}u(n-i)
\end{eqnarray*}
(check the cases $n\geq 0$ and $n<0$ separately).\ The Lemma is proved.
\end{proof}

To complete the proof of Theorem \ref{thmHS}, use Lemma \ref{lemmaDiu} and (\ref{assocform}) to obtain
\begin{eqnarray*}
D_mu(n)v &=& (u(n)v)(-m-1)\mathbf{1}\\
&=&\sum_{i\geq 0}(-1)^i {n\choose i} u(n-i)v(-m-1+i)\mathbf{1}\\
&=&\sum_{i\geq 0} (D_iu)(n)D_{m-i}v.
\end{eqnarray*}
This is the required Hasse-Schmidt property.
\end{proof}

\begin{dfn}\label{VringHSdef} If $V$ is a vertex ring, we call $\underline{D}$ defined as in Theorem \ref{thmHS}  
the \emph{canonical HS derivation of $V$}.\ 
\end{dfn}

A first example of the utility of the canonical HS derivation is the \emph{skew-symmetry} formula.

\begin{lem}\label{lemmaskewsymm} Let $V$ be a vertex ring with canonical HS derivation $\underline{D}$.\ Then for all $u, v{\in}V$ and $n{\in}\mathbf{Z}$ we have
\begin{eqnarray*}
v(n)u{=}(-1)^{n+1}\sum_{i\geq 0} (-1)^iD_i(u(n+i)v).
\end{eqnarray*}
In terms of vertex operators, this reads
\begin{eqnarray}\label{skewsymmform}
Y(v, z)u{=}\sum_{m\geq 0}z^mD_mY(u, -z)v.
\end{eqnarray}
\end{lem}

\begin{proof} Take $w{=}\mathbf{1}$ and $r{=}-1, s{=}0$ in (\ref{JI})(c) to obtain
\begin{eqnarray*}
\sum_{i\geq 0} (-1)^i(u(t+i)v)(-1-i)\mathbf{1} 
&=& \sum_{i\geq 0}(-1)^i {t\choose i} (-1)^{t+1}v(t-i)u(-1+i)\mathbf{1}\\
&=&(-1)^{t+1}v(t)u.
\end{eqnarray*}
Since the left-hand-side is equal to $\sum_{i\geq 0} (-1)^iD_i(u(t+i)v)$,  the Lemma follows.
\end{proof}

A standard way to consider the iterative property of the canonical HS derivation 
involves the  \emph{ring of divided powers} $\mathbf{Z}\langle x\rangle$.\ This is the  commutative ring generated by symbols $x^{[n]}\ (n\geq 0)$ subject to the identity
\begin{eqnarray*}
x^{[m]}x^{[n]} {=} {m+n \choose n}x^{[m+n]}.
\end{eqnarray*}
$\ZZ\langle x\rangle$ can be realized as the subring of 
$\mathbf{Q}[x]$ generated by $\frac{x^n}{n!}\ (n\geq 0)$.\
The iterative property of $\underline{D}$ immediately implies
\begin{lem}\label{lemdpmodule} Suppose that $V$ is a vertex ring
with canonical HS derivation $\underline{D}$.\ Then the association $D_n\mapsto x^{[n]}\ (n{\geq}0)$ makes $V$ into a \emph{left $\mathbf{Z}\langle x\rangle$}-module. $\hfill \Box$
\end{lem}
 
\subsection{Translation-covariance}\label{SStc}
Let  $V$ be a vertex ring with canonical HS derivation $\underline{D}$.\ For
a state $u{\in}V$ we set
\begin{eqnarray}\label{tc1}
\delta_z^{(i)}Y(u, z){:=} \frac{1}{i!}\partial_z^iY(u, z),
\end{eqnarray}
where $\partial_z$ is formal differentiation with respect to $z$.\
 Despite the appearance of $i!$ in the denominator, we have $\delta_z^{(i)}Y(u,z){\in}End(V)[[z, z^{-1}]]$, because 
\begin{eqnarray}\label{dact1}
\delta_z^{(i)}\left(\sum_n u(n)z^{-n-1}\right) &=& \sum_n {-n-1\choose i}u(n)z^{-n-i-1} \notag\\
&=& (-1)^i\sum_n{n\choose i}u(n-i)z^{-n-1}.
\end{eqnarray}
In fact, $\delta_z^{(i)}Y(u, z)$ is the vertex operator for a state in $V$.\ This is part of the next result.

\begin{thm}\label{lend} The following hold for all $u\in V$ and $m\geq 1$.
\begin{eqnarray*}
&&(a)\ D_m\mathbf{1}{=}0,\\
&&(b)\ Y(D_m(u), z) {=} \delta_z^{(m)}Y(u, z),\\
&&(c)\ [D_m, Y(u, z)] {=} \sum_{i=1}^m \delta_z^{(i)}Y(u, z)D_{m-i}.
\end{eqnarray*}
\end{thm}
\begin{proof} (a) amounts to $\mathbf{1}(-i-1)\mathbf{1}{=}0$ for $i{\geq}1$, which follows from
Theorem \ref{thmvacuum}.\ Part (b) follows from Lemma  \ref{lemmaDiu} and (\ref{dact1}).

\medskip
As for (c), use Theorem \ref{thmHS} to see that
\begin{eqnarray*}
[D_m, Y(u, z)]w &=& \sum_n \left(D_m(u(n)w)-u(n)D_m(w)\right)z^{-n-1}\\
&=& \sum_n\sum_{i=0}^m \left(D_i(u)(n)D_{m-i}(w)-u(n)D_m(w)\right)z^{-n-1}\\
&=& \sum_n\sum_{i=1}^m D_i(u)(n)D_{m-i}(w)z^{-n-1}.
\end{eqnarray*}
This shows that
\begin{eqnarray*}
[D_m, Y(u, z)] = \sum_{i=1}^m Y(D_i(u), z)D_{m-i},
\end{eqnarray*}
and then (c) follows from (b).\ This completes the proof of the Theorem.
\end{proof}

\medskip
\begin{dfn}\label{dfntc} We say that $Y(u, z)$ is \emph{translation covariant with respect to $\underline{D}$}
if property (c) of Theorem \ref{lend} holds for \emph{all} $m{\geq}0$.
\end{dfn}

\section{Characterizations of vertex rings}\label{SCharVR}
In Sections \ref{Sbasic} and \ref{SHS} we have shown that the vertex operators in a vertex ring 
are mutually local (Definition \ref{dfnmlocal}), creative (Definition \ref{dfncreate}), and translation-covariant
(Definition \ref{dfntc}).\
In this Section we show that vertex rings can be \emph{characterized} by these properties.\ This amounts to an extension of
the  Goddard axioms \cite{G} for vertex algebras to the general setting of vertex rings.\ To carry this through we need to develop machinery 
to facilitate calculations with quantum fields on an arbitrary abelian group.

\subsection{Fields on an abelian group}\label{SSfields}
\begin{dfn}\label{dfnFV} Let $V$ be an additive abelian group.\ We set
\begin{eqnarray*}
\mathcal{F}(V){:=}\left\{a(z){:=}\sum_{n\in\ZZ} a(n)z^{-n-1}{\in} End(V)[[z,  z^{-1}]]\ {\mid}\ a(n)b{=}0 \ \mbox{for}\
n\gg0, \ \mbox{all}\ b\in V\right\}.
\end{eqnarray*}
\end{dfn}
In this definition, it is understood that the integer $t$ such that $a(n)b{=}0$ for
$n\geq t$ depends on the states $a$ and $b$.\  Clearly, $\mathcal{F}(V)$ is an additive abelian group.\ We say that 
$a(z){\in}\mathcal{F}(V)$ is a \emph{field on $V$}, and call $\mathcal{F}(V)$  the \emph{space of fields on $V$}.

\medskip
Definition \ref{dfnFV}  is, of course, motivated by
the corresponding axiom (\ref{JI})(a) for vertex operators in a vertex ring.\ Indeed, if $V$ is a vertex ring, the state-field correspondence defines a morphism of abelian groups $Y{:}V \rightarrow \mathcal{F}(V)$.

\medskip
We now carry over to fields in $\mathcal{F}(V)$ 
the main properties  that we previously considered for vertex operators in a vertex ring.\ To be clear,
we repeat the relevant definitions in this more general setting.

\medskip
\begin{dfn}\label{field props} Let $V$ be an additive abelian group.\ Let $a(z){=}\sum_n a(n)z^{-n-1}$ and
$b(z){\in}\mathcal{F}(V)$ be fields on $V$,
$v_0\in V$ be a fixed state, and  $\underline{D}:=(Id_V, D_1, ...)$ a sequence
of endomorphisms of $V$.\\
(a) $a(z)$ is \emph{creative with respect to $v_0$ and creates the state $a{\in}V$},
if 
\begin{eqnarray*}
a(n)v_0=0\ (n\geq 0),\ a(-1)v_0=a,\ \ \mbox{i.e.,}\ a(z)v_0 \in a+ zV[[z]].
\end{eqnarray*}
We say that $a(z)$ is merely \emph{creative} (with respect to $v_0$)  if $a(n)v_0=0\ (n\geq 0)$, in which case the state $a(-1)v_0$
that is created is unspecified.\\
(b) $a(z)$ is \emph{translation covariant with respect to $\underline{D}$} if, for all $m\geq 0$, we have 
\begin{eqnarray*}
[D_m, a(z)]=\sum_{i=1}^m \delta_z^{(i)}a(z)D_{m-i}.
\end{eqnarray*}
(c) $a(z)$ and $b(z)$ are \emph{mutually local (of order $t$)} if there is $t\geq 0$ such that
\begin{eqnarray*}
(z-w)^t[a(z), b(w)]=0.
\end{eqnarray*}
We write this as $a(z){\sim_t}b(z)$, or simply $a(z){\sim}b(z)$ if we do not wish to emphasize $t$.
\end{dfn}

\medskip
In (b), the operator $\delta_z^{(i)}$ on fields is defined as in (\ref{dact1}).\ As in Section \ref{SScal}, the mutual locality of $a(z)$ and $b(z)$ is equivalent to
the analog of the locality formula (\ref{locform}) for all integers $r, s$.\

\medskip The thrust of our earlier arguments is that if $V$ is a vertex ring
then the set of vertex operators $\{Y(u, z){\mid}u{\in} V\}$ is a set of mutually local fields on $V$ that are creative with respect to the vacuum
vector $\mathbf{1}$ and translation covariant with respect to the canonical Hasse-Schmidt derivation of $V$.

\subsection{Statement of the existence Theorem}\label{SSexist}
In this Subsection we state the main existence Theorem and make a start on its proof.
\begin{thm}\label{thmexist1} Let $(V, Y, v_0, \underline{D})$ consist of 
an additive abelian group $V$, a state $v_0{\in}V$,\ a sequence of endomorphisms
$\underline{D}:=(Id_V, D_1, ...)$ in $End(V)$ satisfying $D_m(v_0){=}0$ for $m\geq 1$, and a morphism of abelian groups
\begin{eqnarray*}
Y{:} V\rightarrow \mathcal{F}(V),\ u\mapsto Y(u, z):=\sum_n u(n)z^{-n-1}.
\end{eqnarray*}
Suppose that the following assumptions hold for all states $u, v\in V$:
\begin{eqnarray*}
&&\ \ \ \ \ \ \ \ \  Y(u, z)\sim Y(v, z), \\
&&\ \ \ \ \ \  Y(u, z)v_0 \in u+zV[[z]], \\
 &&{[}D_m, Y(u, z){]}\ {=}\sum_{i=1}^m \delta_z^{(i)}Y(u, z)D_{m-i}\ \ (m\geq 0).
\end{eqnarray*}
(In short, $\{Y(u, z){\mid}u{\in} V\}$ is a set of mutually local, creative and translation-covariant fields on $V$.)\ Then $V$ is a vertex ring with state-field correspondence $Y$, vacuum vector $v_0$, and
canonical HS derivation $\underline{D}$.
\end{thm}

\medskip
Given the translation-covariance assumption as in the statement of the Theorem,
the creativity assumption $Y(u, z)v_0\in u+ zV[[z]]$ is equivalent to the stronger assertion 
\begin{eqnarray}\label{morecreate}
Y(u, z)v_0=\sum_{m\geq 0} D_m(u)z^m.
\end{eqnarray}
Indeed, translation-covariance implies that
\begin{eqnarray*}
D_mY(u, z)v_0=\delta_z^{(m)}Y(u, z)v_0,
\end{eqnarray*}
so that $D_m(u)=D_mu(-1)v_0=u(-m-1)v_0$.\ Then we deduce that 
$Y(u, z)v_0=\sum_{m\geq 0}u(-m-1)v_0z^m=
\sum_{m\geq 0} D_m(u)z^m$, as asserted.\ In particular, in the context of Theorem \ref{thmexist1}, once it is known that $V$ is a vertex ring the statement that
$\underline{D}$ is the canonical HS derivation of $V$ follows automatically.

\medskip
Let $u, v{\in} V$.\  Since $Y(u, z)v{=}\sum_n u(n)vz^{-n-1}{\in} V[[z, z^{-1}]]$,  we have bilinear products
$u(n)v$ for all integers $n$, and because $Y(u, z){\in} \mathcal{F}(V)$ then $u(n)v{=}0$
for $n\gg 0$.\ Furthermore, the creativity assumption  means that
$u(n)v_0=0$ for $n\geq 0$ and $u(-1)v_0=u$.\ Thus  (\ref{JI})(a), (b) 
hold, and in order to prove that $V$ is a vertex ring and thereby complete the proof of Therem 
\ref{thmexist1}, it only remains to establish the Jacobi identity.

\medskip
As a first step we have the following result.
\begin{lem}\label{lemmaasslocJI}  Suppose that $V$ is an additive abelian group, with
$a(z){:=}\sum_n a(n)z^{-n-1}$, $b(z){:=}\sum_n b(n)z^{-n-1}$ a pair of fields on $V$.\
Then the modes of $a(z)$ and $b(z)$ satisfy  the associativity  formula
(\ref{assocform}) and the locality formula (\ref{locform}) for all integers $r, s, t$ if, and only if, they satisfy the Jacobi identity (\ref{JI})(c) for all integers $r, s, t$.
\end{lem}
\begin{proof} In Section \ref{SScal} we derived  associativity and locality (cf.\ Lemma \ref{lemlocform}) of vertex operators in a vertex ring
as a purely formal consequence of (\ref{JI})(c), and the proof in the more general set-up we are now in is exactly the same.

\medskip
It remains to show that, conversely, (\ref{JI})(c) is a consequence of the conjunction of associativity and locality of fields in $\mathcal{F}(V)$.\
In fact, standard proofs of this assertion for vertex algebras defined over $\CC$ (e.g.\ \cite{MN}, Proposition 4.4.3) remain valid in the present setting.\
We sketch the details following the proof of Matsuo-Nagatomo (loc.\ cit).

\medskip
For any $r, s, t{\in}\ZZ$ we introduce the notation
\begin{eqnarray*}
&&A(r, s, t) {=} \sum_{i\geq 0} {r\choose i}(a(t+i)b)(r+s-i),\\
&&B(r, s, t){=} \sum_{i\geq 0}(-1)^i {t\choose i} a(r+t-i)b(s+i),\\
&&C(r, s, t) {=} \sum_{i\geq 0}(-1)^{t+i} {t\choose i} b(s+t-i)a(r+i).
\end{eqnarray*}

In these terms,  the Jacobi identity (\ref{JI})(c) for the fields $a(z), b(z)$ just says that for all $r, s, t$ we have
\begin{eqnarray}\label{ABCform}
A(r, s, t)=B(r, s, t)-C(r, s, t).
\end{eqnarray}
On the other hand, as we discussed in Subsection \ref{SScal}, the associativity formula (\ref{assocform}) is nothing but
the case $r=0$ of (\ref{ABCform}), while locality in the form of (\ref{locform}) is just (\ref{ABCform}) for $t\gg 0$.\ So we have to deduce
the general case of (\ref{ABCform}) on the basis of these two special cases.

\medskip
We can do this by first using (\ref{bi2}) in the Appendix to observe that
\begin{eqnarray}\label{Aform}
A(r+1, s, t)=A(r, s+1, t)+A(r, s, t+1).
\end{eqnarray}
Furthermore, exactly the same formula holds if we replace $A$ by $B$ or $C$.

\medskip
 Because (\ref{ABCform}) holds for 
$r=0$ (and \emph{any} $s, t$), an induction using (\ref{Aform}) shows that it holds for all $r\geq 0$.\ Since it also holds
for all big enough $t$ independently of $r, s$, if it is false in general then there is a pair $(r, t)$ for which it is false
and for which $r+t$ is \emph{maximal}.\ But we have
\begin{eqnarray*}
&&A(r, s, t) = A(r+1, s-1, t)-A(r, s-1, t+1)\\
&=&B(r+1, s-1, t)-C(r+1, s-1, t)-B(r, s-1, t+1)+C(r, s-1, t+1)\\
&=&B(r, s, t)-C(r, s, t).
\end{eqnarray*}
So in fact (\ref{ABCform}) holds for all $r, s, t$, and the proof of the Lemma is complete.
\end{proof}

\subsection{Residue products}\label{SSrp}
Because locality of fields is one of the hypotheses of Theorem \ref{thmexist1}, in order to complete
the proof of the Theorem we are reduced (thanks to Lemma \ref{lemmaasslocJI}) to establishing the associativity formula (\ref{assocform}).\ A good way to approach this is through the use of \emph{residue products} in 
$\mathcal{F}(V)$.\ 

\begin{dfn}\label{dfnresprod}
Let $V$ be an additive abelian group with
$a(z){:=}\sum_n a(n)z^{-n-1}$ and $b(z){:=}\sum_n b(n)z^{-n-1}$ a pair of fields in $\mathcal{F}(V)$.\
Let $m$ be any integer.\ The \emph{$m^{th}$ residue product} of $a(z)$ and $b(z)$ is the field
in $\mathcal{F}(V)$, denoted by $a(z)_m b(z)$, whose $n^{th}$ mode is given by the following formula:
\begin{eqnarray}\label{resprod}
&&\ \ \ \ \ \ \ \ \ \ \ \ \ \ \ \ \ (a(z)_{m}b(z))_n{:=}  \notag\\
&&\sum_{i\geq 0}(-1)^i{m\choose i}\left\{a(m{-}i)b(n+i){-}(-1)^mb(m+n{-}i)a(i)\right\}.
\end{eqnarray}
\end{dfn}

It is easy to see that because $a(z)$ and $b(z)$ are fields on $V$, then for any state $u{\in}V$
we have $(a(z)_m b(z))_nu{=}0$ for all large enough $n$.\ Hence,  $a(z)_m b(z)$ 
\emph{is} a field on $V$.\ Thus for any integer $m$, $\mathcal{F}(V)$ equipped with its $m^{th}$ residue product is a nonassociative ring.

\medskip
Motivation for introducing this field stems from the nature of the associativity formula
(\ref{assocform}).\ Indeed, for a vertex ring we can restate the associativity formula in the following
compact and highly suggestive form{:}
\begin{eqnarray}\label{vringresprod}
Y(u(t)v, z) {=}Y(u, z)_t Y(v, z).
\end{eqnarray}

\medskip
In proving Theorem \ref{thmexist1}, we of course do not know that $V$ is a vertex ring.\
Nevertheless our goal is to establish  (\ref{vringresprod}) for the fields  $Y(u, z)$ defined in Theorem \ref{thmexist1}, 
this being equivalent to the desired associativity.\
First we need to develop some general facts about residue products of fields.

\begin{lem}\label{lemtcresprod} 
Suppose that $a(z), b(z)$ are creative with respect to $v_0$ and that $b(z)$ creates $v$.\ Then
$a(z)_{m}b(z)$ is creative with respect to $v_0$ and creates $a(m)v$
\end{lem}
\begin{proof} Let $n{\geq}-1$. Since $a(n)v_0{=}b(n)v_0{=}0\ (n\geq 0)$ and $b(-1)v_0{=}v$,
we have
\begin{eqnarray*}
&&(a(z)_{m}b(z))_nv_0 \\
&=& \sum_{i\geq 0}(-1)^i{m\choose i}\left\{a(m-i)b(n+i)-((-1)^mb(m+n-i)a(i)\right\}v_0\\
&=&\delta_{n, -1}a(m)v.
\end{eqnarray*}
This completes the proof of the Lemma.
\end{proof}

\begin{lem}\label{lemDong} Suppose that $a(z), b(z), c(z){\in}\mathcal{F}(V)$ are pairwise mutually local fields.\
Then $a(z)_mb(z)$ and $c(z)$ are also mutually local fields for all integers $m$.
\end{lem}
\begin{proof} Standard proofs of this result for vertex algebras over $\mathbf{C}$ (e.g., \cite{MN},
Proposition 2.1.5) also hold for vertex rings with the proof unchanged.
\end{proof}

\subsection{The relation between residue products and translation-covariance}\label{SSrptc}
The main result of this Subsection is
\begin{thm}\label{thmtcfields} Let $V$ be an additive abelian group with a sequence
of endomorphisms $\underline{D}=(Id_V, D_1, ...)$ in $End(V)$.\ 
Suppose that $a(z)$ and $b(z)$ are fields on $V$ that are translation-covariant with respect
to $\underline{D}$.\ Then $a(z)_m b(z)$ is also translation-covariant with respect to $\underline{D}$
for all integers $m$.
\end{thm}

In order to establish this result we first prove a result of independent interest.
\begin{thm}\label{thmdeltatc} Let $V$ be an additive abelian group.\ Then $(Id_V, \delta_z, \delta_z^{(2)}, ...)$ is an iterative HS derivation of the 
nonassociative ring consisting of $\mathcal{F}(V)$ equipped with
its $m^{th}$ residue product.
\end{thm}
\begin{proof} The iterative property (which we do not use) is straightforward to prove, and we skip the details.\ As for the HS property,
let $a(z){=}\sum_n a(n)z^{-n-1}$ and $b(z){=}\sum_n b(n)z^{-n-1}$ lie in $\mathcal{F}(V)$.\ By (\ref{dact1}) we have
\begin{eqnarray*}
\delta_z^{(i)}a(z){=}(-1)^i \sum_n A_i(n)z^{-n-1}, \ \mbox{with}\ A_i(n){=}{n\choose i}a(n-i).
\end{eqnarray*}
With analogous notation for $\delta_z^{(j)}b(z)$, we see that if $i+j=\ell\geq 0$ then
\begin{eqnarray*}
&&((\delta_z^{(i)}a(z))_m(\delta_z^{(j)}b(z)))_n \\
&=& (-1)^{\ell}
\sum_{t\geq 0}(-1)^t{m\choose t}\left\{A_i(m-t)B_j(n+t)-(-1)^mB_j(m+n-t)A_i(t)\right\},
\end{eqnarray*}
and similarly
\begin{eqnarray*}
&&\delta_z^{(\ell)}(a(z)_mb(z)) = (-1)^{\ell}\sum_n {n\choose\ell}(a(z)_{(m)}b(z))_{n-\ell}z^{-n-1}\\
=&&(-1)^{\ell}\sum_n {n\choose\ell}\sum_{t\geq 0}(-1)^t{m\choose t}
\left\{a(m-t)b(n-\ell+t)-(-1)^mb(m+n-\ell-t)a(t)\right\}z^{-n-1}.
\end{eqnarray*}
So it suffices to show for all integers $m, n$ that
\begin{eqnarray}\label{tcsum}
&& {n\choose\ell}\sum_{t\geq 0}(-1)^t{m\choose t}
\left\{a(m-t)b(n-\ell+t)-(-1)^mb(m+n-\ell-t)a(t)\right\}= \notag\\
&&\ \ \ \ \ \sum_{i+j=\ell}\sum_{t\geq 0}(-1)^t{m\choose t}\left\{A_i(m-t)B_j(n+t)-(-1)^mB_j(m+n-t)A_i(t)\right\}.
\end{eqnarray}
Note that
\begin{eqnarray*}
A_i(m-t)B_j(n+t)&=&{m-t\choose i}{n+t\choose j}a(m-t-i)b(n+t-j),\\
B_j(m+n-t)A_i(t)&=& {m+n-t\choose j}{t\choose i} b(m+n-t-j)a(t-i).
\end{eqnarray*}
Thus (\ref{tcsum}) will follow from
\begin{eqnarray}\label{tc1}
&& {n\choose\ell}\sum_{t\geq 0}(-1)^t{m\choose t} a(m-t)b(n-\ell+t)\\
&=&\sum_{i=0}^\ell\sum_{t\geq 0}(-1)^t{m\choose t}{m-t\choose i}{n+t\choose \ell-i}a(m-t-i)b(n+t-\ell+i) \notag
\end{eqnarray}
and
\begin{eqnarray}\label{tc2}
&& {n\choose\ell}\sum_{t\geq 0}(-1)^t{m\choose t}b(m+n-\ell-t)a(t)\\
&=&\sum_{i=0}^{\ell}\sum_{t\geq 0}(-1)^t{m\choose t}{m+n-t\choose \ell-i}{t\choose i} b(m+n-t-\ell+i)a(t-i). \notag
\end{eqnarray}

To prove (\ref{tc2}), notice that for   $p\geq0$ the coefficient of $b(m+n-\ell-p)a(p)$ on the right-hand-side
 is equal to
\begin{eqnarray*}
&&\sum_{i=0}^{\ell}(-1)^{p+i}{m\choose p+i}{m+n-p-i\choose \ell-i}{p+i\choose i} \\
&=&(-1)^{p}{m\choose p}\sum_{i=0}^{\ell}(-1)^{i}{m-p\choose i}{m+n-p-i\choose \ell-i}\\
&=&(-1)^p{m\choose p}{n\choose \ell}
\end{eqnarray*}
where the last equality follows from (\ref{bi4}).\ This proves (\ref{tc2}), and we can establish (\ref{tc1})  in exactly the same way. 
The proof of  Theorem \ref{thmdeltatc}  is complete.
\end{proof}

We turn to the proof of Theorem \ref{thmtcfields}.\ Choose $\ell{\geq}0$
and use the operator identity
$[D_{\ell}, AB]=[D_{\ell}, A]B + A[D_{\ell}, B]$ to obtain
\begin{eqnarray*}
&&[D_{\ell}, a(z)_mb(z)] = \sum_n [D_{\ell}, (a(z)_mb(z))_n ]z^{-n-1}\\
=&&\sum_n\sum_{i\geq 0}(-1)^i{m\choose i}[D_{\ell}, \left\{a(m-i)b(n+i)-((-1)^mb(m+n-i)a(i)\right\}]z^{-n-1}\\
=&&[D_{\ell}, a(z)]_mb(z)+a(z)_m[D_{\ell}, b(z)]\\
=&&\sum_{i=1}^{\ell}\left\{ (\delta_z^{(i)}a(z)D_{\ell-i})_mb(z)+a(z)_m(\delta_z^{(i)}b(z)D_{\ell-i} )\right\}.
\end{eqnarray*}

On the other hand, by Theorem \ref{thmdeltatc} we have
\begin{eqnarray*}
\sum_{i=1}^\ell \delta^{(i)}_z (a(z)_mb(z))D_{\ell-i} &=& \sum_{i=1}^\ell 
\sum_{j=0}^i (\delta_z^{(j)}a(z))_m(\delta_z^{(i-j)}b(z))D_{\ell-i}\\
&=&\sum_{r=0}^{\ell-1} \sum_{p+q+r=\ell} (\delta_z^{(p)}a(z))_m(\delta_z^{(q)}b(z))D_{r}.
\end{eqnarray*}
Thus we must establish the following identity:
\begin{eqnarray}\label{deltaDid}
&&\sum_{i=1}^{\ell}\left\{ (\delta_z^{(i)}a(z)D_{\ell-i})_mb(z)+a(z)_m(\delta_z^{(i)}b(z)D_{\ell-i} )\right\}\\
&=&\sum_{r=0}^{\ell-1} \sum_{p+q+r=\ell} (\delta_z^{(p)}a(z))_m(\delta_z^{(q)}b(z))D_{r}. \notag
\end{eqnarray}

\medskip
For various fields $c(z)=\sum_n c(n)z^{-n-1}$ we have to consider expressions of the form
\begin{eqnarray*}
&&((c(z)D_t)_{(m)}b(z))_n\\
=&& \sum_{r\geq 0}(-1)^r{m\choose r}\left\{c(m-r)D_tb(n+r)-((-1)^mb(m+n-r)c(r)D_t\right\}\\
=&& \sum_{r\geq 0}(-1)^r{m\choose r}\\
&&\left\{c(m-r)[D_t, b(n+r)]+c(m-r)b(n+r)D_t-((-1)^mb(m+n-r)c(ir)D_t\right\}\\
=&&((c(z))_{(m)}b(z))_nD_t +\sum_{r\geq 0}(-1)^r{m\choose r}c(m-r)[D_t, b(n+r)],
\end{eqnarray*}
and similarly
\begin{eqnarray*}
&&a(z)_{(m)}(d(z)D_t)_n \\
&=& 
\sum_{r\geq 0}(-1)^r{m\choose r}\left\{ a(m-r)d(n+r)D_t -(-1)^md(m+n-r)D_ta(r)  \right\}\\
&=& 
\sum_{r\geq 0}(-1)^r{m\choose r}\\
&&\left\{ a(m-r)d(n+r)D_t -(-1)^md(m+n-r)[D_t, a(r)] 
-(-1)^md(m+n-r) a(r)D_t\right\}\\
&=&(a(z)_{(m)}d(z))_nD_t +\sum_{r\geq 0}(-1)^r{m\choose r}\left\{-(-1)^md(m+n-r)[D_t, a(r)] \right\}.
\end{eqnarray*}
As a result, we obtain
\begin{eqnarray*}
&&\sum_{i=1}^{\ell}\left\{ (\delta_z^{(i)}a(z)D_{\ell-i})_{(m)}b(z)+a(z)_{(m)}(\delta_z^{(i)}b(z)D_{\ell-i}) \right\}\\
=&&\sum_{i=1}^\ell ((\delta_z^{(i)}a(z))_{(m)}b(z))D_{\ell-i}+(a(z)_{(m)}(\delta^{(i)}_zb(z)))D_{\ell-i}+\\
&& \sum_{i=1}^\ell \sum_{r\geq 0}\sum_n (-1)^r{m\choose r}\left\{c_i(m-r)[D_{\ell-i}, b(n+r)]-(-1)^m
d_i(m+n-r)[D_{\ell-i}, a(r)]\right\}z^{-n-1}\\
=&& \sum_{i=1}^\ell \sum_{r\geq 0}\sum_n (-1)^r{m\choose r} \sum_{j=1}^{\ell-i}(-1)^j    \\
&&\left\{c_i(m-r){n+r\choose j} b(n+r-j)    -(-1)^m
d_i(m+n-r) {r\choose j} a(r-j) \right\}D_{\ell-i-j}z^{-n-1},
\end{eqnarray*}
where we have set $c_i(z) = \delta_z^{(i)}a(z)$ and $d_i(z)= \delta_z^{(i)}b(z)$.

\medskip
Comparing this with (\ref{deltaDid}), we are reduced to proving the following equality:
\begin{eqnarray*}
&&\ \ \   \ \ \ \ \ \  \ \ \ \ \ \ \ \ \  \sum_{i=1}^\ell \sum_{r\geq 0}\sum_n (-1)^r{m\choose r} \sum_{j=1}^{\ell-i}(-1)^j   \\
&&\left\{c_i(m-r){n+r\choose j} b(n+r-j)   -(-1)^m
d_i(m+n-r) {r\choose j} a(r-j) \right\}D_{\ell-i-j}z^{-n-1} \\
&=&\sum_{r=0}^{\ell-1} \sum_{p+q+r=\ell} (\delta_z^{(p)}a(z))_m(\delta_z^{(q)}b(z))D_{r}\\
&=&\sum_n \sum_{r=0}^{\ell-1} \sum_{i+j+r=\ell} \sum_{s\geq 0} (-1)^s{m\choose s}\left\{c_i(m-s)d_j(n+s)-(-1)^md_j(m+n-s)c_i(s)\right\}z^{-n-1}D_r,
\end{eqnarray*}
that is
\begin{eqnarray*}
&&\ \ \   \ \ \ \ \ \  \ \ \ \ \ \ \ \ \  \sum_{i=1}^\ell \sum_{s\geq 0} (-1)^s{m\choose s} \sum_{j=1}^{\ell-i}(-1)^j   \\
&&\left\{c_i(m-s){n+s\choose j} b(n+s-j)   -(-1)^m
d_i(m+n-s) {s\choose j} a(s-j) \right\}D_{\ell-i-j} \\
&=&\sum_{r=0}^{\ell-1} \sum_{i+j+r=\ell} \sum_{s\geq 0} (-1)^s{m\choose s}\left\{c_i(m-s)d_j(n+s)-(-1)^md_j(m+n-s)c_i(s)\right\}D_r.
\end{eqnarray*}
Now
\begin{eqnarray*}
&&\ \ \   \ \ \ \ \ \  \ \ \ \ \ \ \ \ \  \sum_{i=1}^\ell \sum_{s\geq 0} (-1)^s{m\choose s} \sum_{j=1}^{\ell-i}(-1)^j   \\
&&\left\{c_i(m-s){n+s\choose j} b(n+s-j)   -(-1)^m
d_i(m+n-s) {s\choose j} a(s-j) \right\}D_{\ell-i-j} \\
&=&\sum_{u=0}^{\ell-1}\sum_{i+j+u=\ell}  \sum_{s\geq 0} (-1)^s{m\choose s} (-1)^j   \\
&&\left\{c_i(m-s){n+s\choose j} b(n+s-j)   -(-1)^m d_i(m+n-s) {s\choose j} a(s-j) \right\}D_{u}, 
\end{eqnarray*}
so we need for fixed  $1\leq v\leq \ell$ that
\begin{eqnarray*}
&&\sum_{i+j=v}  \sum_{s\geq 0} (-1)^s{m\choose s} (-1)^j   \\
&&\left\{c_i(m-s){n+s\choose j} b(n+s-j)   -(-1)^m d_i(m+n-s) {s\choose j} a(s-j) \right\}\\
&=& \sum_{i+j=v} \sum_{s\geq 0} (-1)^s{m\choose s}\left\{c_i(m-s)d_j(n+s)-(-1)^md_j(m+n-s)c_i(s)\right\}.
\end{eqnarray*}
But this follows directly from the definition of the fields $c_i(z), d_j(z)$, and the proof of
Theorem \ref{thmtcfields} is complete. $\hfill\Box$

\subsection{Completion of the proof of Theorem \ref{thmexist1}}\label{SScompletion}
We have already seen in Subsection \ref{SSrp} that only the associativity formula
(\ref{assocform}) for the fields  $Y(u, z)\ (u{\in}V)$ remains to be proved, and that furthermore this is equivalent to proving
the identity (\ref{vringresprod}).\ We have now assembled all of the pieces that
allow us to carry this out.\ We first record a Lemma that we will need again later.
\begin{lem}\label{lemmadz} Suppose that $d(z){=}\sum_n d(n)z^{-n-1}{\in}\mathcal{F}(V)$ is translation-covariant, mutually local with all
fields $Y(u, z)\ (u\in V)$ and creative with respect to $v_0$.\ Then $d(z){=}0$ if, and only if, $d(z)$ creates $0$, i.e., $d(-1)v_0{=}0$.
\end{lem}
\begin{proof} We have to prove that $d(z){=}0$ on the basis of the assumption that $d(-1)v_0=0$.\ To see this,
we first show that $d(z)v_0=0$.\
 Let $m\geq 1$.\ Then $D_m(v_0){=}0$, so that
\begin{eqnarray*}
[D_m, d(z)]v_0 {=} D_md(z)v_0{=} \sum_{n<0} D_md(n)v_0z^{-n-1}.
\end{eqnarray*}
On the other, by translation-covariance we also have
\begin{eqnarray*}
[D_m, d(z)]v_0{=} \sum_{i=1}^m \delta_z^{(i)}d(z)D_{m-i}v_0{=}\delta_z^{(m)}d(z)v_0{=}\sum_{n<0} {-n-1\choose m}d(n)v_0z^{-n-m-1}.
\end{eqnarray*}
This shows that for $m{\geq}1, n{<}0$ we have
\begin{eqnarray*}
D_md(n)v_0 {=}  {m-n-1\choose m}d(n-m)v_0.
\end{eqnarray*}
Because $d(-1)v_0{=}0$ by hypothesis, we obtain for all $m{\geq}1$ that
\begin{eqnarray*}
0{=}D_md(-1)v_0{=} d(-1-m)v_0,
\end{eqnarray*}
and therefore $d(z)v_0{=}0$, as claimed.

\medskip
Now by assumption, $d(z)$ is mutually local with every field $Y(u, z)\
(u{\in} V)$.\ Suppose that $(z-w)^N[d(z), Y(u, z)]{=}0$ for some $N\geq 0$.\  Then
\begin{eqnarray*}
&&z^Nd(z)u {=}Res_ww^{-1}(z-w)^Nd(z)Y(u, w)v_0 \\
&&\ \ \ \ \ \ \ \ \ \  \ {=} Res_ww^{-1}(z-w)^NY(u, w)d(z)v_0 = 0.
\end{eqnarray*}
This shows that $d(z)u{=}0$ for all $u{\in} V$, so $d(z){=}0$, and the proof
of the Lemma is complete.
\end{proof}

\medskip
To complete the proof of Theorem \ref{thmexist1},  choose $u, v{\in}V, t{\in}\mathbf{Z}$ and set 
\begin{eqnarray*}
d(z){=}Y(u, z)_{t}Y(v, z)-Y(u(t)v, z).
\end{eqnarray*}
\ We have to show that $d(z){=}0$.\ Indeed, $u(t)v{\in}V$, so that $Y(u(t)v, z)$ is creative, translation-covariant and mutually local with all fields
$Y(a, z)\ (a{\in}V)$ by hypothesis, and
it creates $u(t)v$.\ On the other hand,
$Y(u, z)_tY(v, z)$ is also creative, mutually local with all fields $Y(a, z)$, and translation-covariant by Lemmas \ref{lemtcresprod}, \ref{lemDong}
and Theorem \ref{thmtcfields} respectively, and it also creates $u(t)v$ (loc.\ cit).\ Therefore,
$d(z)$ satisfies the assumptions of the previous Lemma,  and the conclusion
$d(z){=}0$ follows.\ This completes the proof of Theorem \ref{thmexist1}.

\subsection{Generators of a vertex ring and a refinement of Theorem \ref{thmexist1}}
In this Subsection we deduce a refinement of Theorem \ref{thmexist1} which only involves 
\emph{generators} of $V$.

\begin{dfn}\label{defVgens} Let $U\subseteq V$ be a subset
of a vertex ring $V$.\ We say that $U$ \emph{generates} $V$ if $V$ is generated (as an additive abelian group) by the states
\begin{eqnarray*}
\left\{u_1(i_1)...u_k(i_k)\mathbf{1} \mid u_j\in U, i_j\in\mathbf{Z},\ j=1, ..., k \right\}.
\end{eqnarray*}
If this holds, we write $V=\langle U \rangle$.
\end{dfn}

We shall prove
\begin{thm}\label{thmexist2} Let $(V, v_0, \underline{D})$ consist of 
an additive abelian group $V$, a state $v_0{\in} V$, and a sequence of endomorphisms
$\underline{D}{:=}(Id_V, D_1, ...)$ in $End(V)$ satisfying $D_m(v_0){=}0\ (m{\geq} 1)$.\ Suppose that
$U{\subseteq} V$,  that to each $u{\in} U$ is attached a field $u(z){=}\sum_n u(n)z^{-n-1}{\in}
\mathcal{F}(V)$,  and that  the following assumptions hold for all $u, v\in U$:
\begin{eqnarray*}
&&u(z) {\sim} v(z), \\
&&u(z)v_0 {\in} u{+}zV[[z]],\\
&&[D_m, u(z)]{=}\sum_{i=1}^m \delta_z^{(i)}u(z)D_{m-i}\ \ (m\geq 1),\\
&&\mbox{$V$ is generated (as an additive abelian group) by}\\
&&\ \ \ \ \ \ \ \ \ \ \ \ \ \ \  \mbox{$\left\{u_1(m_1)...u_k(m_k)v_0 \mid u_j\in U, m_j\in\mathbf{Z}, j=1, ..., k  \right\}$.}
\end{eqnarray*}
Then $V{=}(V, Y, v_0, \underline{D})$ is a vertex ring generated by $U$ and
with state-field correspondence $Y$ satisfying $Y(u, z)=u(z)\ (u\in U)$, vacuum vector $v_0$, and
canonical HS derivation $\underline{D}$.
\end{thm}
\begin{proof} We associate to each state $a:=u_1(m_1)...u_k(m_k)v_0\ (u_j\in U)$ the field
\begin{eqnarray*}
a(z):=u_1(z)_{m_1}(u_2(z)_{m_2}...(u_k(z)_{m_k}Id_V)...),
\end{eqnarray*}
and to each such sum of states the corresponding sum of fields.\ Thanks to
the hypotheses of the Theorem together with Lemmas \ref{lemtcresprod}, \ref{lemDong}
and Theorem \ref{thmtcfields}, we create in this way a set of mutually local,
 translation-covariant, creative fields associated to the states of $V$.\ 
  
  \medskip
  We aver that this process is \emph{well-defined}, i.e., each state is associated to a \emph{unique} field on $V$.\ If this is so, the linear extension of the association
$a\mapsto a(z)$ is the state-field correspondence we are after, and the present Theorem follows
directly from Theorem \ref{thmexist1}.

\medskip
As for well-definedness, it suffices to show that if we have a relation of the form
\begin{eqnarray*}
\sum_{j=1}^t u_{1_j}(i_{1_j})...u_{k_j}(i_{k_j})v_0=0\ \ (u_{m_j}\in U),
\end{eqnarray*}
then the field associated to the state on the left-hand-side by the process described above
also vanishes.\ Let the field in question be denoted by $d(z)$.\ Now $d(z)$ is mutually local with all other fields associated to the states of $V$
and translation-covariant by Lemma \ref{lemDong} and Theorem \ref{thmtcfields} respectively, and it is creative and creates $0$ by Lemma \ref{lemtcresprod}.\ Then $d(z)=0$ holds by Lemma \ref{lemmadz}.\ This completes the proof of
Theorem \ref{thmexist2}.
\end{proof}

\begin{rmk}\label{remgenresprod} The proof shows that if $u=u_1(m_1)\hdots u_k(m_k)v_0$ then
\begin{eqnarray*}
Y(u, z)= Y(u_1, z)_{m_1}( \hdots (Y(u_k, z)_{m_k}Id_V)\hdots).
\end{eqnarray*}
\end{rmk}

\subsection{Formal Taylor expansion and an alternate existence Theorem}
In this Subsection we illustrate some additional techniques dealing with formal series, in particular
the formal Taylor expansion, by proving an alternate characterization of vertex rings.\ The result we are after,
the proof of which follows a line of argument due to Haisheng Li \cite{Li}, is the following.
\begin{thm}\label{thmexist3} Let $(V, Y', v_0, \underline{D})$ consist of 
an additive abelian group $V$, a state $v_0{\in}V$, an \emph{iterative} sequence of endomorphisms
$\underline{D}:=(Id_V, D_1, ...)$ in $End(V)$ satisfying $D_m(v_0){=}0$ for $m\geq 1$, and a morphism of abelian groups
\begin{eqnarray*}
Y'{:} V\rightarrow End(V)[[z, z^{-1}]],\ u\mapsto Y'(u, z):=\sum_n u(n)z^{-n-1}.
\end{eqnarray*}
Suppose that the following assumptions hold for all states $u, v\in V$:
\begin{eqnarray*}
&&\ \ \ \ \ \ \ \ \ Y'(u, z)\sim Y'(v, z), \\
&&\ \ \ \ \ \ Y'(u, z)v_0\in u+zV[[z]], \\
&& {[}D_m, Y'(u, z){]} =\sum_{i=1}^m \delta_z^{(i)}Y'(u, z)D_{m-i}\ \ (m\geq 0).
\end{eqnarray*}
Then $V$ is a vertex ring with state-field correspondence $Y'$, vacuum vector $v_0$, and
canonical HS derivation $\underline{D}$.\ (In particular, each $Y'(u, z)\in\mathcal{F}(V)$.)
\end{thm}

\begin{rmk} The statement of Theorem \ref{thmexist3} is very similar to that of Theorem \ref{thmexist1}, but it differs in two crucial ways{:}\ first, the formal series $Y'(u, z)$ are \emph{not} assumed to be fields on $V$,
merely  series of operators on $V$.\ For this reason, we have used $Y'$ to denote such series; the prime
indicates that $Y'(u, z)$ is not necessarily a field.\ (At the end of the day, of course, the conclusion of the Theorem is that
$Y'(u, z)$ \emph{is} a field.)\ The second difference is that $\underline{D}$ is assumed to be \emph{iterative}.\
Note that the other assumptions of locality, creativity and translation-covariance do not require $Y'(u, z)$ to be
a field on $V$ - they make sense for all series.
\end{rmk}

We turn to the proof of Theorem \ref{thmexist3}, which we give in a sequence of Lemmas.\
 We adopt the notation and assumptions of the Theorem until further notice.\ In the following, $y$ and $z$
 are  a pair of formal variables.

\begin{lem}\label{lemaltcreate}
 We have
\begin{eqnarray*}
Y'(u, z)v_0=\sum_{m\geq 0} D_m(u)z^m.
\end{eqnarray*}
\end{lem}
\begin{proof} The proof follows from translation-covariance, just as (\ref{morecreate}) is deduced from the
assumptions of Theorem \ref{thmexist1}.
\end{proof}

\medskip
 The expression $\delta^{(i)}_z=\frac{1}{i!}\partial_z\ (i\geq 0)$ introduced in Subsection \ref{SStc} may be
 applied to series in $V[[z, z^{-1}]]$, as may the \emph{exponential}
$e^{y\partial_z}$.\ However they cannot generally be applied to states in any meaningful way.\ As a substitute
we may use $D_i$ in place of $\delta^{(i)}_z$, when the exponential becomes $\sum_{m\geq 0} D_my^m$.

\begin{lem}\label{lemDconj} We have
\begin{eqnarray*}
\left\{\sum_{m\geq 0} D_my^m\right\}Y'(u, z)\left\{\sum_{m\geq 0} (-y)^mD_m\right\}  = e^{y\partial_z}Y'(u, z). 
\end{eqnarray*}
\end{lem}
\begin{proof} We have to prove that if $n{\geq}0$ then
\begin{eqnarray*}
\sum_{m=0}^n (-1)^{n-m}D_mY'(u, z)D_{n-m} =\delta^{(n)}_zY'(u, z).
\end{eqnarray*}
But the lhs of this equation is equal to
\begin{eqnarray*}
&&\sum_{m=0}^n (-1)^{n-m}\left([D_m, Y'(u, z)]+Y'(u, z)D_m\right)D_{n-m}\\
&=&\sum_{m=0}^n (-1)^{n-m}\left(\sum_{i=0}^m \delta^{(i)}_zY'(u, z)D_{m-i}\right)D_{n-m}\\
&=&\sum_{i=0}^m \delta^{(i)}_zY'(u, z)\sum_{m=0}^n (-1)^{n-m}D_{m-i}D_{n-m}\\
&=&\sum_{i=0}^m \delta^{(i)}_zY'(u, z)\sum_{m=0}^n (-1)^{n-m}{n-i\choose n-m}D_{n-i}\\
&=& \delta^{(n)}_zY'(u, z).
\end{eqnarray*}
\end{proof}

\begin{lem}\label{lemFTE} (Formal Taylor expansion.)\ We have
\begin{eqnarray*}
Y'(u, z+y)=\left\{\sum_{m\geq 0} D_my^m\right\}Y'(u, z)\left\{\sum_{m\geq 0} (-y)^mD_m\right\}. 
\end{eqnarray*}
(Convention (\ref{binexp}) for binomial expansions is in effect here.)
\end{lem}
\begin{proof} We have
\begin{eqnarray*}
e^{y\partial_z}Y'(u, z)&=&\sum_{i\geq 0} y^i \delta_z^{(i)}Y'(u, z)=\sum_{i\geq 0}\sum_{n\in\ZZ} \frac{y^i}{i!} n(n-1)...(n-i+1)u(-n-1)z^{n-i}\\
&=&\sum_{i\geq 0}\sum_{n\in\ZZ} {n\choose i} u(-n-1)y^iz^{n-i}=Y'(u, z+y).
\end{eqnarray*}
Now apply Lemma \ref{lemDconj}.
\end{proof}

\begin{lem}\label{lemfdprop} If $Y'(u, z){\sim_t}Y'(v, z)\ (t\geq 0)$ then $u(n)v{=}0$ for $n{\geq}t$.\ In particular,
$Y'(u, z)$ is a field on $V$.
\end{lem}
\begin{proof}
Using Lemma \ref{lemaltcreate}, Corollary \ref{rmkDinvert} (applicable thanks to iterativity of $\underline{D}$),
Lemma \ref{lemFTE} and locality, we have 
\begin{eqnarray}\label{Liarg}
&&(z-y)^tY'(u, y)Y'(v, z)v_0 \notag\\
&=& (z-y)^tY'(v, z)Y'(u, y)v_0\notag\\
&=&(z-y)^tY'(v, z) \left(\sum_{m\geq 0} D_{m}y^m\right)u\\
&=&(z-y)^t\left(\sum_{m\geq 0} D_{m}y^m\right)Y'(v, z-y)u. \notag
\end{eqnarray}

\medskip
In this sequence of equalities the first expression contains \emph{no negative powers of $z$},
whereas in the last expression there are \emph{no negative powers of $y$} thanks to our binomial convention.\
So both sides lie in $V[[y, z]]$, and setting $z{=}0$ in the first expression returns the pure powers of
$y$, equal to
\begin{eqnarray*}
(-y)^tY'(u, y)v.
\end{eqnarray*}
We deduce that
\begin{eqnarray*}
u(n)v{=}0\  \mbox{for}\ n\geq t,
\end{eqnarray*}
which completes the proof of the Lemma.
\end{proof}

Now that we know that $Y'(u, z){\in}\mathcal{F}(V)$ for $u{\in}V$, Theorem \ref{thmexist3} 
follows from Theorem \ref{thmexist1}.\ 
 
 \begin{rmk}It is possible to avoid
 Theorem \ref{thmexist1} altogether and give an independent proof of Theorem \ref{thmexist3} (\cite{Li}).\
 We shall not use this fact, so we skip the details, leaving them to the reader as a (nontrivial) exercise.
 \end{rmk}

\section{Categories of vertex rings}
We consider the category of vertex rings and various subcategories.\ A particularly useful result
(Theorem \ref{thmladjoint}) says that the category of commutative rings is a \emph{coreflective}
subcategory of the category of vertex rings.\ For relevant background in category theory, see \cite{Mac}.

\subsection{The category of vertex rings}
Let $U$ and $V$ be two vertex rings with state-field correspondences $Y_U, Y_V$ respectively.\ We frequently abuse notation by using
$\mathbf{1}$ for the vacuum element in either $U$ or $V$, denoting products in
both $U$ and $V$ by $a(n)b$, and writing $Y$ for either $Y_U$ or $Y_V$, making the distinction between $U$ and $V$ explicit only
when necessary to avoid  confusion.

\begin{dfn}
A \emph{morphism} $f{:}U{\rightarrow}V$ is a morphism of additive abelian groups
that preserves all products in the sense that $f(u(n)v){=}f(u)(n)f(v)\ (u, v{\in} U, n\in\mathbf{Z})$ and also
preserves vacuum elements, i.e.,  $f(\mathbf{1}){=}\mathbf{1}$.\ In terms of vertex operators,
\begin{eqnarray*}
fY_U(u, z)v {=} Y_V(f(u), z)f(v).
\end{eqnarray*}
\end{dfn}

 \medskip
 With this definition, vertex rings and their morphisms define a large category that
 we denote by $\bf{Ver}$.
  
 \begin{dfn}
 A  \emph{left ideal} in $V$ is an
 additive subgroup $I{\subseteq} V$ such that $a(n)u\in I$  for all $a\in V, u\in I, n\in\mathbf{Z}$.\ 
 Similarly, a \emph{right ideal} satisfies $u(n)a\in I$ for all $a\in V, u\in I, n\in\mathbf{Z}$.\
 A \emph{$2$-sided ideal} is one that is both a left and right ideal.

\medskip
$V$ is called \emph{simple} if it has no nontrivial 2-sided ideals.
\end{dfn}

\medskip
 In a vertex algebra over $\mathbf{C}$ the notions of left, right and $2$-sided ideals coincide,
 but that is not quite right for a general vertex ring.\ Recall (cf.\ Lemma \ref{lemdpmodule}) the left action of the ring $\mathbf{Z}\langle x\rangle$ of divided powers.
\begin{lem}\label{lemideal} Let $V$ be a vertex ring.\ The following are equivalent for a subgroup $I{\subseteq} V$.
\begin{eqnarray*}
&&(a)\ \mbox{$I$ is a $2$-sided ideal},\\
&&(b)\ \mbox{$I$ is a left ideal and a left $\mathbf{Z}\langle x\rangle$-submodule}.
\end{eqnarray*}
\end{lem}
\begin{proof} In order for $I$ to be a left $\mathbf{Z}\langle x\rangle$-submodule of $V$, it is necessary and sufficient that 
$D_na {=} a(-n-1)\mathbf{1}{\in} I$ for all $a{\in} I, n{\geq} 0$.\ So clearly
(a)$\Rightarrow$(b).\ On the other hand, suppose that (b) holds.\ Use skew-symmetry
(Lemma \ref{lemmaskewsymm}) to see that if $a{\in} I, u{\in} V, n\in\mathbf{Z}$ then
\begin{eqnarray*}
a(n)u {=} (-1)^{n+1}\sum_{i\geq 0} (-1)^iD_iu(n+i)a {\in} I,
\end{eqnarray*}
whence $I$ is a right as well as a left ideal, and (a) holds.\ The Lemma is proved. 
\end{proof}

\medskip
The $2$-sided ideals of $V$ are the kernels of morphisms in ${\bf Ver}$.\  If $I{\subseteq} V$ is a $2$-sided ideal then the quotient group
$V/I$ inherits the structure of vertex ring in the natural way, and the standard isomorphism
theorems apply.

\begin{ex}\label{exideal} Many elementary constructions in commutative ring theory have vertex ring analogs.\
 We give some examples that will arise later on.\\
(i) \emph{Maximal 2-sided ideals}.\ These are proper 2-sided ideals $I{\subseteq} V$ (proper means not equal to $V$)  that are maximal in the lattice of all 2-sided ideals.\ They exist by dint of a standard application of
Zorn's Lemma in which one considers the 2-sided ideals that do \emph{not} contain $\mathbf{1}$.\ $I$ is
maximal if, and only if, $V/I$ is a simple vertex ring.\\
(ii)\ \emph{The torsion ideal}.\ This is the torsion subgroup $T(V){\subseteq} V$ of the underlying abelian group.\
Because all products in $V$ are bilinear, $T(V)$ is a 2-sided ideal of $V$ whose quotient $V/T(V)$ is
a torsion-free vertex ring.\\
(iii) The \emph{principal 2-sided ideal generated by $v\in V$}, i.e.\ the smallest 2-sided ideal of $V$ containing $v$, consists of all finite sums  $u_1(n_1)\hdots u_k(n_k)v(n)\mathbf{1}$ with $u_1, \hdots, u_k\in V$ and
integers $n_1, \hdots, n_k, n$.\ This is straightforward to check using Lemma \ref{lemideal}.\\
(iv) The ring of rational integers $\mathbf{Z}$ is a vertex ring (see the following Subsection for more on this),
indeed $\mathbf{Z}$ is an \emph{initial object} in the category $\mathbf{Ver}$.\  The characteristic map $\mathbf{Z}\rightarrow V, n\mapsto n\mathbf{1}$ is the unique morphism in $Hom_{\mathbf{V}}(\mathbf{Z}, V)$.\ The
kernel is $p\mathbf{Z}$ for an integer $p\geq 0$ that we naturally call the \emph{characteristic} of $V$, denoted
by $char V$.\ Note that $T(V){=}V$ if, and only if, $char V{>}0$.
\end{ex}

\medskip
The image $f(U)$ of a morphism of vertex rings $f{:} U\rightarrow V$ is a \emph{vertex subring} of
$V$ in the usual sense.\ That is, it contains $\mathbf{1}$ and is \emph{closed} with respect to
all $n^{th}$ products of its elements.

\medskip
Suppose that $U{\subseteq} V$, and let $\langle U \rangle$ be the additive subgroup
of $V$ generated by all states $u_1(m_1)\hdots u_k(m_k)\mathbf{1}$ for all
$u_j\in U, m_j\in\mathbf{Z}, k\geq 0$.\ We assert that $\langle U\rangle$ is a vertex subring of $V$.\
Indeed, the case $k=0$ shows that $\mathbf{1}\in\langle U \rangle$.\ Moreover,
by the associativity formula in the form (\ref{vringresprod})  we have
\begin{eqnarray*}
Y(u_1(m_1)\hdots u_k(m_k)\mathbf{1}, z)= Y(u_1, z)_{m_1}(\hdots Y(u_k, z)_{m_k}Id_V)\hdots )
\end{eqnarray*}
and from the very nature of residue product (cf.\ Definition \ref{dfnresprod}) the modes of this field
are sums of compositions of the modes of the states $u_i$.\ Thus each $n^{th}$
product of any pair of generating states $u_1(m_1)\hdots u_k(m_k)\mathbf{1}$ and 
 $u_1'(m_1')\hdots u_t'(m_t')\mathbf{1}$ lies in $\langle U\rangle$, and the fact that
 $\langle U \rangle$ is a vertex subring of $V$ follows immediately.\ We remark that use
  of the symbol $\langle U\rangle$ here is consistent with  Definition \ref{defVgens} inasmuch as 
  $U$ generates the vertex subring $\langle U\rangle$ of $V$.
 
\subsection{Commutative rings with HS derivation}
\begin{thm}\label{thmcatderring} Suppose that $A$ is a commutative ring with identity $1$ that admits
an iterative Hasse-Schmidt derivation $\underline{D}{=}(Id_A, D_1, \hdots)$.\
For $u, v{\in} A$ define
\begin{eqnarray*}
u(n)v{=} \left\{ \begin{array}{cc} 
 D_{-n-1}(u)v,& n<0\\
 0, &  n\geq 0\end{array} \right.
\end{eqnarray*}
Equipped with these products, $A$ is a vertex ring with vacuum element $1$, canonical
HS derivation $\underline{D}$, and vertex operators
\begin{eqnarray*}
Y(u, z) {=} \sum_{n\geq 0} (D_{n}u)z^{n},
\end{eqnarray*}
 where $D_n(u)$ is the endomorphism of $A$ corresponding to multiplication by $D_n(u)$.
\end{thm}
\begin{proof} We will verify the axioms of associativity, locality, and translation-covariance.\ Then
Theorem \ref{thmexist1} applies, and the present Theorem follows.

\medskip
First we check the associativity formula (\ref{assocform}).\ Because all $n^{th}$ products
 are $0$ for $n\geq 0$, this reduces to showing that for $s, t\leq -1$ we have
\begin{eqnarray*}
(u(t)v)(s) = \sum_{i\geq 0}(-1)^i {t\choose i} \left\{u(t-i)v(s+i)\right\},
\end{eqnarray*}
 i.e.,
\begin{eqnarray*}
D_{-s-1}(D_{-t-1}(u)v) = \sum_{i\geq 0}(-1)^i {t\choose i} D_{-t+i-1}(u)D_{-s-i-1}(v).
\end{eqnarray*}
To prove this, set $m{=}-s-1, n{=}-t-1\geq 0$.\ Then  the left-hand-side 
of the previous display is equal to
\begin{eqnarray*}
&&D_m(D_n(u)v) {=}\sum_{i+j=m} (D_i\circ D_n)(u)D_j(v)\\
&&\ \ \ \ \ \ \ \ \ \ \ \ \ \ \ \ \ {=} \sum_{i+j=m} {i+n\choose i}(D_{i+n})(u)D_j(v)\\
&&\ \ \ \ \ \ \ \ \ \ \ \ \ \ \ \ \ {=}\sum_{i\geq 0}(-1)^i {t\choose i} D_{-t+i-1}(u)D_{-s-i-1}(v).
\end{eqnarray*}
This establishes associativity.

\medskip
Next, because $A$ is commutative then
$D_m(u)$ and $D_n(v)$ commute as multiplication operators on $A$. Therefore,
$[Y(u, z), Y(v, w)]=0$, so that locality (\ref{locform1}) is obvious.\

\medskip
As for translation-covariance, use the previous displayed calculation to see that
\begin{eqnarray*}
[D_m, Y(u, z)]v&=& \sum_{n\geq 0}[D_m, D_n(u)]vz^n\\
&=&\sum_{n\geq 0} (D_m(D_n(u)v)-D_n(u)D_m(v))z^n\\
&=&\sum_{n\geq 0}\sum_{i\geq 1} {n+i\choose i} D_{n+i}(u)D_{m-i}(v)z^n\\
&=&\sum_{i\geq 1} \delta_z^{(i)}Y(u, z)D_{m-i}v.
\end{eqnarray*}
This completes the proof of the Theorem.
\end{proof}

All vertex rings with the property that all modes $u(n)$ \emph{vanish} for  $n\geq 0$
arise from the construction in Theorem \ref{thmcatderring}.
\begin{thm}\label{thmcatcommring} Suppose that $V$ is a vertex ring such that 
$u(n){=}0$ for all $u\in V$ and all $n{\geq} 0$.\ Then $V$ is a commutative ring with identity $\mathbf{1}$
with respect to the $-1^{th}$ product, and $V$ is a vertex ring of the type described in Theorem \ref{thmcatderring}.
\end{thm}
\begin{proof} We have $u(-1)\mathbf{1}{=}\mathbf{1}(-1)u{=}u$, so $\mathbf{1}$ is an
identity with respect to the $-1^{th}$ operation.\ Moreover, by (\ref{commform}) we see that
\begin{eqnarray*}
u(-1)v-v(-1)u &=& u(-1)v(-1)-v(-1)u(-1))\mathbf{1}\\
&=&  \sum_{i\geq 0} {-1\choose i}(u(i)v)(-2-i)\mathbf{1}=0.
\end{eqnarray*}
This shows that $V$ is
commutative with respect to the $-1^{th}$ product, and similarly using (\ref{assocform})
we find that it is associative.\ So $V$ is indeed a commutative ring with respect to the $-1^{th}$ product.

\medskip
That $\underline{D}$ is an
iterative Hasse-Schmidt derivation of this commutative ring is a special case of
Theorem \ref{thmHS}.\ Finally, we have for $m\geq 0$, 
\begin{eqnarray*}
D_m(u)(-1)v &=& (u(-m-1)\mathbf{1})(-1)v\\
&=& \sum_{i\geq 0}(-1)^i {-m-1\choose i} u(-m-1-i)\mathbf{1}(-1+i)v\\
&=& u(-m-1)v,
\end{eqnarray*}
which says that
\begin{eqnarray*}
Y(u, z)v&=&\sum_{m\geq 0} u(-m-1)vz^m = \sum_{m\geq 0} D_m(u)(-1)vz^m.
\end{eqnarray*}
This shows that $V$ is the type of vertex ring described in Theorem \ref{thmcatderring}, 
thus completing the proof of the present Theorem.
\end{proof}

\medskip
As a special case of Theorem \ref{thmcatderring}, consider \emph{any} commutative ring $A$, equipped with the 
trivial  HS derivation (cf.\ Example \ref{trivHS}).\ Then by Theorems \ref{thmcatderring} and \ref{thmcatcommring} we deduce
\begin{thm}\label{thmcommring} Let $A$ be a commutative ring with identity $1$.\ Then $A$ is a vertex ring with vacuum element $1$ and with canonical HS derivation that is \emph{trivial}.\  All $n^{th}$ products $u(n)v$ are zero for $n\not= -1$ and $u(-1)v=uv$ is the product in $A$.
$\hfill \Box$
\end{thm}

\medskip
Let $\bf{Comm}$ be the category of unital commutative rings and unital ring morphisms.\ Let
$\bf{CommHS}$ be the category of unital commutative rings equipped with an iterative HS derivation 
$\underline{D}$.\
If $(A, \underline{D}), (A', \underline{D}')$ are two objects in $\bf{CommHS}$, a morphism of 
$f: (A, \underline{D}) \rightarrow (A', \underline{D}')$ is a unital ring morphism $f:A\rightarrow A'$ 
such that $fD_m=D_m'f$ for all $m\geq 0$.\ 
$\bf{Comm}$ is \emph{isomorphic} to the full subcategory of $\bf{CommHS}$ whose
objects consist of those $(A, \underline{D})$ for which $\underline{D}=(Id_A, 0, 0, ...)$ is trivial.\ We have
the following commuting diagram of functors, each of which are insertions.

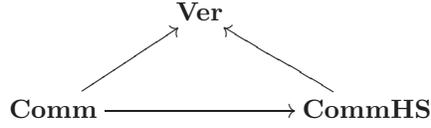
\[\xymatrix{
& \bf{Ver} &\\
\bf{Comm}\ar[ru]\ar[rr] && \bf{CommHS}\ar[lu]
 }
\]

\section{The center of a vertex ring} 

\subsection{Basic properties}
\begin{dfn}The \emph{center} $C(V)$ of a vertex ring $V$ consists  of all $u\in V$ such that $Y(u, z)=u(-1)$.
\end{dfn}

\begin{ex}\label{exvac} (i) $\mathbf{1}{\in} C(V)$.\ Therefore, the image of the characteristic map
(Example \ref{exideal}(iv)) is contained in $C(V)$.\ Since $C(V)$ is a commutative ring with respect to the $-1$ product (see Lemma \ref{lemCVcomm} for the easy proof) then $charV{=}charC(V)$.\\
(ii)\ If $a{\in} C(V)$ then $a(-1)V$ is the 2-sided ideal of $V$ generated by $a$.
\end{ex}
\begin{proof} That $\mathbf{1}{\in} C(V)$ is a restatement of Theorem \ref{thmvacuum}, and the remaining assertions (i) follow.\ As for (ii), it suffices to show that $a(-1)V$ is a 2-sided ideal.\ But because 
$Y(a, z)=a(-1)$, it is readily checked using (\ref{commform}) and (\ref{assocform}) that $a(-1)$ commutes with modes $b(n)$ and associates
in the sense that $(a(-1)b)(n)=a(-1)b(n)$ for all $b\in V$ and all $n\in\mathbf{Z}$, and our assertion
about $a(-1)V$ follows.
\end{proof}

The next result is very useful.\ In alternate parlance it says that $C(V)$ consists of the \emph{$D$-constants}.
\begin{thm}\label{thmcenter1} Let $V$ be a vertex ring with canonical Hasse-Schmidt derivation 
$\underline{D}{=}(Id_V, D_1, \hdots)$, and let $u{\in} V$.\ Then the following are equivalent:
\begin{eqnarray*}
&&(a)\ u{\in} C(V), \\
&&(b)\  D_iu {=}0\ \mbox{for}\ i\geq 1.
\end{eqnarray*}
\end{thm}
\begin{proof} If (a) holds
then $D_i(u) {=} u(-i-1)\mathbf{1}{=}0$ for $i\geq 1$, so (b) holds.\ The implication
$(b){\Rightarrow} (a)$ requires a bit more effort.\ By Lemma \ref{lemmaDiu}, if (b) holds then 
\begin{eqnarray}\label{ukcond}
{i+k\choose i}u(k){=}0\ \ (k{\in}\mathbf{Z}, i{\geq} 1).
\end{eqnarray}

\medskip
Suppose first that $k{\leq} {-}2$.\ Then we may take $i{=}{-}k{-}1$ in (\ref{ukcond}) to see that
$u(k){=}0$.\ Now suppose that $k{\geq} 0$.\ In order to prove that $u(k){=}0$, we use the following result{:}\
the integers ${k+i\choose i}$, where $i$ ranges over all
positive integers, have \emph{greatest common divisor $1$}.\ To see this, we have
\begin{eqnarray*}
(1-x)^{-(k+1)} {=} 1{+} \sum_{i\geq 1} {k+i\choose i}x^i.
\end{eqnarray*}
So if all ${k+i\choose i}$ for $i>0$ are divisible by a prime $p$ then
$(1-x)^{-(k+1)}$, and therefore also  its inverse $(1-x)^{k+1}$, are congruent
to $1$(mod $p$).\ But this is false because the coefficient of $x^{k+1}$ in the latter
polynomial is $\pm 1$.

\medskip
It follows that for any  $k{\geq} 0$, we can find an integral linear combination of the binomial coefficients
${i+k\choose i}$ (with $i{\geq} 1$ varying) equal to $1$.\ Then from (\ref{ukcond}) we deduce
that in fact $u(k)=0$.\ This completes the proof of the Theorem.
\end{proof}

\medskip
We give some first applications of Theorem \ref{thmcenter1}.\ 
\begin{cor} Suppose that $V$ is a vertex ring whose canonical HS derivation 
is trivial.\ Then $V$ is of the type described in Theorem \ref{thmcommring}.
\end{cor}
\begin{proof} We have $D_m=0$ for $m\geq1$, so by Theorem \ref{thmcommring}
we find that $Y(u, z)=u(-1)$ for all $u\in V$, and therefore $V{=}C(V)$.\ Now the Corollary follows from
Theorem \ref{thmcatcommring}.
\end{proof}

\medskip
We may state the Corollary is as follows{:}\ if $\bf{Comm}'$ is the full subcategory
of $\bf{Ver}$ consisting of objects which are vertex rings whose canonical HS derivation is \emph{trivial},
then $\bf{Comm}'$ is \emph{isomorphic} to the category $\bf{Comm}$ of commutative rings.\ As a result,
we may, and shall, identify the two categories $\bf{Comm}$ and $\bf{Comm}'$.\
The next result says that thus identified, $\bf{Comm}$ is \emph{coreflective} in $\bf{Ver}$ (cf.\ \cite{Mac}, Chapter IV, Section 3)

\begin{thm}\label{thmladjoint}
The functorial insertion $K{:}\bf{Comm}{\rightarrow} \bf{Ver}$ has a \emph{right adjoint} $C{:}\bf{Ver}\rightarrow \bf{Comm}$.\
$C$ is the \emph{center functor} $V{\rightarrow} C(V)$ that assigns to each vertex ring $V$ its center
$C(V)$.
\end{thm}

\medskip
We discuss the proof of the Theorem in the next few Lemmas.

\medskip
Theorem \ref{thmladjoint} may be reformulated in several ways
(loc.\ cit.)\ One particularly useful variant, which is more or less equivalent to the Theorem, is the following{:}
\begin{lem}\label{lemfcentral} Suppose that $A$ is a commutative ring, $V$ a vertex ring, and
$f{:}A\rightarrow V$ a morphism of vertex rings.\ Then $f(A){\subseteq} C(V)$.
\end{lem}
\begin{proof}
Suppose to begin with, and in somewhat more generality than we need, that 
$(A, \underline{D})$ is an object in $\bf{CommHS}$ and that $\underline{D'}$ is the canonical HS derivation of $V$.\ Then for $u\in A$ we have
\begin{eqnarray*}
f(D_iu) &=& f(u(-i-1)1)= f(u)(-i-1)\mathbf{1} = D_i'f(u).
\end{eqnarray*}

In the case at hand, all higher endomorphisms $D_i\ (i\geq 1)$ of $A$ \emph{vanish}.\ Thus
the last calculation shows that $f(A)$ is annihilated by all $D_i'\ (i\geq 1)$.\ By Theorem
\ref{thmcenter1} we deduce that $f(A){\subseteq} C(V)$, and the Lemma is proved.
\end{proof}

\begin{lem}\label{lemCVcomm} Let $V$ be a vertex ring.\ Then $C(V)$ is a
commutative ring with identity $\mathbf{1}$ with respect to the ${-}1^{th}$ product.\ The assignment
$V\mapsto C(V)$ is the object map of a functor $C{:} \bf {Ver}\rightarrow {\bf Comm}$.
\end{lem}
\begin{proof} We have  $\mathbf{1}{\in} C(V)$ by Example \ref{exvac}, and
because $u(-1)\mathbf{1}{=}\mathbf{1}(-1)u{=}u$, $\mathbf{1}$ is the identity element for $C(V)$.\
Moreover, if $u, v{\in} C(V)$ and $n{\geq} 1$ then $D_iu{=}D_iv{=}0$ for $i{\geq} 1$ and hence
\begin{eqnarray*}
D_nu(-1)v = \sum_{i+j=n} (D_iu)(-1)D_jv=0.
\end{eqnarray*}
This shows that $C(V)$ is closed with respect to the $-1^{th}$ product.

\medskip Taking $r{=}s{=}{-}1$ in (\ref{commform}), applying both sides to $\mathbf{1}$, and using
$u(i){=}0$ for $i{\geq} 0$ shows that $u(-1)v{=}v(-1)u$, and a similar application of (\ref{assocform})
shows that multiplication in $C(V)$ is associative.\ This completes the proof of the first assertion of the Lemma.

\medskip
To complete the proof that $C{:}\bf {Ver}{\rightarrow} {\bf Comm}$ is a functor, suppose that $f{:}U{\rightarrow} V$
is a morphism of vertex rings.\ Restriction of $f$ to $C(U)$ is a morphism of
vertex rings $res{f}{:}C(U)\rightarrow V$, and by Lemma \ref{lemfcentral} the image of this
morphism lands in $C(V)$.\ This says that by defining $C(f):= res{f}$ we obtain a commuting diagram
\[\xymatrix{
 U\ar[d]_{f} \ar[r] &C(U)\ar[d]^{C(f)}\\
V\ar[r] & C(V)
 }
\]
That $C$ is a functor then follows immediately.
\end{proof}

\medskip
To complete the proof of Theorem \ref{thmladjoint}, let $A$ be a commutative ring
and  $V$ a vertex ring.\ What we must show is that there is a natural bijection of Hom-sets
\begin{eqnarray*}
\varphi: Hom_{\bf Ver}(K(A), V) \rightarrow Hom_{\bf Comm}(A, C(V)).
\end{eqnarray*}
Indeed, $K(A)$ is just $A$ regarded as a vertex ring, and application of Lemma \ref{lemfcentral} shows that
the two displayed Hom-sets consist of the same functions (with different codomains).\ This gives us the required bijection, and naturality follows.$\hfill \Box$

\subsection{Vertex $k$-algebras}\label{SSvertkalg}
Fix a vertex ring $U$.\ As usual\ (cf.\ Chapter II, Section 6 of \cite{Mac}), the \emph{comma category} 
$(U{\downarrow} \bf{Ver})$ consists of the `$\bf{Ver}$-objects under $U$'.\ Precisely, it has objects consisting of morphisms $U{\rightarrow} V$ in $\bf{Ver}$.\ A morphism $f$ in $(U\downarrow \bf{Ver})$ from $U{\rightarrow} V_1$ to $U{\rightarrow} V_2$ is a morphism $f{:}V_1\rightarrow V_2$ in $\bf{Ver}$ such that the following diagram in $\bf{V}$ commutes:
\[\xymatrix{
& U\ar[dl]\ar[dr] &\\
V_1\ar[rr]_f && V_2
 }
\]

\begin{dfn}
Let $k$ be a commutative ring, considered as an object in $\bf{Ver}$.\ We call the objects in  $(k{\downarrow} \mathbf{V})$
\emph{vertex $k$-algebras}.\
\end{dfn}

Suppose that $\varphi{:}k{\rightarrow} V$ is a vertex $k$-algebra.\ By Lemma \ref{lemfcentral} we have
$\varphi(k){\subseteq} C(V)$.\ We claim that the left action of $\varphi(k)$ on $V$ by the $-1$ operation
induces a left action
\begin{eqnarray*}
k{\times} V\rightarrow V, (a, v)\mapsto \varphi(a)(-1)v\ (a{\in} k)
\end{eqnarray*}
which turns $V$ into a unital left $k$-module such that all $n^{th}$ products in $V$ are $k$-linear.\
Indeed, this amounts to the identities
\begin{eqnarray*}
t(-1)(u(n)v){=}(t(-1)u)(n)v {=}u(n)(t(-1)v))\ \ (t{\in} C(V), u, v{\in} V, n{\in}\mathbf{Z}),
\end{eqnarray*}
and these are easily proved using (\ref{assocform}) and (\ref{commform}) and  the fact
that $Y(t, z)=t(-1)$ for $t\in C(V)$.

\medskip
We have a category k$\bf{Ver}$ whose objects are vertex rings which are also
unital left $k$-modules such that all $n^{th}$ products are $k$-linear.\ Morphisms
in k$\bf{Ver}$ are morphisms in $\bf{Ver}$ that are also $k$-linear.\ Our remarks in the previous paragraph then amount to showing that there is a functor $(k{\downarrow} \bf{Ver}){\rightarrow}k\bf{Ver}$.\
On the other hand, given an object $V$ in k$\bf{Ver}$,  the map $a\mapsto a.\mathbf{1}\ (a{\in} k)$ defines a morphism $k{\rightarrow} V$ in $\bf{Ver}$
and thereby an object in the comma category $(k{\downarrow} \bf{Ver})$
and thereby a functor $k\bf{Ver}{\rightarrow} (k{\downarrow} \bf{Ver})$.\ These two functors are inverse to each other
up to natural equivalence, and we  obtain the following result.
\begin{thm} Let $k$ be a commutative ring.\ There is an equivalence of categories
 \begin{eqnarray*}
 (k{\downarrow} \mathbf{Ver}) \cong k\bf{Ver}.
 \end{eqnarray*}\ $\hfill \Box$
\end{thm}

We complete this Subsection with some examples of vertex $k$-algebras related to \emph{base-change}.\ Suppose we have a unital morphism of commutative rings
$\varphi{:} k{\rightarrow}R$ and a vertex $k$-algebra $V$.\ The base-change
\begin{eqnarray*}
V_R:= R\otimes_kV
\end{eqnarray*}
produces a left $R$-module where $b(a\otimes u) = (ba)\otimes u\ (a, b\in R, u\in V).$\ We easily check
that $V_R$ is a vertex $R$-algebra if we define $n^{th}$ modes in the obvious way,
i.e., $(a\otimes u)(n)(b\otimes v) := ab\otimes (u(n)v)$.

\medskip
For example, if we take $R{:=}k[[t]]$ with an indeterminate $t$ and the natural embedding 
$k{\rightarrow} R$, we obtain the \emph{power series vertex $k[[t]]$-algebra with coefficients in $V$}:
\begin{eqnarray*}
V[[t]]:= k[[t]]{\otimes_k} V.
\end{eqnarray*}
This vertex algebra  and similar ones with
$R{=}k[t], k[[t, t^{-1}]]$, or $k[t, t^{-1}]$, for example, and $k{=}\mathbf{C}$, are commonly used in VOA  theory, where they are usually regarded as vertex $\mathbf{C}$-algebras.

\medskip
$V[[t]]$ plays a r\^{o}le in an alternate treatment of HS derivations which is standard in the theory of  rings with derivation\ (cf.\ Section 1 of \cite{Mats}).\ When pursued, this line of argument  leads to connections between vertex rings and formal group laws.\ We will not carry this out here.
\begin{lem} Suppose that $V$ is a vertex $k$-algebra with canonical HS derivation
$\underline{D}=(Id_V, D_1, ...)$.\ The map
\begin{eqnarray*}
\alpha{:}V{\rightarrow} V[[t]],\ \ u\ {\mapsto} \sum_{m\geq 0} D_m(u)t^m
\end{eqnarray*}
is a morphism of vertex $k$-algebras.\ Moreover, it has a unique extension
to an \emph{automorphism} $\alpha{:}V[[t]]{\rightarrow} V[[t]]$ satisfying 
$\alpha(\mathbf{1}\otimes t){=}\mathbf{1}\otimes t$.
\end{lem}
\begin{proof} The argument in \cite{Mats} is essentially unchanged.\ To prove the first part of the Lemma, 
notice that $\alpha$ fixes the vacuum
$\mathbf{1}$ because $D_m(\mathbf{1}){=}0\ (m\geq 1)$.\ Moreover
\begin{eqnarray*}
&&\alpha(u(n)v){=}\sum_{m\geq 0} D_m(u(n)v)t^m {=} \sum_{m\geq 0} \sum_{i+j=m} D_i(u)(n)D_j(v)t^m\\
&&\ \ \ \ \ \ \ \ \ \ \ \ \ {=}\left(\sum_{i\geq 0}D_i(u)t^i\right)(n)\sum_{j\geq 0} D_j(v)t^j {=} \alpha(u)(n)\alpha(v).
\end{eqnarray*}
This shows that $\alpha$ is a morphism of vertex $k$-algebras.

\medskip
We assert that there is a \emph{unique} extension of $\alpha$ to a morphism
$\alpha{:}V[[t]]{\rightarrow} V[[t]]$ of vertex $k[[t]]$-algebras satisfying $\alpha(t){=}t$.\ Indeed, because powers of $t$ associate (because $t{\in}C(V[[t]])$) it is easy to see that if the extension exists then we must have
$\alpha(t^n)=\alpha(t)^n=t^n\ (n\geq 1)$, so that
\begin{eqnarray*}
\alpha\left(\sum_{i\geq 0} a_it^i \right) = \sum_{i\geq 0} \alpha(a_i)t^i.
\end{eqnarray*}
Thus there is only one possible extension, and the defined action of
$\alpha$ on $V[[t]]$ \emph{does} work, because
\begin{eqnarray*}
&&\alpha\left(\left(\sum_i a_it^i\right)(n)\left(\sum_j b_jt^j\right)\right) = \alpha\left(\sum_k \left(\sum_{i+j=k} 
a_i(n)b_j\right) t^k\right)\\
&&{=}\sum_k \left(\sum_{i+j=k} \alpha(a_i)(n)\alpha(b_j)\right)t^k = \alpha\left(\sum_i a_it^i\right)(n)\alpha\left(\sum_j b_jt^j\right).
\end{eqnarray*}

In fact, we assert that so defined, $\alpha$ is \emph{surjective}, hence is an \emph{automorphism} of $V[[t]]$.\ Given any 
$\sum_j b_jt^j$, we have to solve recursively for $a_i{\in} V$ that satisfy
\begin{eqnarray}\label{auto}
\sum_i \alpha(a_i)t^i {=}\sum_j b_jt^j.
\end{eqnarray}

\medskip
Since $\alpha(a_0){=}a_0{+} O(t)$ we must have $a_0{=}b_0$.\ Suppose we have found $a_0, ..., a_{n-1}$
such that (\ref{auto}) holds modulo $t^n$.\ Let $\alpha(a_i) {=} \sum_{m} a_{im}t^m\ (0\leq i\leq n-1)$.\ By
(\ref{auto}) we must have
\begin{eqnarray*}
b_n {=} \sum_{i=0}^{n-1} a_{in}{+}a_n,
\end{eqnarray*}
so $a_n$ is uniquely determined.\ The Lemma is proved. 
\end{proof}

\subsection{Idempotents}\label{SSidem}
We consider idempotents in a vertex ring.\ They have a r\^{o}le to play in Part II.
 \begin{dfn} Let $V$ be a vertex ring.\ An \emph{idempotent} is an element
 $e{\in} V$ such that $Y(e, z)e{=}e$, i.e., $e(n)e{=}\delta_{n, -1}e$.
 \end{dfn}
 
 \begin{ex}\label{exidem} An idempotent in the commutative ring $C(V)$ is an idempotent of $V$.
 \end{ex}
 \begin{proof} If $e{\in} C(V)$ is an idempotent then $Y(e, z){=}e(-1)$
 and $e {=} e(-1)e$.
 \end{proof}

\begin{lem}\label{lemmaidem} Let $V$ be a vertex ring with $e{\in} V$.\ Then the following are equivalent.
\begin{eqnarray*}
&&(a)\ e(n)e{=}0\ \mbox{for}\ n{\geq} 0\ \mbox{and}\ e(-1)e{=}e,\\
&&(b)\ e\ \mbox{is an idempotent in $V$},\\
&&(c)\ e\ \mbox{is an idempotent in $C(V)$}.
\end{eqnarray*}
\end{lem}
\begin{proof} After Example \ref{exidem}, the implications (c)$\Rightarrow$(b)$\Rightarrow$(a) are obvious, so we only need to prove that (a)$\Rightarrow$(c).\ We adapt a standard argument.

\medskip
 First note that if $e(n)e{=}0$ for $n{\geq} 0$ then all modes of $e$ \emph{commute} with each other thanks to Lemma \ref{lemcommvector}.\ Let $\underline{D}=(Id, D_1, \hdots)$ be the canonical HS derivation of $V$.\ By
Theorem \ref{thmcenter1} it suffices to show that $D_n(e){=}0$ for $n{\geq} 1$.\ We prove this by induction on $n$.
First we have
\begin{eqnarray*}
D_1(e){=}D_1(e(-1)e){=}D_1(e)(-1)e+e(-1)D_1(e){=}2e(-1)D_1(e),
\end{eqnarray*}
where we have used that modes of $e$ commute.\ Therefore
\begin{eqnarray*}
e(-1)D_1(e){=}2e(-1)^2D_1(e){=}2e(-1)D_1(e),
\end{eqnarray*}
leading to $0{=}e(-1)D_1(e){=}2e(-1)D_1(e){=}D_1(e)$.

\medskip
Similarly, if $n{\geq} 2$ then
$D_n(e){=}D_ne(-1)e{=}\sum_{i=0}^n D_i(e)(-1)D_{n-i}(e)$, and by induction it follows that
$D_n(e){=}2 e(-1)D_n(e)$.\ We  deduce that $D_n(e){=}0$ just as in the case $n{=}1$, and the proof
of the Lemma is complete.
\end{proof}

\medskip

Let $k$ be a commutative ring and $V$ a vertex $k$-algebra with canonical HS derivation
$\underline{D}=(Id, D_1, \hdots)$.\ The \emph{endomorphism algebra} of $V$ is defined as follows:
\begin{eqnarray*}
E(V){:=}\{f{\in} End_k(V){\mid} fY(v, z){=}Y(v, z)f\ (v\in V), \ fD_m{=}D_mf\ (m{\geq} 1)\}.
\end{eqnarray*}
Here, $End_k(V)$ is the $k$-algebra of $k$-linear endomorphisms of the $k$-module $V$.\ The next result follows
an argument  in \cite{DM1}.

\begin{lem}\label{lemmaEV} There is an \emph{isomorphism} of $k$-algebras 
\begin{eqnarray*}
\varphi{:}C(V)\stackrel{\cong}{\longrightarrow} E(V),\ a\mapsto \varphi_a{:} v\mapsto a(-1)v\ (a{\in} C(V), v{\in} V).
\end{eqnarray*}
\end{lem}
\begin{proof} Define $\varphi_a\ (a{\in}C(V))$ as in the statement of the Lemma.\ If $a{\in} C(V)$ then $a(-1)$ commutes with all modes $v(n)$, and it follows 
immediately that $\varphi_a$ also commutes with all $v(n)$.\ Moreover,
\begin{eqnarray*}
D_m\varphi_av{=}(a(-1)v)(-m-1)\mathbf{1}{=}a(-1)v(-m-1)\mathbf{1}{=}\varphi_aD_mv,
\end{eqnarray*}
 showing
that $\varphi$ commutes with each $D_m$.\ Therefore,  $\varphi_a{\in} E(V)$.\ Thus
$\varphi{:}a{\mapsto} \varphi_a$ defines a map $C(V){\rightarrow} E(V)$, and because we have
$(a(-1)b)(-1)v{=}a(-1)b(-1)v$ for $a, b{\in} C(V)$ then $\varphi$ is a morphism of rings.

\medskip
To see that $\varphi$ is \emph{surjective}, let $f{\in} E(V)$ and $v{\in} V$.\ By creativity we have
 $f(v){=}fv(-1)\mathbf{1}{=}v(-1)f(\mathbf{1})$.\ Furthermore
$D_nf(\mathbf{1}){=} fD_n(\mathbf{1}){=}0$ for $n{\geq} 1$, so $f(\mathbf{1}){\in} C(V)$ by Theorem \ref{thmcenter1}.\
Finally, we have $\varphi_{f(\mathbf{1})}(v){=}f(\mathbf{1})(-1)v {=}v(-1)f(\mathbf{1}){=}f(v)$,
showing that $f{=}\varphi_{f(\mathbf{1})}$.\ The Lemma now follows.
\end{proof}

Idempotents in a vertex $k$-algebra $V$ determine decompositions of $V$ into direct sums of ideals,
just as for commutative rings.\ If $e$ is an idempotent in $V$ then $e\in C(V)$ by Lemma
\ref{lemmaidem}, and it follows easily that 
\begin{eqnarray}\label{dirsumdecomp}
V {=} e(-1)V{\oplus} (\mathbf{1}-e)(-1)V
\end{eqnarray}
is a decomposition of $V$ into the direct sum of ideals, each of which is itself
a vertex $k$-algebra.\ (The corresponding vacuum elements are $e$ and $\mathbf{1}-e$.)
The projection $V{\rightarrow} e(-1)V$ is the idempotent in $E(V)$ that corresponds to $\varphi_e$
in the isomorphism described in Lemma \ref{lemmaEV}.\ In particular,
$V$ is \emph{indecomposable} as a vertex $k$-algebra if, and only if, $C(V)$ is an indecomposable commutative ring.

\medskip
Conversely, given a pair of vertex $k$-algebras $U, V$, their \emph{direct sum} $U{\oplus}V$ is a vertex $k$-algebra
with $Y(u{\oplus} v, z)=Y(u, z){\oplus} Y(v, z)\ (u{\in} U, v{\in} V)$ and 
$\mathbf{1}_{U{\oplus} V}=\mathbf{1}_U{\oplus}\mathbf{1}_V$.\ The vacuum elements $\mathbf{1}_U, \mathbf{1}_V$
become idempotents in $U{\oplus} V$.\ This construction defines a
\emph{product} in $\mathbf{Ver}$.\ Moreover, the abelian group $0$ is a terminal object, so that 
$\mathbf{Ver}$ has all finite products.

\subsection{Units}
\begin{dfn}\label{dfnunit} Let $V$ be a vertex ring.\ A \emph{unit} in $V$ is an element
 $a{\in} V$ such that for some $b{\in} V$ we have $Y(a, z)b{=}\mathbf{1}$, i.e., $a(n)b{=}\delta_{n, -1}\mathbf{1}$.
  \end{dfn}
 
  \begin{ex} A unit in the commutative ring $C(V)$ is a unit of $V$.
 \end{ex}
 \begin{proof} If $a{\in} C(V)$ is a unit, $Y(a, z){=}a(-1)$
 and $a(-1)b {=} \mathbf{1}$ for some $b{\in} C(V)$.
 \end{proof}

\begin{lem}\label{lemmaunit} Let $V$ be a vertex ring, and suppose that $a, b{\in} V$ satisfy
$Y(a, z)b{=}\mathbf{1}$.\ Then $Y(b, z)a{=}\mathbf{1}$ and $a, b{\in} C(V)$.
\end{lem}
\begin{proof} Let $(Id, D_1, \hdots)$ be the canonical HS derivation of $V$.\ We have $a(n)b=\delta_{n, -1}\mathbf{1}$, and by skew-symmetry 
(Lemma \ref{lemmaskewsymm}) it follows that
\begin{eqnarray*}
&&b(n)a{=}(-1)^{n+1}\sum_{i\geq 0}(-1)^i D_i(a(n+i)b) \\
&&\ \ \ \ \ \ \  {=} (-1)^{n+1}\sum_{i\geq 0}(-1)^i \delta_{n+i, -1}D_i(\mathbf{1})=\delta_{n, -1}\mathbf{1}.
\end{eqnarray*}
This proves the first assertion of the Lemma.

\medskip
Next we observe that $[a(r), b(s)]=0$ for all $r, s\in\mathbf{Z}$.\ This follows from
Lemma \ref{lemcommvector}.\ Now let $m\geq 1$.\ Then using (\ref{assocform}) we have
\begin{eqnarray*}
&&0 {=} (b(-m-1)a)(-1)b \\
&&\ \ \ {=} \sum_{i\geq 0}(-1)^i{-m-1\choose i}\{b(-m-1-i)a(-1+i)+(-1)^ma(-m-2-i)b(i)\}b\\
&&\ \ \ {=} b(-m-1)a(-1)b = b(-m-1)\mathbf{1}=D_m(b).
\end{eqnarray*}

This shows that $D_m(b){=}0\ (m\geq 1)$, and  similarly  $D_m(a){=}0\ (m\geq 1)$.
Therefore $a$ and $b$ lie in $C(V)$  by Theorem \ref{thmcenter1}.\ This completes the proof of
the Lemma.
\end{proof}

 \begin{ex}\label{exunit} (i) Let $V$ be a vertex ring with $a{\in} C(V)$.\ The 2-sided ideal $a(-1)V$ of $V$ generated by $a$ (cf.\ Example \ref{exvac}(ii)) satisfies $a(-1)V=V$ if, and only if, $a$ is a unit.\\
 (ii) If $V$ is \emph{simple} then $C(V)$ is a \emph{field}.
 \end{ex}
 \begin{proof} (i)\ $a(-1)V{=}V {\Leftrightarrow} a(-1)b{=}\mathbf{1}\ (\exists\  b{\in} V) {\Leftrightarrow}
 Y(a, z)b{=}\mathbf{1}$.\ (ii) follows immediately from (i).
 \end{proof}

\subsection{Tensor product of vertex rings}
The category $\bf{Ver}$ of vertex rings has a \emph{coproduct} given by the tensor product
$U{\otimes} V$ of a pair of vertex rings $U, V$.\ (It also has a product, given by the direct sum
construction explained in the previous Subsection.)\ We give a proof using
 Theorem \ref{thmexist1}.

\medskip
The underlying abelian group is the tensor product of the abelian groups $U$ and $V$,
and the vacuum element for $U{\otimes} V$ is taken to be $\mathbf{1}{\otimes}\mathbf{1}$.\ 
The vertex operators for $U{\otimes} V$ are defined in the obvious manner, i.e., if
$u{\in} U, v{\in} V$ we define
\begin{eqnarray*}
Y(u{\otimes} v, z) {=} Y(u,z){\otimes} Y(v, z).
\end{eqnarray*}
In terms of modes, this means that
\begin{eqnarray*}
(u{\otimes} v)(n) {=} \sum_{i+j=n-1} u(i){\otimes} v(j),
\end{eqnarray*}
which shows in particular that $Y(u{\otimes} v, z){\in}\mathcal{F}(U{\otimes} V)$.\ Moreover,
\begin{eqnarray*}
Y(u{\otimes} v, z)(\mathbf{1}{\otimes}\mathbf{1}){=}\sum_{m\geq 0} \sum_{i+j=m} 
u(-i-1)\mathbf{1}{\otimes} v(-j-1)\mathbf{1}z^{m}.
\end{eqnarray*}
This proves the creative property, moreover it shows that the canonical HS-derivation $\underline{D'}=(Id, D_1',  \hdots)$ for 
$U{\otimes} V$ must necessarily be defined by
\begin{eqnarray}\label{D'def}
D_m'{=}\sum_{i+j=m} D_i{\otimes} D_j,
\end{eqnarray}
where we have abused notation by setting $(Id, D_1,  \hdots)$ for the canonical HS derivations of \emph{both} 
$U$ and $V$.\ 
The mutual locality of the fields $Y(u{\otimes} v, z)$ for $u, v\in V$ is an easy consequence of the locality of
the $Y(u, z)$, so it remains to establish translation-covariance.

\medskip
Although it is not necessary to know that $\underline{D'}$ is an HS-derivation in order to apply
Theorem \ref{thmexist1}, in the present instance it will be convenient to know in advance that it is.\ To see this,
we calculate
\begin{eqnarray*}
D_m'(u{\otimes} v)(n)(a\otimes b)&=&\sum_{i+j=n-1}D_m'(u(i)a{\otimes} v(j)b)\\
&=&\sum_{p=0}^m \sum_{i+j=n-1}(D_pu(i)a){\otimes} (D_{m-p}v(j)b)\\
&=&\sum_{p=0}^m \sum_{i+j=n-1}\sum_{r=0}^p\sum_{s=0}^{m-p} (D_{r}u)(i)(D_{p-r}a){\otimes}(D_sv)(j)(D_{m-p-s}b)\\
&=&\sum_{p=0}^m \sum_{r=0}^p\sum_{s=0}^{m-p} (D_{r}u{\otimes} D_sv)(n)(D_{p-r}a{\otimes} D_{m-p-s}b)\\
&=&\sum_{t=0}^m D_t'(u{\otimes} v)(n)D_{m-t}'(a{\otimes} b),
\end{eqnarray*}
and this is what we required.

\medskip
Turning to the proof of translation-covariance, we calculate
\begin{eqnarray*}
&&Y(D_m'(u{\otimes} v), z) {=} Y\left(\sum_{p+q=m} D_p(u){\otimes} D_q(v), z\right)\\
&=&\sum_{p+q=m}Y(D_p(u), z){\otimes} Y(D_q(v), z) {=} \sum_{p+q=m} \delta_z^{(p)}Y(u,z)\otimes \delta_z^{(q)}Y(v, z) \\
&=&
(-1)^m\sum_{p+q=m} \sum_{\ell}\sum_n {\ell\choose p}{n\choose q}u(\ell-p)\otimes v(n-q)z^{-\ell-n-2}.
\end{eqnarray*}

The coefficient of $z^{-t-1}$ in this expression is equal to
\begin{eqnarray*}
&&(-1)^m\sum_{p+q=m} \sum_{\ell} {\ell\choose p}{t-\ell -1\choose q}u(\ell-p){\otimes} v(t-\ell-q-1)\\
&=&(-1)^m\sum_{p+q=m} \sum_{i} {p+i\choose p}{t-p-i-1\choose q}u(i){\otimes} v(t-i-m -1)\\
&=&(-1)^m{t\choose m}\sum_i u(i){\otimes} v(t-i-m-1)\ \ (\mbox{use (\ref{bi1}) and (\ref{bi4})}))\\
&=&(-1)^m{t\choose m} (u{\otimes} v)(t-m),
\end{eqnarray*}
and this establishes that
\begin{eqnarray*}
&&Y(D_m'(u{\otimes} v), z)= (-1)^m\sum_t {t\choose m} (u\otimes v)(t-m)z^{-t-1} = \delta_z^{(m)}Y(u{\otimes} v, z).\\
\end{eqnarray*}

Having already shown that $\underline{D'}$ is an HS derivation, the proof that $Y(u{\otimes} v, z)$ is translation-covariant with respect to $(Id, D_1', D_2', \hdots)$ now
follows, using the previous display, in exactly the same way that part (c) of Theorem \ref{lend} is deduced from part (b).\ This completes the proof that $U{\otimes} V$, equipped with vacuum vector
$\mathbf{1}{\otimes}\mathbf{1}$, vertex operators $Y(u, z){\otimes} Y(v, z)$, and endomorphsms $\underline{D'}$, 
is indeed a vertex ring.

\medskip
 We also observe that the same proof works \emph{mutatis mutandis} if $U$ and $V$ are both vertex $k$-algebras for some commutative ring $k$ and $U{\otimes} V$ is taken to mean the tensor product $U{\otimes}_k V$
of $k$-modules.

\medskip
We have the usual diagram
\begin{eqnarray}\label{tpuniversal}
U \stackrel{i}{\longrightarrow} U\otimes V \stackrel{j}{\longleftarrow} V
\end{eqnarray}
where $i{:}u\mapsto u{\otimes}\mathbf{1}$ and $j{:} v\mapsto \mathbf{1}{\otimes} v$.\ Both $i$ and $j$
are  morphisms of vertex $k$-algebras, and the images of $i$ and $j$ jointly generate
$U\otimes V$ on account of the formula $u{\otimes} v=(u{\otimes} \mathbf{1})(-1)(\mathbf{1}{\otimes} v)$.\
The universal property of (\ref{tpuniversal}) that shows it is a coproduct in $k\mathbf{Ver}$ is then easy to see, and we have proved
\begin{thm} Suppose that $U$ and $V$ are a pair of vertex $k$-algebras.\ Then $U{\otimes}_k V$ carries a natural structure of vertex $k$-algebra with vertex operators defined by $Y(u{\otimes} v, z){=}Y(u, z){\otimes} Y(v, z)$, vacuum vector
$\mathbf{1}{\otimes}\mathbf{1}$, and  canonical HS-derivation $\underline{D'}$ defined by (\ref{D'def}).\
$U{\otimes}_k V$ defines a coproduct in $k\mathbf{Ver}$. $\hfill \Box$
\end{thm}

\noindent
\begin{ex} If $k{\rightarrow} R$ is a morphism of commutative rings and $V$ a vertex $k$-algebra,
the tensor product $R {\otimes}_k V$ is a vertex $k$-algebra where  $R$ has the trivial HS-derivation.\ Since
$(a{\otimes}\mathbf{1})(n)=a(n){\otimes} Id=\delta_{n, -1}a(-1){\otimes}\mathbf{1}\ (a{\in} R)$ then
$i(R)$ is contained in $C(R{\otimes}_k V)$ and $R{\otimes}_k V$ is then a vertex $R$-algebra.\ This is
the base-change that we discussed in Subsection \ref{SSvertkalg}.
\end{ex}

\section{Virasoro vertex $k$-algebras}\label{SVir}
The Virasoro algebra $Vir$ (over $\mathbf{C}$) is a well-known Lie algebra
giving rise to vertex algebras over $\mathbf{C}$ in a standard manner
(\cite{FZ}, \cite{LL}, \cite{MN}).\ In this Section we show how to construct Virasoro vertex
$k$-algebras associated to \emph{any} commutative ring $k$ and any \emph{quasicentral charge} in $k$
(cf.\ Definition \ref{dfnqcc}).\ This leads to the Definition and first examples of \emph{vertex operator
algebras over a commutative ring $k$}.

\subsection{The Lie algebra $Vir_k$}
Fix a commutative ring $k$.\ We consider a free $k$-module
\begin{eqnarray*}
Vir_k := \oplus_{n\in\ZZ} kL(n) \oplus kK
\end{eqnarray*}
with $k$-basis consisting of $L(n)\ (n{\in}\ZZ)$ and $K$.\ We make $Vir_k$ into a Lie algebra
over $k$ by defining the bracket relations{:}
\begin{eqnarray}\label{Virdef}
&&[L(m), L(n)] {=} (m-n)L(m+n) + \delta_{m+n, 0}{m+1\choose 3}K,\\
&&{[}L(m), K{]} {=} 0. \notag
\end{eqnarray}

\begin{rmk}\label{rmkVirdef} If $k{=}\mathbf{C}$, one usually replaces ${m+1\choose 3}$
with $\frac{m^3-m}{12}$ in this definition.\ This  amounts to a rescaling of the central element
$K$.\ We prefer to use (\ref{Virdef}) because it makes sense for all rings $k$, whereas
the other choice is problematic if $2$ is not a unit in $k$.
\end{rmk}

$Vir_k$ has a triangular decomposition into Lie $k$-subalgebras 
\begin{eqnarray*}
Vir_k {=} Vir_k^+\oplus Vir_k^0\oplus Vir_k^-,
\end{eqnarray*}
where
\begin{eqnarray*}
&&Vir_k^+:= \oplus_{m>0} kL(m), \ Vir_k^-:=\oplus_{m<0} kL(m),\\
&& Vir_k^0:= kL(0)\oplus kK.
\end{eqnarray*}
$Vir_k^{\geq}:= Vir_k^+ {\oplus} Vir_k^0$ is also Lie $k$-subalgebra of $Vir_k$.

\medskip
For each $c'\in k$, there is a 1-dimensional free $Vir_k^{\geq}$-module $U_{c'}:=kv_0$
such that $L(m).v_0=0\ (m\geq 0)$ and $K.v_0=c'v_0$.\ We consider the induced (Verma)
module
\begin{eqnarray}\label{verc'def}
Ver_{c'} := Ind_{Vir^{\geq}}^{Vir} kv_0 = Vir_k{\otimes}_{Vir_k^{\geq}} U_{c'}.
\end{eqnarray}
$Ver_{c'}$ is a free $k$-module with a basis  consisting of states
\begin{eqnarray*}
\{u(n_1, ..., n_k):= L(-n_1)...L(-n_k).v_0 \mid n_1\geq n_2 \geq ...\geq n_k\geq 1\}.
\end{eqnarray*}
Adopting standard notation, we set
\begin{eqnarray*}
&&\omega(z) := \sum_n \omega(n)z^{-n-1}= \sum_{n} L(n)z^{-n-2},
\end{eqnarray*}
so that $\omega(n)=L(n-1)$.

\begin{lem}\label{lemvirprops} The following hold:
\begin{eqnarray*}
&&(a)\ \omega(z)\in \mathcal{F}(Ver_{c'}),\\
&&(b)\ \omega(z)\sim_4 \omega(z).
\end{eqnarray*}
\end{lem}
\begin{proof} If $n$ and $n_1{\geq}  ...{\geq} n_k{\geq} 1$ are integers with $n{>}n_1{+}...{+}n_k$, an
induction based on the relations in $Vir_k$ shows that $L(n).u(n_1, ..., n_k)=0$.\ Part
(a) of the Lemma follows easily from this statement.\

\medskip
Part (b) asserts that $(z-y)^4[\omega(z), \omega(y)]{=}0$ (cf.\ Definition \ref{field props}(c)).\ This is
a famous relation whose proof in the case $k{=}\mathbf{C}$ carries over unchanged
to the present context.\ We skip the details and refer the reader to \cite{MN}, Section 9.4 and \cite{LL}, Section 6.1.
\end{proof}

\subsection{The Virasoro vertex ring $M_k(c', 0)$}\label{SSVvring}
We set
\begin{eqnarray*}
M_k(c', 0){:=} M(c', 0){:=} Ver_{c'}/I_{c'},
\end{eqnarray*}
where $I_{c'}{:=} Vir_k.L(-1)$  is the $Vir_k$-submodule of $Ver_{c'}$  generated
by $L(-1)$.\ $M(c', 0)$ is a free $k$-module with basis consisting of states
\begin{eqnarray*}
\{u'(n_1, ..., n_k){:=} L(-n_1)...L(-n_k).v_0+I_{c'}  \mid n_1{\geq} n_2 {\geq} ...{\geq} n_k{\geq} 2\}.
\end{eqnarray*}

\medskip
There is an induced action of operators and fields such as
 $L(n)$ and $\omega(z)$ on  $M(c', 0)$.\ We will often use the same symbol for such operators
 on both the Verma module and its quotient.\ This should cause no confusion.
 
 \medskip
 The bulk of this Section is taken up with the proof of the next result.
 \begin{thm}\label{thmVirex} Set $\omega{:=}L(-2)v_0 {+}I_{c'}$.\ Then $M_k(c', 0)$ is a vertex $k$-algebra generated by 
 $\omega$, with $Y(\omega, z){=}\omega(z)$.
 \end{thm}
 
 \begin{dfn}\label{dfnqcc} Continuing the discussion in Remark \ref{rmkVirdef}
we call $c'$ the \emph{quasicentral charge},
and we call $c=2c'$ the \emph{central charge} of $M_k(c', 0)$.
\end{dfn}

 We begin the proof of Theorem \ref{thmVirex} by first noting that as an immediate consequence of Lemma
 \ref{lemvirprops}, $\omega(z)$ is a self-local field on $M(c', 0)$.\ Moreover,
 \begin{eqnarray*}
\omega(z).v_0 {=} \sum_n (L(n)v_0+I_{c'})z^{-n-2}{=} \omega+\sum_{n\geq 3}(L(-n)v_0+I_{c'})z^{n-2},
\end{eqnarray*}
so that $\omega(z)$ is creative and creates $\omega$.

\medskip
We are going to apply Theorem \ref{thmexist2}.\ We need
a  sequence of endomorphisms $\underline{D}= (Id, D_1, ...)$ of $M(c', 0)$
satisfying $D_m(v_0){=}0\ (m\geq 1)$, and with respect to which $\omega(z)$ is translation covariant.\
If $k$ is a $\mathbf{Q}$-algebra
we could take $D_m {=} \frac{L(-1)^m}{m!}$, but  this is not defined for general $k$.\ The strategy for proving the Theorem is to first prove it when $k$ is torsion-free, then deduce the
general case by a base-change.

\medskip
The way in which $\underline{D}$ must be defined is dictated by the requirement that
it should be an HS derivation.\ Thus for $k, m\geq 0$ we inductively define
\begin{eqnarray*}
D_0(v_0){:=}v_0, D_m(v_0){:=}0\ (m\geq 1)
\end{eqnarray*}
and
\begin{eqnarray}\label{Dmdef}
&&D_mu'(n_1, ..., n_k){:=}\sum_{i=0}^m (D_i\omega)(-n_1+1)D_{m-i}u'(n_2, ..., n_k)\ \ (m\geq 0),\\
&&\ \ \ \ \ \ \ \  \ \ D_i(\omega)(n){:=}(-1)^i{n\choose i}L(n-i-1)\ \ (i\geq 0, n\in\mathbf{Z}).\notag
\end{eqnarray}
This defines each $D_m$ on the $k$-base of states $\{u'(n_1, ..., n_k)\}$, and we extend
the definition by $k$-linearity to $M(c', 0)$.

\medskip
For example, we have $D_0(\omega){=}D_0(\omega)(-1).v_0{=}L(-2).v_0 {=:}\omega$, and  it follows from
(\ref{Dmdef}) that $D_0=Id$.\  More generally,
\begin{lem}\label{lemD1m} For all $m\geq 0$ we have
\begin{eqnarray*}
L(-1)^m {=}D_1^m{=}m!D_m.
\end{eqnarray*}
\end{lem}
\begin{proof} 
By construction,
\begin{eqnarray*}
&&D_mu'(n_1, ..., n_k) {=} \sum_{i=0}^m (-1)^i{-n_1+1\choose i}L(-n_1-i)D_{m-i}u'(n_2, ..., n_k)\\
 &&\ \ \ \ \ \ \ \ \ \ \ \ \ \ \  \ \ \ \ \ {=} \sum_{i=0}^m {n_1+i-2\choose i}L(-n_1-i)D_{m-i}u'(n_2, ..., n_k),
 \end{eqnarray*}
 so if we set  $D_{[m]}{:=}m!D_m$, then
 \begin{eqnarray*}
D_{[m]}u'(n_1, ..., n_k){=} \sum_{i=0}^m{m\choose i}(n_1+i-2)...(n_1-1)L(-n_1-i)D_{[m-i]}.u'(n_2, ..., n_k).
\end{eqnarray*}

Now in order to show that $L(-1)^m{=}m!D_m$, it suffices to show that $L(-1)^m$ satisfies the same recursive identity as $D_{[m]}$, that is
\begin{eqnarray*}
L(-1)^mu'(n_1, ..., n_k) {=}  \sum_{i=0}^m{m\choose i}(n_1+i-2)...(n_1-1)L(-n_1-i) L(-1)^{m-i}u'(n_2, ..., n_k).
\end{eqnarray*}
To do this, use induction on $m$ to see that the left-hand-side is equal to
\begin{eqnarray*}
&& \sum_{i=0}^{m-1}{m-1\choose i}(n_1+i-2)...(n_1-1)L(-1)L(-n_1-i) L(-1)^{m-1-i}u'(n_2, ..., n_k)\\
=&&\sum_{i=0}^{m-1}{m-1\choose i}(n_1+i-2)...(n_1-1)\\
&&\ \ \ \ \ \left\{(n_1+i-1)L(-n_1-i-1)+L(-n_1-i)L(-1) \right\}L(-1)^{m-1-i}u'(n_2, ..., n_k)\\
=&&\sum_{i=0}^{m-1}{m-1\choose i}(n_1+i-2)...(n_1-1)L(-n_1-i)L(-1)^{m-i}u'(n_2, ..., n_k)+\\
&&\sum_{i=0}^{m-1}{m-1\choose i}(n_1+i-1)...(n_1-1)L(-n_1-i-1)L(-1)^{m-1-i}u'(n_2, ..., n_k)\\
=&&\sum_{i=0}^{m-1}{m-1\choose i}(n_1+i-2)...(n_1-1)L(-n_1-i)L(-1)^{m-i}u'(n_2, ..., n_k)+\\
&&\sum_{j=1}^{m}{m-1\choose j-1}(n_1+j-2)...(n_1-1)L(-n_1-j)L(-1)^{m-j}u'(n_2, ..., n_k)\\
=&&\sum_{i=0}^{m-1}{m-1\choose i}(n_1+i-2)...(n_1-1)L(-n_1-i)L(-1)^{m-i}u'(n_2, ..., n_k)+\\
&&\sum_{i=1}^{m}{m-1\choose i-1}(n_1+i-2)...(n_1-1)L(-n_1-i)L(-1)^{m-i}u'(n_2, ..., n_k)\\
=&&\sum_{i=1}^{m-1}{m\choose i}(n_1+i-2)...(n_1-1)L(-n_1-i)L(-1)^{m-i}u'(n_2, ..., n_k)+\\
&&\ \ L(-n_1)L(-1)^mu'(n_2, ..., n_k)+(n_1+m-2)...(n_1-1)L(-n_1-m)u'(n_2, ..., n_k)\\
=&&\sum_{i=0}^{m}{m\choose i}(n_1+i-2)...(n_1-1)L(-n_1-i)L(-1)^{m-i}u'(n_2, ..., n_k).
\end{eqnarray*}

This  establishes the identity $L(-1)^m{=}m!D_m$.\ In particular, $L(-1){=}D_1$,
whence also $L(-1)^m{=}D_1^m$.\ This competes the proof of the Lemma.
\end{proof}

We can now prove Theorem \ref{thmVirex} in the case when $k$ is
\emph{torsion-free}.\ We have to show that
\begin{eqnarray*}
[D_m, \omega(z)]{=}\sum_{i=1}^m \delta^{(i)}\omega(z)D_{m-i},
\end{eqnarray*}
i.e., for all $n{\in}\mathbf{Z}$, 
\begin{eqnarray*}
[D_m, L(n)] {=} \sum_{i=1}^m (-1)^i {n+1\choose i}L(n-i)D_{m-i}.
\end{eqnarray*}

Because $k$ is torsion-free, it suffices to prove the identity that results upon
multiplying 
each side by $m!$\ Then by Lemma \ref{lemD1m} it suffices to show that
\begin{eqnarray*}
[L(-1)^m, L(n)] &=&
\sum_{i=1}^m (-1)^i (n+1)...(n-i+2){m\choose i}L(n-i)L(-1)^{m-i}.
\end{eqnarray*}
This can be proved by induction on $m$, as follows:
\begin{eqnarray*}
[L(-1)^m, L(n)]&=&L(-1)[L(-1)^{m-1}, L(n)]+[L(-1),L(n)]L(-1)^{m-1}\\
&=&\sum_{i=1}^{m-1} (-1)^i (n+1)...(n-i+2){m-1\choose i}L(-1)L(n-i)L(-1)^{m-i-1}\\
&&+(-1-n)L(n-1)L(-1)^{m-1}\\
&=&\sum_{i=1}^{m-1} (-1)^i (n+1)...(n-i+2){m-1\choose i}\\
&&\left\{[L(-1), L(n-i)]L(-1)^{m-i-1}+L(n-i)L(-1)^{m-i}\right\}\\
&&+(-1-n)L(n-1)L(-1)^{m-1}\\
&=&\sum_{i=1}^{m-1} (-1)^i (n+1)...(n-i+2){m-1\choose i}\\
&&\left\{(-1-n+i)L(n-i-1)L(-1)^{m-i-1}+L(n-i)L(-1)^{m-i}\right\}\\
&&+(-1-n)L(n-1)L(-1)^{m-1}\\
&=&\sum_{i=1}^{m-1} (-1)^i (n+1)...(n-i+2){m-1\choose i}L(n-i)L(-1)^{m-i}\\
&&+\sum_{i=1}^{m} (-1)^i (n+1)...(n-i+2){m-1\choose i-1}L(n-i)L(-1)^{m-i}\\
&=&\sum_{i=1}^m (-1)^i (n+1)...(n-i+2){m\choose i}L(n-i)L(-1)^{m-i}.
\end{eqnarray*}

This completes the proof that $M_k(c', 0)$ is a vertex ring when $k$ is torsion-free.\ By construction, all products in $Vir_k$ are $k$-linear, so that it
is a vertex $k$-algebra.

\medskip
Now suppose that $k$ is an \emph{arbitrary} commutative ring.\ We can find
a torsion-free commutative ring $R$ and a surjective ring morphism
$\psi{:}R{\rightarrow} k$.\ (E.g., take $R{=}\mathbf{Z}[x_a{\mid}a{\in} k] $ with $\psi{:}x_a\mapsto a$.)\ 
Because $R$ is torsion-free, the case of
Theorem \ref{thmVirex} already established shows that $M{:=}M_R(c', 0)$ is
a vertex $R$-algebra for any $c'{\in} R$.

\medskip
Let $I{=}ker\psi$.\ We claim that $IM{=} \{\sum Iu'(n_1, ..., n_k)\mid n_1{\geq} ... {\geq} n_k{\geq} 2\}$ is a $2$-sided ideal
of $M$.\ Indeed, this follows immediately (cf.\ Lemma
\ref{lemideal}) because the operators $L(n)$ are $R$-linear.\ Thus, $M/IM$
carries the structure of a vertex $R$-algebra.\ On the other hand,
the very construction of $M_k(c', 0)$ (as  $Vir_k$-module) shows that it arises as
the base change $k{\otimes}_RM$, where $c'$ is an element of
$R$ that projects onto $c'$.

\medskip
There is an isomorphism of $R$-modules
\begin{eqnarray*}
 M/IM \stackrel{\cong}{\longrightarrow} (R/I){\otimes}_R M, \ m+IM\mapsto 1{\otimes} m\ (m{\in} M),
\end{eqnarray*}
and by transporting the vertex structure of $M/IM$ using this isomorphism, we
obtain the desired vertex $k$-algebra structure on $(R/I)\otimes_R M=k\otimes_RM= M_k(c', 0)$.\
This completes the proof of Theorem \ref{thmVirex}. $\hfill\Box$

\subsection{Virasoro vectors}
Virasoro vectors in vertex rings are ubiquitous and useful.

\begin{dfn}
Suppose that $k$ is a commutative ring, $V$ a vertex $k$-algebra, 
and $c'{\in} k$.\ 
A \emph{Virasoro element (vector) of quasicentral charge $c'$ in $V$} is a state $\omega{\in} V$ such that if 
$Y(\omega, z):= \sum_n L(n)z^{-n-2}$ is the vertex operator for $\omega$,
then the modes $L(n)$ satisfy 
\begin{eqnarray*}
[L(m), L(n)] = (m-n)L(m+n)+\delta_{m+n, 0}{m+1\choose 3}c'Id_V.
\end{eqnarray*}
In other words, the $L(n)$  furnish a representation of the Virasoro Lie $k$-algebra
$Vir_k$ (\ref{Virdef}) on $V$ in which the central element $K$ acts as multiplication by $c'$.\
We call $V$ a \emph{Virasoro vertex $k$-algebra of quasicentral charge $c'$} if 
$V$ is generated by a Virasoro element of quasicentral charge $c'$.
\end{dfn}

The category $k\mathbf{Ver}_{c'}$ of \emph{vertex $k$-algebras of quasicentral charge $c'$} is
defined as follows:\ objects are pairs $(V, \omega)$ where
$V$ is a vertex $k$-algebra and $\omega$ is a Virasoro element of quasicentral charge $c'$ in $V$;
a morphism $\alpha{:} (U, \nu)\rightarrow(V, \omega)$ is a morphism of vertex $k$-algebras
$\alpha{:} U{\rightarrow} V$ such that $\alpha(\nu){=}\omega$.

\begin{thm}\label{thmMkVc} Fix a commutative ring $k$ and an element $c'{\in} k$.\
Then $M_k(c', 0)$ is an \emph{initial object} in $k\mathbf{Ver}_{c'}$.
\end{thm}
\begin{proof} By Theorem \ref{thmVirex} $M_k(c', 0)$ is an object in $k\mathbf{Ver}_{c'}$.

\medskip
  Let $(U, \nu)$ be an object in $k\mathbf{Ver}_{c'}$.\ We have to show that there is
a \emph{unique} morphism of vertex $k$-algebras $\alpha{:}(M(c', 0), \omega) \rightarrow (U, \nu)$.\
Because $\langle\omega\rangle {=} M_k(c', 0)$ we may, and shall, assume without loss that $U{=}\langle\nu\rangle$.

\medskip
First we prove that there is at most one such $\alpha$.\ Write
$Y(\omega, z){=}\sum_n L(n)z^{-n-2}$ and $Y(\nu, z){=}\sum_n L'(n)z^{-n-2}$.\
The construction of $M(c', 0)$ in Subsection \ref{SSVvring} shows that it has a \emph{free} $k$-basis
given by states $L(-n_1).\hdots L(-n_k)v_0$ for $n_1{\geq} \hdots {\geq} n_k{\geq} 2$.\ Furthermore,
the relations satisfied by the modes $L'(n)$ together with the creativity statement $L'(n)\mathbf{1}{=}0\ (n{\geq} {-}1)$
show that $U$ is \emph{spanned} by states $L'(-n_1).\hdots L'(-n_k)\mathbf{1}$ for $n_1{\geq} \hdots {\geq} n_k{\geq} 2$.

\medskip
Now because $\alpha$ is a morphism of vertex $k$-algebras and $\alpha(\omega){=}\nu$ then we have $\alpha L(n)v{=}L'(n)\alpha(v)$ for $v{\in} M(c', 0), n{\in}\mathbf{Z}$.\ It follows that $\alpha(L(-n_1).\hdots L(-n_k)v_0)=L'(-n_1).\hdots L'(-n_k)\mathbf{1}$  for $n_1{\geq} \hdots {\geq} n_k{\geq} 2$, and therefore $\alpha$ is uniquely determined.

\medskip
It remains to show that, so defined, $\alpha$ really \emph{is} a morphism of vertex $k$-algebras.\ But this is clear (in principle) because all of the relations satisfied by the operators in $M(c', 0)$
are consequences of the Virasoro relations (\ref{Virdef}), and they will therefore also hold in $U$.\
Since $\alpha$ merely exchanges $L'(n)$ for $L(n)$ then it is a morphism of vertex $k$-algebras.\ We leave further details to the reader.
\end{proof}

\subsection{Graded vertex rings}
The appropriate notion of $\mathbf{Z}$-grading for vertex algebras over $\mathbf{C}$
is well-known, and carries over unchanged to vertex rings.\ 

\begin{dfn}\label{dfnZgrade} Let $k$ be a commutative ring and $V$ a vertex $k$-algebra.\ We say that
$V$ is \emph{$\mathbf{Z}$-graded} (or simply graded, since we will not consider other kinds of gradings) if there is a decomposition of $V$ into $k$-submodules 
\begin{eqnarray*}
V {=} \oplus_{k\in\mathbf{Z}} V_k
\end{eqnarray*}
with the following property: if $u{\in} V_k, v{\in} V_{\ell}$ and $n{\in}\mathbf{Z}$, then
\begin{eqnarray*}
u(n)v{\in} V_{k+\ell-n-1}.
\end{eqnarray*}
We say that homogeneous states $u\in V_k$ have \emph{weight} $k$, written $wt(u){=}k$.
\end{dfn}

\begin{lem}\label{lemmavacwt} Suppose that $V$ is a graded vertex $k$-algebra.\ Then $\mathbf{1}{\in} V_0$.
\end{lem}
\begin{proof} Write $\mathbf{1} {=} \sum_k \mathbf{1}_k$ where $\mathbf{1}_k{\in} V_k.$\ Then for
$u{\in}V_{\ell}$ we have
\begin{eqnarray*}
u {=}\mathbf{1}(-1)u{=} \left(\sum_k \mathbf{1}_k\right)(-1)u {=} \sum_k \mathbf{1}_k(-1)u,
\end{eqnarray*}
and $wt(\mathbf{1}_k(-1)u){=}k+\ell$.\ Therefore, we must have $\mathbf{1}_k(-1)u{=}0$ whenever
$k{\not=}0$, and since this holds for all homogeneous $u$ then $\mathbf{1}_k(-1){=}0$ for $k\not=0$.\
Therefore $\mathbf{1}_k=\mathbf{1}_k(-1)\mathbf{1}=0$ for $k\not=0$, and consequently
$\mathbf{1}{=}\mathbf{1}_0$, as required.
\end{proof}

\begin{ex} (i)\ Suppose that $V$ is a graded vertex $k$-algebra such that $V{=}V_0$.\ Then
$V$ has a trivial HS derivation and is a commutative ring as in Theorem \ref{thmcommring}.\\
(ii) Suppose that $V_{\ell}{=}0$ for $\ell{<}0$.\ Then $V_0$ is a commutative $k$-algebra with respect to the
$-1$ operation, and $\langle V_0\rangle$ is a vertex $k$-subalgebra of $V$ of the type
prescribed in Theorem \ref{thmcatderring}.
\end{ex}
\begin{proof} Let $u, v{\in} V_0$.\ Then $u(n)v{\in} V_{-n-1}$, and this is $0$ if $n{\geq} 0$ in the
situation of either (i) or (ii).\ Indeed, if $V{=}V_0$ then $u(n)v{=}0$ for $n{\not=}{-}1$, and the assertion of (i) follows.\\
(ii) In this case we can still conclude that $u(r)v(s){=}v(s)u(r)$ for all integers $r, s$ by (\ref{commform}),
so that $\langle V_0\rangle$ consists of sums of states of the form $u{:=} u_1(n_1)\hdots u_k(n_k)\mathbf{1}$
with $u_i{\in} V_0$ and $n_i{<}0$.\ We assert that if $v{:=}v_1(m_1)\hdots v_r(m_r)\mathbf{1}$ is another
such state ($v_i{\in} V_0, m_i<0$) then $u(n)v{=}0$ for $n{\geq} 0$.\ Indeed by (\ref{assocform}) we have
\begin{eqnarray*}
&&u(n)v {=}(u_1(n_1)\hdots u_k(n_k)\mathbf{1})(n)v\\
&&\ \ \ \ \ \ \ \ {=}\sum_{i\geq 0} (-1){n_1\choose i}\left\{u_1(n_1-i)w(n+i)\right\}v
\end{eqnarray*}
where we have set $w=u_2(n_2)\hdots u_k(n_k)\mathbf{1}$, and our assertion follows by induction
on $k$.\ Now we see that the hypotheses, and therefore also the conclusions, of Theorem \ref{thmcatcommring}
hold, and the assertions of (ii) follow easily.
\end{proof}

\begin{lem}\label{Dmwt} Suppose that $V$ is a graded vertex $k$-algebra with canonical HS derivation
$\underline{D}=(D_0, D_1, \hdots)$.\ Then $D_m$ has weight $m$ as an operator, i.e.,
\begin{eqnarray*}
D_m{:} V_k{\rightarrow} V_{k+m}.
\end{eqnarray*}
\end{lem}
\begin{proof} Let $u{\in} V_k$.\ By definition of $D_m$\ (cf.\ Theorem \ref{thmHS}),
$D_m(u) {=}  u(-1-m)\mathbf{1}$.\ By Lemma \ref{lemmavacwt} it then follows that
$wt(D_m(u)) {=} wt(u){+}wt(\mathbf{1}){+}m {=}k{+}m$.
\end{proof}

\begin{ex}\label{Mgraded}  $M{=}M_k(c', 0)$ is a graded vertex $k$-algebra
\begin{eqnarray*}
M = \oplus_{\ell\geq 0} M_{\ell}
\end{eqnarray*}
 where (by definition, and using the notation of Section \ref{SVir})  $M_{\ell}$ is spanned by those states $u'(n_1, \hdots, n_k)$ satisfying $n_1+\hdots+n_k{=}\ell$.\ Furthermore, $L(0)$ leaves each $M_{\ell}$ invariant
 and acts as multiplication by $\ell$ on $M_{\ell}$.
 \begin{proof}
  Recall that
 $M$ is generated as a vertex ring by the state $\omega{=}L(-2)\mathbf{1}$, with $Y(\omega, z){=}\sum_n L(n)z^{-n-2}$, and has a $k$-basis consisting of the states
\begin{eqnarray*}
u'(n_1, \hdots, n_k) {=} L(-n_1)\hdots L(-n_k)\mathbf{1}{=}\omega(1-n_1)\hdots \omega(1-n_k)\mathbf{1}
\end{eqnarray*}
with $n_1{\geq} \hdots {\geq} n_k{\geq} 2$.\ For example,  $\omega{\in} M_2$.

\medskip
First we show that $L(-n)$ has weight $n$ as an operator, i.e.,
$L(-n)u'(n_1,\hdots, n_k)$ has weight $n{+}n_1{+}\hdots{+}n_k$.\ (Because $M$ has no nonzero states of negative weight, this includes the statement  that if $n{+}n_1{+}\hdots{+}n_k<0$
then $L(-n)u'(n_1,\hdots, n_k)=0$.)\ Suppose to begin with that $k{=}0$, so that $u'{=}\mathbf{1}$.\ If also $n{\leq 1}$
then $L(-n)\mathbf{1}{=}0$ by the creation axiom, and there is nothing to prove.\ On the other hand,
if $n{\geq} 2$ then $L(-n)\mathbf{1}{=} \omega(1-n)\mathbf{1}{=} D_{n-2}(\omega)$, and by Lemma \ref{Dmwt} this has weight $wt(\omega)+n-2{=}n$, as required.\ This completes the case $k{=}0$.\ We proceed by
induction on $n_1{+}\hdots{+}n_k$.\ 

\medskip
 If $n{\geq} n_1$ then $L(-n)u'(n_1, \hdots, n_k){=}u'(n, n_1, \hdots, n_k)$ has weight
 $n{+}n_1{+}\hdots{+} n_k$ by definition.\ If $n{<}n_1$ then
\begin{eqnarray}\label{calcLn}
&&L(-n)u'(n_1, \hdots, n_k){=}L(-n)L(-n_1)u'(n_2, \hdots, n_k) \\
&&{=}\left\{(n_1-n)L(-n-n_1){+}L(-n_1)L(-n){+}\delta_{n+n_1, 0}{1-n\choose 3}c'  \right\} u'(n_2, \hdots, n_k),\notag
\end{eqnarray}
and the desired result follows by induction.\ This completes the proof of our assertion that $L(-n)$ has weight $n$ as an operator.

\medskip
Now according to Remark \ref{remgenresprod},
\begin{eqnarray*}
Y(u'(m_1, \hdots, m_r), z) {=} Y(\omega, z)_{1-m_1}(\hdots (Y(\omega, z)_{1-m_r}Id_M)\hdots )
\end{eqnarray*}
is a composition of residue products.\ Using (\ref{resprod}), it follows easily from
the special case previously established that
the $n^{th}$ product $u'(m_1, \hdots, m_r)(n)u'(n_1, \hdots, n_k)$ has weight
$n_1{+}\hdots{+}n_k{+}m_1{+}\hdots{+} m_r{+}n-1$, as required.

\medskip
Finally, we prove the assertions about $L(0)$.\ Indeed,  taking $n{=}0$ in (\ref{calcLn}) yields
\begin{eqnarray*}
&&L(0)u'(n_1, \hdots, n_k) {=} L(0)L(-n_1)u'(n_2, \hdots, n_k) \\
&&\ \ \ \ \ \ \ \ \ \ \ \ \ \ \ \ \ \ \ \ \ \ \ \ {=} \left\{n_1L(-n_1){+}L(-n_1)L(0) \right\} u'(n_2, \hdots, n_k),
\end{eqnarray*}
and the proof that $L(0)u'(n_1, \hdots, n_k) =(n_1{+}\hdots{+}n_k)u'(n_1, \hdots, n_k)$ follows by
induction.
\end{proof}
\end{ex}

\subsection{Vertex operator $k$-algebras}
Our definition is modeled on the definition of a VOA over $\mathbf{C}$ and the example of
$M_{k}(c', 0)$.
\begin{dfn}  Let $(V, \omega)$ be an object in $k\mathbf{Ver}_{c'}$, i.e., a vertex $k$-algebra with  Virasoro element
$\omega$ of quasicentral charge $c'$ and  $Y(\omega,z):=\sum_n L(n)z^{-n-2}$,  and let the canonical HS derivation be $\underline{D}=(Id, D_1, \hdots)$.\ We call $(V, \omega)$ a \emph{vertex operator $k$-algebra of quasicentral charge $c'$} if $V$ carries a $\mathbf{Z}$-grading in the sense of Definition \ref{dfnZgrade}:
\begin{eqnarray*}
V {=} \oplus_{\ell\in\mathbf{Z}} V_{\ell}
\end{eqnarray*}
that satisfies the following additional assumptions:
\begin{eqnarray*}
&& (a)\  V_{\ell}\ \mbox{is a finitely generated $k$-module},\\
&&(b)\  V_{\ell}{=}0\ \mbox{for}\ \ell \ll 0,\\
&&(c)\ L(0)\ \mbox{acts on $V_{\ell}$ as multiplication by $\ell$},\\
&&(d)\ L(-1)^m{=} m!D_m\  (m{\geq} 0).
\end{eqnarray*}
\end{dfn}

\begin{ex} (i)\ $M_k(c', 0)$ (and any of its graded quotients) are vertex operator $k$-algebras of quasi-central charge $c'$.\
This follows from the construction of $M_k(c', 0)$, Lemma \ref{lemD1m}, Theorem \ref{thmMkVc} and Example \ref{Mgraded}.\\
(ii)\ A vertex operator algebra over $\mathbf{C}$ of central charge $c$ is a vertex operator 
$\mathbf{C}$-algebra of quasicentral charge $2c$.\\
(iii)\ If $V$ is a  vertex operator $\mathbf{Z}$-algebra then the base change
$k{\otimes} V$ is a vertex operator $k$-algebra.
\end{ex}

\begin{lem} Suppose that $(V, \omega)$ is a vertex operator $k$-algebra.\ Then
$\omega{\in} V_2$.
\end{lem}
\begin{proof} By creativity we have $\omega {=} L(-2)\mathbf{1}$.\ Then $L(0)\mathbf{1}{=}0$ by Lemma \ref{lemmavacwt} and hypothesis (c), 
and 
 $L(0)\omega{=}L(0)L(-2)\mathbf{1}
{=}([L(0), L(-2)]{+}L(-2)L(0))\mathbf{1}{=}2L(-2)\mathbf{1}{=}2\omega$, as required.
\end{proof}

\medskip
\begin{center}{\bf Part II: Pierce bundles of Vertex Rings}
\end{center}

The second part of this paper revolves around the idea of a \emph{Pierce bundle of vertex rings}, which 
is a special kind of \emph{\'{e}tale bundle}.\ It transpires that Pierce's theory 
for commutative rings \cite{RSP} may be carried over \emph{en bloc} to the setting of vertex rings.

\section{\'{E}tale bundles of vertex rings}\label{Setale}

\subsection{Basic definitions}\label{SSbasicbundledefs}
 In this Subsection we give the basic definitions concerning \'{e}tale bundles of vertex rings
$p{:}E{\rightarrow}X$, regardless of whether they are Pierce bundles or not.\ 
 We refer the reader to the relevant sections of \cite{Mac} and  \cite{MacM} for additional background
on bundles and categories.

\medskip
An \'{e}tale bundle in the category $\bf{Top}$ of topological spaces and continuous maps
is a \emph{surjective} morphism $p{:}E{\rightarrow} X$ that is a \emph{local homeomorphism}
in the sense that $E$ is covered by open sets $U_i$ such that the restriction $p|U_i$
of $p$ to each $U_i$ is a homeomorphisms whose image $p(U_i)$ is open in $X$.

\medskip
The \emph{stalks} of the bundle are defined by $E_x{:=}p^{-1}(x)\ (x{\in} X)$.

\medskip
An \'{e}tale bundle of abelian groups is an \'{e}tale bundle
$p{:}E{\rightarrow}X$ such that each stalk $E_x$ has the structure of an additive abelian group.\
Moreover, the operations in the stalks are \emph{continuous} in the following sense{:}\
if $\Delta{:=}\{(t, u){\in}E{\times}E{\mid}p(t)=p(u)\}$ then the map
$\mu{:}\Delta{\rightarrow}E, (t, u)\mapsto t{-}u$ is \emph{continuous}.\ The relevant diagram is
\begin{eqnarray*}
  \xymatrix{
 E &\Delta\ar[l]_{\mu} \ar[d]_{\pi_1}\ar[r]^{\pi_2} & E\ar[d]^p\\
&E \ar[r]_p  & X 
 }
\end{eqnarray*}
where the square is a pullback in $\bf{Top}$.\ In particular, $\pi_i$ is projection on the $i^{th}$ coordinate,
$E{\times}E$ has the product topology and $\Delta{\subseteq}E{\times}E$ the subspace topology.\
 Note that $(t, u){\in}\Delta$ if, and only if, $t, u$ lie in the 
\emph{same stalk} $E_{p(t)}$, in which case $t{-}u$ is the group operation in $E_{p(t)}$.

\medskip
We set $0_x\ (x{\in}X)$ as the \emph{zero element} in $E_x$.\ It follows from the definitions that
the map $0_X{:}X\rightarrow E,\ x\mapsto 0_x$ is \emph{continuous}.

\medskip
\begin{dfn} An \'{e}tale bundle of vertex rings is an \'{e}tale bundle of abelian groups
$p{:}E{\rightarrow}X$ such that each stalk $E_x$ has the structure of a vertex ring,
and the operations $u(n)v\ (u, v{\in} E, n{\in} \ZZ)$ are continuous in the same sense as before.\ Furthermore,
if $\mathbf{1}_x$ is the vacuum element of $E_x$ then the map $\mathbf{1}_X{:}X\rightarrow E, x\mapsto \mathbf{1}_x$
must be continuous.
\end{dfn}

\begin{ex}(a)\ If each $E_x$ is a commutative ring regarded as a vertex ring with trivial canonical HS
derivation, the last definition reduces to the usual definition of an \'{e}tale bundle of commutative rings.\\
(b)\ Each $E_x$ is a \emph{discrete} subspace of $E$.\ On the other hand, if $V$ is any vertex ring equipped with the
discrete topology and $X{\times}V$ has the product topology, then the canonical projection
$p{:}X{\times}V{\rightarrow}X$ is an \'{e}tale bundle of vertex rings with constant  stalk $V$.

\end{ex}

\subsection{Nonassociative vertex rings and sections}

Let $f{:}Y{\rightarrow}X$ be continuous.\ The \emph{inverse image} of an \'{e}tale bundle
$p{:}E{\rightarrow} X$ of (say) abelian groups or commutative rings, 
with respect to $f$, denoted $f^*(E)$, is the pullback
  \begin{eqnarray}\label{Ediag} 
  \xymatrix{
f^*(E)\ar[d]_{\pi_1}\ar[r]^{\pi_2} & E\ar[d]^p\\
Y\ar[r]_f  & X  
 }
\end{eqnarray}
Thus $f^*(E){:=}\{(y, e){\in}Y{\times E}{\mid} f(y){=}p(e)\}$ with coordinate projections $\pi_i$
as before.

\medskip 
It is a standard fact  that $\pi_1{:}f^*(E){\rightarrow}Y$ is again an \'{e}tale bundle of abelian groups or commutative rings.\ However, if $p$ is a bundle of vertex rings then $f^*(E){\rightarrow}Y$ may not b.\ We discuss this in further detail below, in the context of local sections.\ Note that the $n^{th}$ product
on stalks $f^*(E)_y$ are (like the other operations) defined in the natural way, i.e.,
\begin{eqnarray*}
(y, t)(n)(y, u)):=(y, t(n)u)\ \ (t, u{\in}E, n\in\ZZ).
\end{eqnarray*}
We record 
\begin{lem}\label{lempullback} If $p{:}E{\rightarrow}X$ is an \'{e}tale bundle of vertex rings and $f{:}Y{\rightarrow}X$ a continuous map, then
the pullback $f^*(E)$ is an \'{e}tale bundle of abelian groups. $\hfill \Box$
\end{lem}

Let $p{:}E{\rightarrow} X$ be an \'{e}tale bundle with $U{\subseteq}X$ any subset.\
The set of \emph{local sections} over $U$ is defined by
\begin{eqnarray*}
\Gamma(U, E){:=}\{\sigma{:}U{\rightarrow} E{\mid}\sigma\ \mbox{is continuous and}\ p\circ\sigma{=}Id_U\}.
\end{eqnarray*}
The set of \emph{global sections} is $\Gamma(X, E)$.\ If $i{:}U{\rightarrow}X$ is insertion, then
$\Gamma(U, X)$ is isomorphic (as a set, or abelian group) to $\Gamma(U, i^*(E))$.

\begin{ex} The \emph{zero section} $0_X$ and \emph{vacuum section} $\mathbf{1}_X$
both lie in $\Gamma(X, E)$.
\end{ex}

If $p{:}E{\rightarrow}X$ is a bundle of  vertex rings, the algebraic operations enjoyed by the stalks
may be transported to sections in a pointwise fashion.\ Explicitly, the $n^{th}$
product of sections $\sigma, \tau{\in}\Gamma(U, E)$ is defined by
\begin{eqnarray*}
(\sigma(n)\tau)(u):=\sigma(u)(n)\tau(u)\ \ (u{\in}U).
\end{eqnarray*}
This is meaningful inasmuch as $\sigma(u), \tau(u){\in}E_u$, so that their $n^{th}$ product is defined.\ The group operations are defined similarly.

\medskip
It is immediate that $\Gamma_U{:=}\Gamma(U, E)$ is an additive abelian group.\ 
 It is equipped with $n^{th}$ products, so we may define formal series $Y(\sigma, z)\in End(\Gamma_U)[[z,z^{-1}]]$ in the expected way:
\begin{eqnarray*}
Y'(\sigma, z)\tau{:=}\sum_{n\in \ZZ}\sigma(n)\tau z^{-n-1}.
\end{eqnarray*}

The $n^{th}$ products on $\Gamma_U$ being defined in a pointwise fashion, some of their basic properties
carry over unchanged from the corresponding property in the vertex rings that comprise the stalks
of the bundle.\ Thus we have 
a vacuum element $\mathbf{1}_U{:}u \mapsto \mathbf{1}_u\ (u{\in}U), Y'(\mathbf{1}_U, z){=}Id_{\Gamma_U}$,
and the creativity formula $Y'(\sigma, z)\mathbf{1}_U=\sigma+\hdots$ also holds.

\medskip
The sequence
of endomorphisms $\underline{D}{:=}(Id, D_1, \hdots)$ defined, as in Theorem \ref{thmHS}, 
by $D_m(\sigma){:=}\sigma(-m-1)\mathbf{1}_U$ is iterative and
HS with respect to all products, moreover $D_m(\mathbf{1}_U){=}0_U\ (m{\geq}1)$.\ Additionally, each $Y'(\sigma, z)$
is translation-covariant with respect to $\underline{D}$.

\medskip
Motivated by Theorem \ref{thmexist3} and the preceding remarks,
we make the
\begin{dfn}
A \emph{nonassociative vertex ring} $(V, Y', v_0, \underline{D})$ consists of an additive abelian group $V$,
a state $v_0\in V$, an \emph{iterative} sequence of endomorphisms
$\underline{D}:=(Id_V, D_1, ...)$ in $End(V)$ satisfying $D_m(v_0){=}0$ for $m\geq 1$, and a morphism of abelian groups
\begin{eqnarray*}
Y'{:} V\rightarrow End(V)[[z, z^{-1}]],\ u\mapsto Y'(u, z):=\sum_n u(n)z^{-n-1},
\end{eqnarray*}
satisfying for all $u, v{\in}V${:}
\begin{eqnarray*}
&&\ \ \ \ \ \  Y'(u, z)v_0\in u+zV[[z]], \\
&& {[}D_m, Y'(u, z){]} =\sum_{i=1}^m \delta_z^{(i)}Y'(u, z)D_{m-i}\ \ (m\geq 0).
\end{eqnarray*}
\end{dfn}

We might have incorporated more of the properties enjoyed by $\Gamma_U$ in this Definition,
for example the HS property of $\underline{D}$.\ It is arguably more
natural to also assume that each $Y'(u, z){\in}\mathcal{F}(V)$.\  
As it is, we can paraphrase Theorem \ref{thmexist3} as follows{:} 

\begin{lem}\label{lemnearvringchar} A nonassociative vertex ring $V$ is a vertex ring
if, and only if, the locality condition $Y'(u, z){\sim}Y'(v, z)$ holds for all $u, v{\in}V$.\ $\hfill \Box$
\end{lem}

 We encapsulate the earlier discussion in the following
\begin{lem}\label{lemgammanearvring} Suppose that $p{:}E{\rightarrow}X$ is an \'{e}tale bundle of vertex rings.\ Then
for any subset $U{\subseteq}X$, $\Gamma_U$ is a nonassociative vertex ring.
$\hfill \Box$
\end{lem}

\medskip
The question of whether $\Gamma_U$ is in fact a
vertex ring is murkier.\ At issue is the locality property.\ Certainly
for $u{\in}U$ we have $Y'(\sigma, z)(u){\sim_t}Y'(\tau, z)(u)$ for $t{\gg}0$, 
however the choice $t$ depends on $u$, whereas we need
a $t$ that works uniformly for $u{\in}U$.\
We state a standard fact  used repeatedly in what follows.
\begin{lem}\label{lemzsects} Let $p{:}E{\rightarrow}X$ be an \'{e}tale bundle (of sets) with $U{\subseteq}X$ open and
$\sigma, \tau{\in}\Gamma(U, E)$.\ Then the set $\{x{\in}U{\mid}\sigma(x){=}\tau(x)\}$ is \emph{open} in $X$.\
In particular (taking $\tau{=}0_U$), the zero set of $\sigma$, i.e., $\{x{\in}U{\mid}\sigma(x){=}0_x\}$ is \emph{open}.\ 
$\hfill \Box$
\end{lem}

\begin{thm}\label{thmpresheaf} Let $p{:}E{\rightarrow}X$ be an \'{e}tale bundle of vertex rings
over a \emph{compact} space $X$.\ Then $\Gamma_Y$ is a vertex ring if $Y{\subseteq} X$ is \emph{closed}.
\end{thm}
\begin{proof} First consider the case $Y{=}X$, and let $\sigma, \tau{\in}\Gamma_X$.\ By Lemmas
\ref{lemnearvringchar} and  \ref{lemgammanearvring} it suffices to show that
$Y'(\sigma, z){\sim}Y'(\tau, z)$.

\medskip
Let $x{\in}X$.\ By locality in the stalk $E_x$, there is a nonnegative integer $t_x$ such that
\begin{eqnarray}\label{locformsect}
(z-w)^{t_x}[Y'(\sigma, z), Y'(\tau, w)](x){=}0_x.
\end{eqnarray}
By Lemma \ref{lemzsects} (\ref{locformsect}) continues to hold throughout
an open neighborhood $U_x{\subseteq}U$ of $x$  for all integers  ${\geq}t_x$.\
As $x$ ranges over $X$ we obtain an open cover $\{U_x\}$ of $X$, and by compactness
there is a finite subcover $X{=}\cup_{i=1}^n U_{x_i}$.\ Then if we choose $t$ to be the
\emph{maximum} of  $t_{x_1}, \hdots, t_{x_n}$ then  
$Y'(\sigma, z){\sim}_{t}Y'(\tau, z)$ holds throughout $X$, as required. 
  
  \medskip
  Finally, if $i{:}Y{\rightarrow}X$ is a closed subset then $Y$ is compact, and we have
  the pullback bundle $i^*(E){\rightarrow}Y$ (cf.\ Lemma \ref{lempullback}) which \emph{is} a bundle of vertex rings
  thanks to compactness.\ 
  Now (b) follows from the special case $Y{=}X$. 
\end{proof}

\section{Pierce bundles of vertex rings}
 So far, we have not discussed any interesting examples of \'{e}tale bundles of vertex rings.\ We will rectify this situation
 in the present Section by constructing \emph{Pierce bundles}.\ We follow \cite{RSP} closely.
 
  \subsection{The Stone space of a vertex ring}\label{SSBV}
 Fix a vertex ring $V$ with center $C{=}C(V)$.\
  Let $B{=}B(V){=}B(C)$ be the set of idempotents
 in $V$.\ By Lemma \ref{lemmaidem} these are precisely the idempotents of $C$.\ By a standard procedure
 we can equip $B$ with the structure of a \emph{Boolean ring} (cf.\  \cite{RSP} or \cite{Jac}, Chapter 8).\
 This means that
 $B$ is a commutative ring in which every element is idempotent.\ Indeed, 
 multiplication in $B$  is that of $C$ and addition $\oplus$ is defined by
 \begin{eqnarray*}
e{\oplus}f{:=} e{+}f{-}2ef\ \ (e, f{\in}B).
\end{eqnarray*}

$B$ is also a \emph{Boolean algebra} (complemented, distributive lattice with $0, 1$) if we define 
$e\vee f{:=}e{+}f{-}ef, e\wedge f{:=} ef$ and $e'{:=}\mathbf{1}-e$.\ This fact will not play much of
a r\^{o}le in what follows.

\begin{dfn}\label{dfnSS} The \emph{Stone space} of $V$ is the topological space $X{=}Spec(B)$.
\end{dfn}

Here, and below, $Spec(R)$ for a commutative ring $R$ denotes the set of prime ideals of
$R$ equipped with its Zariski topology.\ 
As is well-known, the Boolean property of $B$  implies that $X$ has additional topological properties beyond those that hold in
a general prime spectrum.\ We state some of the main properties of $X{=}Spec(B)$ that we use.\
 
\medskip
 It is easy to see that all prime ideals of $B$ are maximal, so that the points of $X$ are the maximal ideals of $B$.\ Moreover
$X$ is a \emph{Boolean space}, i.e., it is a
\emph{totally disconnected, compact Hausdorff space}.\ (Totally disconnected means that the  connected components of $X$ are single points.)\ For further details, see (\ref{VNRequivs}).
 
 \medskip
 A \emph{basis} for the Zariski topology on $X$ consists of the sets
 \begin{eqnarray*}
N(e){:=}\{M{\in}X{\mid} e{\notin}M\}\ \ (e{\in}B).
\end{eqnarray*}
Note that $N(e'){=}N(e)'$ is the \emph{complement} of $N(e)$ in $X$, so that 
$N(e)$ is a \emph{clopen} (closed and open) subset in $X$.\ Indeed, the sets $\{N(e){\mid}e{\in}B\}$
comprise \emph{all} of the clopen subsets of $X$.\ Because $X$ is Hausdorff, clopen
is the same as compact and open.\ We have in addition that
\begin{eqnarray}\label{Nelattice}
N(e){\cup} N(f) = N(e\vee f),\ \ N(e){\cap} N(f) = N(e{\wedge} f).
\end{eqnarray}

\medskip
The map $e{\mapsto} N(e)$, which induces 
a \emph{lattice isomorphism} between $B$ and clopen subsets of $B$ thanks to (\ref{Nelattice}), is
 the \emph{Stone duality}.\ This explains our nomenclature
in Definition \ref{dfnSS}.

\medskip
The following property of $X$ is very powerful.
\begin{lem}\label{lemmapp}(Partition property) Suppose that $\{X_i\}$ is an open cover of $X$.\
Then there are clopen sets $N_1, \hdots, N_k$ such that each
$N_i$ is contained in some $X_j$ and the $N_i$ \emph{partition} $X$ in the sense that
$\cup_i N_i{=}X$ and $N_i{\cap} N_j{=}\phi\ (i{\not=}j)$.
\end{lem}

In terms of idempotents, the partition property means that if $N_i{=}N(e_i)\ (e_i{\in}B)$, then
\begin{eqnarray*}
e_ie_j{=}\delta_{ij}e_i,\ \ \ \sum_i e_i {=} \mathbf{1}.
\end{eqnarray*}

\medskip
The following Example is telling.\

\begin{ex}\label{exhomeo} Recall that an associative ring $R$ is  \emph{von Neumann regular} 
if, for every $x{\in}R$, there is $y{\in}R$ such that $xyx{=}x$.\ If $R$ is a commutative von Neumann regular ring then there is a \emph{homeomorphism} $Spec(R){\rightarrow} Spec(B(R)), P{\mapsto} P\cap B(R)$. (For further discussion about von Neumann regular rings and vertex rings, see Subsection \ref{SvNr}.)
\end{ex}
\begin{proof} See \cite{Good}, Theorem 8.25.
\end{proof}

\subsection{The Pierce bundle of a vertex ring}\label{SSPiercebun}
Let $V$ be a vertex ring with Stone space $X{=}Spec(B(V))$.\ Following Pierce \cite{RSP}, we  construct an \'{e}tale bundle 
$\mathcal{R}{\rightarrow}X$ of vertex rings associated to this data.

\medskip
For each  $M{\in}X$ set
\begin{eqnarray}\label{Mbardef}
\overline{M}{:=}\cup_{e{\in}B\cap M} e(-1)V.
\end{eqnarray}
\begin{lem}\label{lembarM}  $\overline{M}$ is a 2-sided ideal in $V$.
\end{lem} 
\begin{proof} For an idempotent  $e{\in}M{\in}B$ we know that $e(-1)V$ is a 2-sided ideal of
$V$ (Example \ref{exvac}(ii)), hence so is $\sum_{e\in M} e(-1)V$.\ The main point of the Lemma is that this sum coincides with the displayed \emph{union}.\ This follows from the identity $(e{\oplus}f)e {=} e$
in $B$, which implies that
if $u, v{\in} V$ then
\begin{eqnarray*}
e(-1)u{+}f(-1)v = (e{\oplus}f)(-1)(e(-1)u{+}f(-1)v).
\end{eqnarray*}
The Lemma follows.
\end{proof}

By Lemma \ref{lembarM}, each $V/\overline{M}$ is a vertex ring, and we let $0_{\overline{M}}, \mathbf{1}_{\overline{M}}$
denote the zero element and vacuum element respectively.\
Introduce the \emph{disjoint union}
\begin{eqnarray*}
\mathcal{R}{:=} \bigcup_{M\in B} V/\overline{M},
\end{eqnarray*}
and define
\begin{eqnarray*}
\pi{:}\mathcal{R}{\rightarrow}X,\ \ v{+}\overline{M}{\mapsto}M.
\end{eqnarray*}

The result we are after is
\begin{thm}\label{thmVXB} Let $V$ be a vertex ring with Stone space $X$.\ Then
$\pi{:}\mathcal{R}{\rightarrow}X$ is an \'{e}tale bundle of vertex rings.
\end{thm}

\medskip
First we introduce a section $\sigma_v$ of the set map $\pi$ for each $v{\in}V$.\ These will be instrumental in everything that follows, and are defined in the natural way, namely
\begin{eqnarray}\label{sectiondef}
\sigma_v{:}X{\rightarrow}\mathcal{R},\ \ M{\mapsto}v{+}\overline{M}.
\end{eqnarray}

\medskip
We may, and shall, topologize $\mathcal{R}$ by taking the family of sets $\{\sigma_v(N(e))\}$ for $v{\in}V$ and $e{\in}B$ 
as a \emph{basis}.\ The justification for this is the same as the case of commutative rings
(\cite{RSP}, pp\ 16-17).\ We give details to highlight some common techniques and results that we use.

\begin{lem}\label{lemsigmaid} Suppose that $u, v{\in}V$  satisfy $\sigma_u(M){=}\sigma_v(M)$ for some 
$M{\in}X$.\ Then there is $e\in B{\setminus}{M}$ such that $\sigma_u(L){=}\sigma_v(L)$ for all $L{\in}N(e)$.\ 
\end{lem}
 \begin{proof} We have $0_{\overline{M}}=\sigma_u(M){-}\sigma_v(M){=}(u{-}v){+}\overline{M}$.\ So
 $u{-}v\in\overline{M}\Rightarrow u{-}v=f(-1)w\ (f{\in} B\cap M, w{\in} V)$, and we take 
 $e{:=}\mathbf{1}{-}f\in B{\setminus}{M}$.\
 Then if $L{\in}N(e)$ we have $f{\in}L$ and we obtain
 $\sigma_u(L){=}u{+}\overline{L}=v{+}f(-1)w{+}\overline{L}=v{+}\overline{L}{=}\sigma_v(L)$.
 \end{proof}

Now consider idempotents $e, f{\in}B$ such that $M{\in}N(e){\cap} N(f)$, and suppose
further that $u, v{\in}V$ satisfy $\sigma_u(M){=}\sigma_v(M)$.\ By Lemma \ref{lemsigmaid},
$\sigma_u, \sigma_v$ agree on a clopen neighborhood $N(g)$ of $M$ for some $g{\in}B$.\ But we have
$M{\in}N(efg){\subseteq}N(e)\cap N(f)\cap N(g)$, so that
$\sigma_u(M){\in}\sigma_u(N(e))\cap\sigma_v(N(f))$.\ This suffices to establish our assertion
about the topology on $\mathcal{R}$.

\medskip
Now it is clear that $\pi$ is a continuous surjection onto $X$ and a local homeomorphism.\ Indeed, $\sigma_v(N(e))$ is mapped homeomorphically onto the open set $N(e)$.\ Thus the bundle $\pi{:}\mathcal{R}{\rightarrow}X $ is 
\emph{\'{e}tale}.\

\begin{ex} $\sigma_v{\in}\Gamma(X, \mathcal{R})$ for all $v{\in}V$.\ 
$\sigma_{0}{:}M\mapsto 0_{\overline{M}}, \sigma_{\mathbf{1}}{:}M\mapsto \mathbf{1}_{\overline{M}}\ (M{\in}X)$ are the \emph{zero} and \emph{vacuum} sections respectively.
\end{ex}
\begin{proof} We need the \emph{continuity} of $\sigma_v$.\ But $N(\mathbf{1})=X$, so
$\sigma_v$ is the inverse of the restriction of $\pi$ to $\sigma_v(X)$, which is a homeomorphism.
\end{proof}

\medskip
The stalks of the bundle $\pi{:}\mathcal{R}{\rightarrow}X$ are the vertex rings $\mathcal{R}_M{=}V/\overline{M}$.\
So we can define $n^{th}$ products on sections in the usual way.\ In particular, if $u, v{\in}V$
and $M{\in}X$ we have 
\begin{eqnarray*}(\sigma_u(n)\sigma_v)(M){=}\sigma_u(M)(n)\sigma_v(M){=}(u{+}\overline{M})(n)(v{+}\overline{M})
{=}u(n)v{+}\overline{M}{=}\sigma_{u(n)v}(M).
\end{eqnarray*} 
This shows that
\begin{eqnarray*}
\sigma_u(n)\sigma_v{=}\sigma_{u(n)v}.
\end{eqnarray*}

 Now we can show that $\pi$ is a bundle of vertex rings for which we require the continuity
 of the various algebraic operations (cf.\ Subsection \ref{SSbasicbundledefs}).\ Precisely, set 
 \begin{eqnarray*}
\Delta{:=}\{(a, b){\in}\mathcal{R}{\times}\mathcal{R}{\mid}\pi(a){=}\pi(b)\}.
\end{eqnarray*}
We have to show that the maps $\Delta{\rightarrow}\mathcal{R}, (a, b){\mapsto}a\pm b, a(n)b$
are continuous.

\medskip
For example, fix $(a, b){\in}\Delta, n{\in}\mathbf{Z}$.\ There is $t{\in}V$ such that $a(n)b{\in}\sigma_t(X)$.\  
So there are $u, v{\in}V, M{\in}X$ such that
$a{=}u{+}\overline{M}, b{=}v{+}\overline{M}$, and $a(n)b{=}t{+}\overline{M}$.\ Now we have 
$\sigma_t(M){=}\sigma_{u(n)v}(M)$, so (Lemma \ref{lemsigmaid}) 
$\sigma_t$ and $\sigma_{u(n)v}$ agree on a basic open neighborhood $N(e)$ of $M$.\ 
 Write $\mu_n{:}\Delta{\rightarrow}\mathcal{R}$ for the $n^{th}$ product.\
So $\sigma_t(N(e))$ is a basic open neighborhood of $a(n)b$, and
\begin{eqnarray*}
&&\mu_n(\sigma_{u}(N(e)){\times}\sigma_{v}(N(e))\cap \Delta)
= \bigcup_{M\in N(e)} u(n)v{+}\overline{M} {=} \sigma_{u(n)v}(N(e)){=}\sigma_t(N(e)).
\end{eqnarray*}
The continuity of $\mu_n$ follows.\ The other operations are treated similarly, 
and the proof of Theorem \ref{thmVXB} is complete. $\hfill\Box$

\begin{dfn} The \'{e}tale bundle $\pi{:}\mathcal{R}{\rightarrow}X$ associated to the
vertex ring $V$ is the \emph{Pierce bundle} associated to $V$.
\end{dfn}

\subsection{Some local sections}
We continue with previous notation, in particular $V$ is a vertex ring with Stone space $X$
and associated Pierce bundle $\pi{:}\mathcal{R}{\rightarrow}X$ (Theorem \ref{thmVXB}).
According to Theorem \ref{thmpresheaf}  the local sections
$\Gamma_Y$ for closed $Y{\subseteq}X$ carry the structure of a vertex ring.\ 
We shall explicitly realize these vertex rings as quotients of $V$.\ The main result is the case when $Y{=}X$, and may be stated as follows.
\begin{thm}\label{thmglobaliso} There is an \emph{isomorphism of vertex rings}
\begin{eqnarray*}
\xi{:}V{\stackrel{\cong}{\longrightarrow}}\Gamma(X, \mathcal{R}), \ v{\mapsto}\sigma_v.
\end{eqnarray*}
\end{thm}
\begin{proof} The main issue is to show that $\xi$ is \emph{surjective},
and we do this first.\ Pick any $\alpha{\in}\Gamma(X, \mathcal{R})$ 
and $M{\in}X$.\ Then $\alpha(M){\in}\mathcal{R}_M$,
so $\alpha(M){=}a{+}\overline{M}$ for some $a{\in}V$.\ The sections
$\alpha, \sigma_a$ agree at $M$, so by Lemma \ref{lemzsects} they
 agree on a basic open neighborhood of $M$.\ As $M$ ranges over all points of $X$ we obtain an open cover of $X$,
and on each open set in the cover $\alpha$ agrees with some $\sigma_a\ (a{\in} V)$.

\medskip
By Lemma \ref{lemmapp} there is a partition  $X=N(e_1) \cup ...\cup N(e_r)$ into clopen sets  
($e_j{\in} B$) such that on each $N(e_j)$ $\alpha$ agrees with some $\sigma_{a_j} (a_j{\in} V)$.\
Set 
\begin{eqnarray*}
b{:=}\sum_j e_j(-1)a_j.
\end{eqnarray*}
We will show that $\alpha{=}\sigma_b$.\
Let $L{\in}X$.\ Then there is a unique index $i$ such that $L{\in}N(e_i)$,
and we have to show that $\alpha(L){=}\sigma_b(L)$.

\medskip
The partition property implies that $\sum_j e_j{=}\mathbf{1}$, so that $a_i{=}\sum_j e_j(-1)a_i$.\
Now observe that  $i{\not=} j {\Rightarrow} L{\notin} N(e_j){\Rightarrow}e_j{\in}L
{\Rightarrow}e_j(-1)a_i{\in}\overline{L}$.\ It follows that $a_i{\equiv} e_i(-1)a_i$ (mod $\overline{L}$),
and therefore $\sigma_{a_i}(L){=}\sigma_{e_i(-1)a_i}(L).$\ Furthermore we
have $\sigma_b(L){=}\sigma_{e_i(-1)a_i}(L)$.\

\medskip
On the other hand because $\alpha$ agrees with $\sigma_{a_i}$ on $N(e_i)$ then 
$\alpha(L) {=}\sigma_{a_i}(L)$.\ Therefore, $\alpha(L){=}\sigma_b(L)$
follows from the previous paragraph, and the surjectivity of $\xi$ is established.

\medskip
The field property $\sigma(n)\tau{=}0$ for $n{\gg}0$ ($\sigma, \tau{\in}\Gamma(X, \mathcal{R})$) now follows
because $\sigma{=}\sigma_u$ and $\tau{=}\sigma_v$ for some $u, v{\in}V$, and
$\sigma_u(n)\sigma_v{=}\sigma_{u(n)v}{=}0$ for $n{\gg}0$.\ The Jacobi identity and vacuum identities
follow because they hold in each stalk, and the first assertion of the Theorem,
 that $\Gamma(X, \mathcal{R})$ is a vertex ring, is proved.
 
 \medskip
 $\xi$ is a morphism of vertex rings because
\begin{eqnarray*}
\xi(u(n)v){=}\sigma_{u(n)v} {=} \sigma_u(n)\sigma_v {=} \xi(u)(n)\xi(v)
\end{eqnarray*}
and $\xi(\mathbf{1}){=}\sigma_{\mathbf{1}}$.\
Finally, if $\sigma_v{\in} ker\xi$ then $v{\in} \cap_{M\in X} \overline{M}$, and this intersection is $0$
by Proposition 1.7 of \cite{RSP}.\ This completes the proof of the Theorem.
\end{proof}

\medskip
To describe $\Gamma_Y$ for an arbitrary closed subset $Y{\subseteq} X$ we introduce some additional notation.

\medskip\noindent
The \emph{support} of $\sigma{\in}\Gamma_X$ is
$supp(\sigma){:=} \{M{\in} X{\mid}\sigma(M){\not=}0\}$.\\
For $U{\subseteq} X$, $J[U]{:=}\{\sigma{\in} \Gamma_X{\mid} supp(\sigma){\subseteq} U\}$.\\
For $J{\subseteq} \Gamma_X$, $U[J]{:=} \cup_{\sigma{\in} J} supp(\sigma)$.

\begin{lem}\label{lemsesJ} Let $Y{\subseteq} X$ be closed with complement $Y'$.\
Then there is a short exact sequence 
\begin{eqnarray*}
0 \longrightarrow J[Y']\longrightarrow \Gamma(X, \mathcal{R})\stackrel{res}{\longrightarrow} \Gamma(Y, \mathcal{R}) \longrightarrow 0
\end{eqnarray*}
($res$ is \emph{restriction} from $X$ to $Y$.)
\end{lem}
\begin{proof} To show that $res$ is \emph{surjective}, let $\alpha{\in}\Gamma_Y$.\
If $M{\in}Y$ then $\alpha(M)=a{+}\overline{M}$ for some $a{\in}V$, so $\sigma_a$ and $\alpha$ agree on a basic open neighborhood of $M$.\ If $M{\notin}Y$, choose a basic open neighborhood of
$M$ whose intersection with $Y$ is empty.\ As $M$ ranges over $X$, we obtain in this way
an open cover $\{N_i\}$ of $X$
such that either $\alpha$ agrees with some $\sigma_{a_i} (a_i{\in}V)$ on $N_i$ or else $Y\cap N_i{=}\phi$.\
There is a partition $X{=}N(e_1)\cup \hdots \cup N(e_k)$ such that each $N(e_j)$ is contained in
some $N_i$.\ Define $\sigma{:}X{\rightarrow}\mathcal{R}$ as follows:\
if $N(e_i)\cap Y{=}\phi$ then $\sigma|N(e_i){=}0$, and otherwise 
$\sigma|N(e_i){=}\sigma_{a_i}|N(e_i){=}\alpha|N(e_i)$.\ 
Because the $N(e_i)$s partition $X$ and the restriction of $\sigma$ to $N(e_i)$ belongs to
$\Gamma(N(e_i), \mathcal{R})$, then $\sigma{\in}\Gamma(X, \mathcal{R})$.\ Moreover by construction we
have $\sigma|Y=\alpha$.\ This completes the proof of the surjectivity of $res$.

\medskip
Clearly  $\sigma{\in} ker\ res \Leftrightarrow supp(\sigma)\subseteq Y' \Leftrightarrow \sigma\in J[Y']$, and the Lemma is proved.
\end{proof}

\medskip
\begin{ex}\label{exses} If $M{\in}X$ then $\{M\}$ is closed because $X$ is Hausdorff, and there is an isomorphism of short exact sequences, where $\eta{:}\sigma \mapsto \sigma(M)$,
\[\xymatrix{
0\ar[r] & J[X\setminus{\{M\}}]\ar[r]& \Gamma(X, \mathcal{R})\ar[r]^{res}\ar[d]_{\xi^{-1}} & 
\Gamma(\{M\}, \mathcal{R}) \ar[r]\ar[d]_{\eta} & 0 \\
0\ar[r] & \overline{M}\ar[r]&V\ar[r] & \mathcal{R}_M \ar[r] & 0
 }
\]
$\hfill \Box$
\end{ex}

\medskip
We establish an additional property of the Pierce sheaf $\mathcal{R}\rightarrow X$.
\begin{lem}\label{lemindecstalk}
Each stalk $\mathcal{R}_M$ is \emph{indecomposable}, i.e.\
(cf.\ Subsection \ref{SSidem}) the only idempotents
in $V/\overline{M}$ are $0_{\overline{M}}$ and $\mathbf{1}_{\overline{M}}$.
\end{lem}
\begin{proof}
Suppose that $e{\in}V$ is such that $e{+}\overline{M}\in B(\mathcal{R}_M(V))$.\ Then
$e(-1)e{-}e\in\overline{M}$ and $e(m)e{\in}\overline{M}$ for $m{\not=} 0$.\
Then $\sigma_e$ and $\sigma_{e(-1)e}$ agree
at $M{\in}X$, and
therefore they agree on a basic open 
neighborhood $N_{-1}$ of $M$ by Lemma \ref{lemsigmaid}.\ 
Similarly, $\sigma_{e(m)e}$ \emph{vanishes}
on a basic open neighborhood $N_m$ of $M$ for $m{\not=} 0$.

\medskip
Now $e(m)e{=}0$ for $m{\gg}0$.\ So intersecting $N_{-1}$ and a sufficient
\emph{finite number} of $N_m$'s for $m{\geq} 0$, we obtain a basic open neighborhood
$N(f)$ of $M$ ($f{\in}B$) with the property that $\sigma_{e(m)e}$ \emph{vanishes} on $N(f)$ for all $m\geq 0$
and $\sigma_{e(-1)e}{=}\sigma_e$ on $N(f)$.\ Because $N(f)$ is a clopen set, we can find
a global section $\tau\in\Gamma(X, \mathcal{R})$ that agrees with $\sigma_e$ on $N(f)$ and \emph{vanishes} on $N(f)'$.

\medskip
Consider $\tau(m)\tau$ for $m{\geq}0${:}\ it vanishes on $N(f)'$ because
$\tau$ vanishes there, whereas it agrees with $\sigma_e(m)\sigma_e{=}\sigma_{e(m)e}$ on
$N(f)$ and therefore again vanishes because 
$\sigma_{e(m)e}$ does.\ So
$\tau(m)\tau{=}0$ for $m\geq 0$.\
Similarly,  $\tau(-1)\tau$ vanishes on $N(f)'$ and coincides with $\sigma_e(-1)\sigma_e{=}\sigma_e$
on $N(f)$.\ As a consequence, we have $\tau(-1)\tau{=}\tau$.\ Now by Lemma \ref{lemmaidem} we 
conclude that $\tau$ is an idempotent in 
$\Gamma(X, \mathcal{R})$, and by Theorem \ref{thmglobaliso} it follows that
$\tau{=}\sigma_g$ for some idempotent $g{\in} B$.\
Finally, because $\tau$ agrees with $\sigma_e$ on $N(f)$ we have 
\begin{eqnarray*}
e{+}\overline{M} {=} \sigma_e(M) {=} \tau(M) {=}\sigma_g(M) {=} g{+}\overline{M}{=}0_{\overline{M}}\ \mbox{or}\
 \mathbf{1}_{\overline{M}}
\end{eqnarray*}
according to whether $g{\in}M$ or $g{\notin} M$.\ This completes the proof of the Lemma.
\end{proof}

\begin{dfn}\label{dfnredbundle} Following  Pierce \cite{RSP}, we call an \'{e}tale bundle of vertex rings 
$\mathcal{R}{\rightarrow} X$ \emph{reduced} if (i)\ $X$ is a Boolean space and (ii)\ each stalk
is an indecomposable vertex ring.
\end{dfn}

By Theorem \ref{thmVXB} and Lemma \ref{lemindecstalk} we conclude
\begin{cor}\label{corredbundle} Suppose that $V$ is a vertex ring with Stone space $X$.\ Then
the Pierce bundle $\mathcal{R}{\rightarrow} X$ of $V$ is \emph{reduced}.\ \ $\hfill \Box$
\end{cor}

\section{Von Neumann regular vertex rings}\label{SvNr}

\subsection{Regular ideals}
\begin{dfn}\label{defregideal1} Let $V$ be a vertex ring with Boolean ring $B{=}B(V)$.\ A $2$-sided ideal $I{\subseteq}V$ is called \emph{regular} if it satisfies $I{=}\cup_{e\in I\cap B} e(-1)V$.
\end{dfn}

For example, the ideal $\overline{M}$ in (\ref{Mbardef}) is regular.\ Our aim in this Subsection is to describe
the lattice of all regular 2-sided ideals in a vertex ring $V$.

\medskip
Let $X$ be the Stone space of $V$, with $\mathcal{R}{\rightarrow}X$ 
the associated Pierce bundle of $V$.\ As usual $\mathcal{O}(X)$ denotes the set of open sets in $X$.

\medskip
 Thanks to Theorem \ref{thmglobaliso} we have an isomorphism of vertex rings 
$\xi{:}V{\stackrel{\cong}{\longrightarrow}}\Gamma_X$, so for our purposes it suffices to describe the regular 2-sided ideals of 
$\Gamma_X$.\
We  prove
\begin{thm}\label{thm2sidedreg}  If $U{\in}\mathcal{O}(X)$ then
$J[U]{\subseteq}\Gamma_X$ is a regular 2-sided ideal, and the map
\begin{eqnarray*}
\mathcal{O}(X){\rightarrow} \{\mbox{regular $2$-sided ideals in $\Gamma_X$}\},\ \  U{\mapsto}J[U]
\end{eqnarray*}
is an \emph{isomorphism of lattices} with inverse $J{\mapsto}U[J]$.\ (So the regular 2-sided ideals
are the kernels of the projections described in Lemma \ref{lemsesJ}.)
\end{thm}

Theorem \ref{thm2sidedreg} is a consequence of the following Lemmas \ref{lemmaidemclopen}-\ref{lemJUJ}.\

\begin{lem}\label{lemmaidemclopen} If $\sigma{\in}\Gamma(X, \mathcal{R})$ then $supp(\sigma){\subseteq}X$ is compact, and it is
a \emph{clopen} subset if $\sigma$ is an idempotent.
\end{lem}
\begin{proof} If $\sigma$ vanishes at $M{\in}X$ then it vanishes on an open neighborhood of $M$.\
This shows that the zero set of $\sigma$ is open, so that $supp(\sigma)$ is closed in $X$.\
Furthermore, because $\mathcal{R}{\rightarrow}X$ is reduced by Corollary \ref{corredbundle}, it follows that if $\sigma{\in} 
B(\Gamma(X, \mathcal{R}))$ then $\sigma(M)=0_{\overline{M}}$ or $\mathbf{1}_{\overline{M}}$
for all $M$.\ Since $\{M\in X{\mid}\sigma(M)=\mathbf{1}_{\overline{M}}\}$ is  open 
  then $supp(\sigma)$ is also closed.\ 
\end{proof}

\begin{lem}\label{lemUJU} If $U{\in}\mathcal{O}(X)$  then $J[U]$ is a regular $2$-sided ideal of 
$\Gamma_X$, and $U[J[U]]{=}U$.
\end{lem}
\begin{proof} That $J[U]$ is a 2-sided ideal follows from Lemma \ref{lemsesJ}.\ Let $\sigma{\in}J[U]$, so that 
$\sigma{\in}\Gamma_X$ with $supp(\sigma){\subseteq}U$.\ Because $supp(\sigma)$ is closed by Lemma \ref{lemmaidemclopen}, there is
a clopen set $N$ such that $supp(\sigma){\subseteq}N{\subseteq}U$.\ (Consider the open cover
$\{U, supp(\sigma)'\}$ of $X$.\ Then there is a partition of $X$ into  clopen sets, each contained in either $U$ or $supp(\sigma)'$, and $N$ is the union of those contained in $U$.) 

\medskip
We get an idempotent $\tau{\in}J[U]$ by setting $\tau(M){=}\mathbf{1}_{\overline{M}}\ (M{\in}N)$ together with
$\tau(M){=}0_{\overline{M}}\ (M{\notin} N)$.\ We have $\tau(-1)\sigma=\sigma$ because
$supp(\sigma)\subseteq supp(\tau)$, proving that
$J[U]$ is  regular.

\medskip
Finally, $U[J[U]]$ is the union of  $supp(\sigma)$ for $\sigma{\in}J[U]$, which coincides
with the union of $supp(\tau)$ for the idempotents $\tau{\in}J[U]$.\ By the argument of the first
paragraph, there is a clopen set $N'{\subseteq}U$ containing a given point of $U$ 
(points are closed because $X$ is Hausdorff), and by the argument of the second paragraph $N'$ is the support of an idempotent in $J[U]$.\ This shows that $U[J[U]]{=}U$, and the Lemma is proved.
\end{proof}

\begin{lem}\label{lemJUJ} If $J{\subseteq}\Gamma(X, \mathcal{R})$ is a regular 2-sided ideal 
then $U[J]{\subseteq}X$ is open and $J[U[J]]{=}J$.
\end{lem}
\begin{proof} If $\sigma{=}\tau\sigma$ then $supp(\sigma){\subseteq}supp(\tau)$.\ Thus if $J$ is regular then
the union of the supports $supp(\sigma)$ for $\sigma{\in}J$, i.e., $U[J]$, coincides with the union of the 
corresponding supports for the \emph{idempotents} in $J$.\ But this is open thanks to Lemma \ref{lemmaidemclopen}.

\medskip Now if $\sigma{\in}J[U[J]]$ then $supp(\sigma){\subseteq}U[J]=\cup_{\tau\in B(J)} supp(\tau)$.\
By compactness we get $supp(\sigma){=}\cup_{i=1}^n supp(\tau_i)$ where the $\tau_i\in B(J)$.\
Then $\sigma {=} \sigma(\tau_1\oplus \hdots\oplus \tau_n){\in}J$.\ This proves that $J[U[J]]{\subseteq}J$,
and since the opposite inclusion is obvious then $J{=}J[U[J]]$ and the proof of the Lemma is complete.
\end{proof}

\subsection{Von Neumann regular vertex rings}
Suppose that $R$ is a commutative, unital ring, and let $X{=}Spec(R)$.\ To provide some background and  context, 
we recall (\cite{AM}, \cite{Good}) that the following properties of $R$ are \emph{equivalent}{:}
\begin{eqnarray}\label{VNRequivs}
&&\mbox{$R/Nil(R)$ is a von Neumann regular ring (cf.\  Example \ref{exhomeo}), }\notag\\
&&\mbox{Every principal ideal in $R$ is generated by an idempotent,}\notag\\
&&\mbox{Every f.\ g.\  ideal in $R$ is generated by an idempotent,}\\
&&\mbox{$R$ has Krull dimension $0$}\notag\\
&&\mbox{$X$ is a Hausdorff space}\notag\\
&&\mbox{$X$ is totally disconnected}\notag
\end{eqnarray}
(Krull dimension $0$ means that every prime ideal  is maximal; $Nil(R)$ is the \emph{nilpotent radical} of $R$.)\
Toplogical spaces enjoying the last two properties are called \emph{Boolean}.

\medskip\noindent
\begin{ex} Let $V$ be a vertex ring with associated Boolean ring $B{=}B(V)$.\ Then
we showed in Section \ref{SSBV} that $B$ has Krull dimension $0$.\ Since every element is idempotent then
$Nil(R){=}0$, so $B$ is a commutative von Neumann ring.\ Therefore its prime spectrum $X{=}Spec(B)$ is  Boolean.
$\hfill\Box$
\end{ex}

\begin{dfn}\label{defregideal} A vertex ring $V$
 is called \emph{von Neumann regular} if \emph{every} 
principal 2-sided ideal of $V$ is generated by an idempotent.\ In other words, every principal
2-sided ideal of $V$ is equal to $e(-1)V$ for some $e{\in}B(V)$.
\end{dfn}

\begin{ex} If $R$ is a commutative von Neumann regular ring  then  by (\ref{VNRequivs}) $R$ is
a von Neumann regular vertex ring with trivial HS derivation.
\end{ex}

\begin{lem}\label{lem2sidedreg} Let $V$ be a von Neumann regular vertex ring.\ Then
\emph{every}  2-sided ideal $I{\subseteq}V$  is regular
in the sense of Definition \ref{defregideal}.
\end{lem}
\begin{proof} $I$ is certainly a \emph{sum} of principal 2-sided ideals, hence a sum of ideals generated by idempotents.\  That it is a \emph{union} of such ideals follows as in the proof of Lemma \ref{lembarM}.
\end{proof}

\begin{rmk} This argument shows that in a von Neumann regular vertex ring, every \emph{finitely generated}
2-sided ideal is generated by an idempotent.
\end{rmk}

\begin{cor}\label{coriso} Let $V$ be a von Neumann regular vertex ring with Stone space $X$.\ Then there is an 
\emph{isomorphism of lattices}
\begin{eqnarray*}
\mathcal{O}(X){\stackrel{\cong}{\longrightarrow}} \{\mbox{2-sided ideals of $V$} \}.
\end{eqnarray*}
\end{cor}
\begin{proof} By Corollary \ref{corredbundle} the Pierce bundle 
$\mathcal{R}{\rightarrow} X$ associated to $V$ is reduced.\ Then the present Corollary follows
from Theorem \ref{thm2sidedreg} and Lemma \ref{lem2sidedreg}.
\end{proof}

\begin{lem}\label{lemsimplestalk} Let $V$ be a von Neumann regular vertex ring.\ Then the 
Pierce bundle $\mathcal{R}{\rightarrow}X$ of $V$ is a bundle of \emph{simple} vertex rings.\
That is, each stalk $\mathcal{R}_M$ is a simple vertex ring.
\end{lem}
\begin{proof} According to Corollary \ref{coriso} the maximal ideals of $V$ correspond
to those open subsets of $X$ that are maximal (with respect to containment) among
all \emph{proper} open subsets of $X$.\ Since $X$ is Hausdorff, these are precisely the
subsets $X{\setminus}{\{M\}}\ (M{\in}X)$.\
By Lemma \ref{lemUJU} the maximal ideal of $\Gamma_X$ that corresponds to
$X{\setminus}{\{M\}}$ is $J[X{\setminus}{\{M\}}]$.\ Therefore $\Gamma_X/J[X{\setminus}{\{M\}}]
\cong \mathcal{R}_M$
is \emph{simple}, where the isomorphism follows from Example \ref{exses}.
\end{proof}

\begin{thm}\label{thmCVNvring} Let $V$ be a von Neumann regular vertex ring.\ Then the center $C(V)$ is a \emph{commutative von Neumann regular ring}.
\end{thm}
\begin{proof} \ Let $B{:=}B(V), C{:=}C(V)$ with $X$ the Stone space of $V$.\ First we show that if $M{\in}X$ then $C{+}\overline{M}/\overline{M}$ is a \emph{field}.\
Equivalently, setting $I{:=} C\cap \overline{M}$, we show that $I$ is a \emph{maximal ideal} of $C$.\ It is clear that
$I$ is a proper ideal of $C$.\ Now $\overline{M}$ is generated by the idempotents in $M$, and these all lie
in $C$.\ Thus $\overline{M}$ is generated by the idempotents in $I$.\ Because $\overline{M}$ is a maximal 2-sided ideal of $V$, it follows that the
2-sided ideal $N{:=}\sum_{a\in I} a(-1)V$ of $V$ generated by $I$ is equal to either $V$ or $\overline{M}$.

\medskip
We show that $N{=}V$ leads to a contradiction.\ To achieve this we need the following result:\
if $a{\in} I$ then $a(-1)V{=}e(-1)V$ for some $e{\in} B\cap I$.\ (We know that we can choose $e{\in} B$
because $V$ is von Neumann regular.\ The main point is that in fact $e{\in} I$.) 

\medskip
To prove this assertion, and starting from $a(-1)V{=}e(-1)V$ with $e{\in} B$, there are $u, v{\in} V$ such that 
$a(-1)u{=}e, e(-1)v{=}a$.\
Consider the decomposition $V{=}e(-1)V\oplus(\mathbf{1}-e)(-1)V$.\ Since $a{\in} e(-1)V$ then
we may, and shall, also assume that $u{\in} e(-1)V$.\ Then from
$a(-1)u{=}e$ it follows from Lemma \ref{lemmaunit} (with $e(-1)V$ in place of $V$) that $u{\in} C(e(-1)V)\subseteq C$.\
Since $I$ is an ideal in $C$ and $a{\in} I$ we deduce that $e{=}a(-1)u{\in} I$, as required.

\medskip
 From what we have just proved, we have $N{=}\sum_{e\in B\cap N}e(-1)V{=}
\bigcup_{e\in B\cap N}e(-1)V$.\ Then if $V{=}N$ we have $\mathbf{1}{\in} e(-1)V$ for some
$e{\in} B\cap I$.\ Thus $e(-1)u{=}\mathbf{1}$ for some $u{\in} V$, so that $e$ is a unit in $C$
by Lemma \ref{lemmaunit}.\ Since $e{\in} I$ this is a contradiction.

\medskip
Having shown that $N{\not=} V$, we must conclude that $N{=}\overline{M}$.\ Let $I{\subseteq} J{\subseteq} C$
where $J$ is a \emph{maximal} ideal of $C$.\ We show that $I{=}J$.\ We can proceed as before,
setting $N'{:=}\sum_{b\in N'} b(-1)V$.\ This is a 2-sided ideal of $V$ that contains $\overline{M}$, so
either $N'{=}V$ or $N'{=}\overline{M}$.\ If $N'{=}V$ we get a contradiction just as before,
so $N'{=}\overline{M}$.\ But then $N'$ is generated by the idempotents in $M$, and therefore it
must be equal to $\overline{M}$, as asserted.

\medskip
We have established that $C{+}\overline{M}/\overline{M}$ is a field for each $M{\in} X$.\ Therefore,
setting $\mathcal{S}{:=} \bigcup_{M\in X} C{+}\overline{M}/\overline{M}{\subseteq} \mathcal{R}$,
it follows that $\mathcal{S}{\rightarrow} X$ is an \'{e}tale bundle of \emph{commutative fields} over $X$.\ Then  we can
invoke one of the main results of \cite{RSP} (loc.\ cit.\ Theorem 10.3) to conclude that $C$ is a
von Neumann regular ring.\ This completes the proof of the Theorem.
\end{proof}

\begin{cor} Let $V$ be a von Neumann regular vertex ring with Stone space $X$ and center $C$.\ Then
the lattices of 2-sided ideals in $V$ and $C$ are each isomorphic to $\mathcal{O}(X)$.\ 
\end{cor}
\begin{proof}  $X$ is the
Stone space of $C$ as well as $V$, and because $C$ is von Neumann regular by Theorem
\ref{thmCVNvring} we can apply the commutative ring case of Corollary 9.10 with $C$ in place of $V$.\
Thus $\mathcal{O}(X)$ is also isomorphic to the lattice of 2-sided ideals of $C$, and the Corollary is proved.
\end{proof}

\section{Equivalence of some categories of vertex rings}\label{Seqcat}
In this Section we will discuss the following  diagram of categories and functors.\ The lower half
gives expression to some theorems of Pierce \cite{RSP} while the upper half reflects the extensions of these
theorems to the corresponding categories of vertex rings.

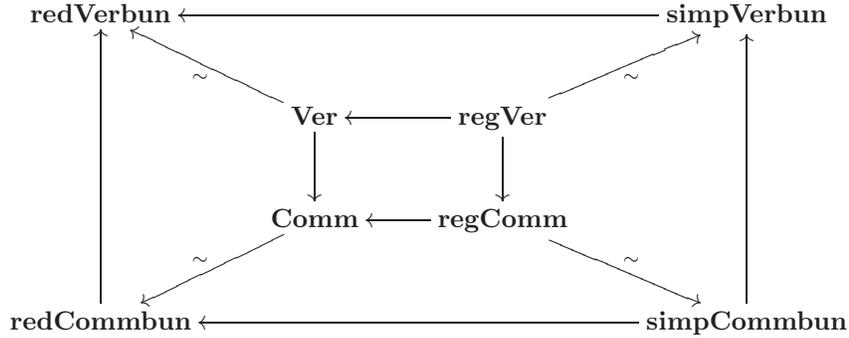
\begin{eqnarray*}
\xymatrix{
\bf{redVerbun}  & &&\bf{simpVerbun}\ar[lll] \\
 &\bf{Ver}\ar[d]\ar[lu]^{\sim} & \bf{regVer}\ar[ru]_{\sim}\ar[l] \ar[d]  \\
 & \bf{Comm} \ar[ld]_{\sim}  & \bf{regComm}\ar[l]\ar[rd]^{\sim} & \\
\bf{redCommbun}\ar[uuu] &&&\bf{simpCommbun}\ar[lll] \ar[uuu]
 }
\end{eqnarray*}

\medskip
The notation is supposed to be self-explanatory.\ $\bf{Ver}{\rightarrow} \bf{Comm}$ is the center functor
from vertex rings to commutative rings that is left adjoint to the insertion
$K{:}\bf{Comm}{\rightarrow}\bf{Ver}$.\ See Section 6, especially Theorem \ref{thmladjoint}.

\medskip
$\bf{regVer}$ and $\bf{regComm}$
are the full subcategories of $\bf{Ver}$ and $\bf{Comm}$ whose objects are
the von Neumann regular vertex rings and von Neumann regular commutative rings
respectively.\ That the inner square commutes then means that the center of
a von Neumann regular vertex ring is a von Neumann regular commutative ring, and
this is the content of Theorem \ref{thmCVNvring}.

\medskip
The four diagonals are categorical equivalences and require more discussion.\ The objects of 
 $\bf{redVerbun}$ are \emph{reduced \'{e}tale bundles
of vertex rings}.\ Recall (Definition \ref{dfnredbundle}) that these are \'{e}tale bundles 
$\mathcal{R}{\rightarrow} X$ of indecomposable
vertex rings over a Boolean base space $X$.\ We will define morphisms  shortly.\ Then
$\bf{simpVerbun}$, $\bf{redCommbun}$ and $\bf{simpCommbun}$ are the full subcategories whose objects are
bundles of \emph{simple} vertex rings, indecomposable commutative rings, and \emph{fields} 
(i.e., simple commutative rings) respectively.\ Thus  the sides
of the outer square are functorial insertions.

\medskip
Morphisms in categories such as $\bf{redVerbun}$ are standard, and are
described as follows.\ Given a pair of objects 
$\mathcal{R}\stackrel{\nu}{\longrightarrow} X, \mathcal{S}\stackrel{\pi}{\rightarrow} Y$, 
a morphism from the first object to the second is a pair of
continuous maps $(f, g)$ as in the diagram

\[\xymatrix{
&X\times_Y\mathcal{S}\ar[ld]_g\ar[d]\ar[r] &\mathcal{S}\ar[d]^{\pi} \\
\mathcal{R}\ar[r]_{\nu} &X\ar[r]_{f} & \bf{Y} 
 }
\]
where the square is a pull-back in $\mathbf{Top}$
and where the following property holds:\ for each $x\in X$, 
\begin{eqnarray*}
g(x, *){:}\mathcal{S}_{f(x)}{\rightarrow} \mathcal{R}_x
\end{eqnarray*}
is a morphism of vertex rings.

\medskip
The lower left diagonal equivalence  $\mathbf{Comm}\stackrel{\sim}{\longrightarrow}\mathbf{redCommbun}$ is a main result of Pierce (\cite{RSP}, Theorem 10.1).\ Moreover upon restricting to
the subcategory of vNr commutative rings, Pierce shows (\cite{RSP}, Theorem 10.3) that the corresponding bundles have stalks which are
\emph{fields} (i.e., simple commutative rings), thereby giving rise to the lower right categorical equivalence
$\mathbf{regComm}\stackrel{\sim}{\longrightarrow}\mathbf{simpCommbun}$

\medskip
These results extend to
equivalences of the corresponding categories of vertex rings.\
The first equivalence, namely $\mathbf{Ver}\stackrel{\sim}{\longrightarrow}\mathbf{redVerbun}$ may be established by
a  proof parallel to that of Pierce (loc.\ cit.)\ The object map assigns to a vertex ring $V$ its Pierce bundle $\mathcal{R}{\rightarrow}X$ described in
Subsection \ref{SSPiercebun}, while the inverse is the global sections functor that assigns to a reduced bundle 
$E{\rightarrow}X$ of vertex rings the vertex ring $\Gamma(X, E)$.\ Theorem \ref{thmglobaliso} shows that the composition of these two functors is equivalent to the identity functor, at least on objects.\ There are a multitude of additional details to check, most of them very similar to \cite{RSP}, and we skip the details here.\ The fourth and final equivalence is
 $\mathbf{regVer} \stackrel{\sim}{\longrightarrow} \mathbf{simpVerbun}$.\
The main point, and the result that we will actually prove, is the following.
\begin{thm}\label{thmsimplestalk} Suppose that $\mathcal{R}{\rightarrow} X$ is an \'{e}tale
bundle of  vertex rings over a Boolean space $X$.\  Then the vertex ring of global
sections $\Gamma(X, \mathcal{R})$ is von Neumann regular if, and only if, $\mathcal{R}$ is a bundle of
\emph{simple} vertex rings.
\end{thm}

\begin{proof} Set $V{:=}\Gamma(X, \mathcal{R})$.\ Then $\mathcal{R}{\rightarrow}X$ is isomorphic to the
Pierce bundle of $V$, so the  implication $\Rightarrow$ follows from Lemma \ref{lemsimplestalk}.

\medskip
Conversely, suppose that $\mathcal{R}{\rightarrow}  X$ is a bundle of simple rings.\ Then certainly each stalk is an
indecomposable vertex ring, so the sheaf is reduced (cf.\ Definition \ref{dfnredbundle}) and we have to show that 
$V$ is von Neumann regular.\
Let $\sigma{\in}\Gamma(X, \mathcal{R})$.\ It suffices to show that $\sigma {=} e(-1)\sigma$ for some idempotent
$e$, because then the 2-sided ideal generated by $\sigma$ is equal to $e(-1)V$.

\medskip
Suppose that $M{\in} supp(\sigma)$.\ Then $0{\not=}\sigma(M){\in}\mathcal{R}_M$, and since $\mathcal{R}_M$
is simple then it coincides with the 2-sided ideal generated by $\sigma(M)$.\ Therefore
there is an equation of the form
\begin{eqnarray*}
\sum v_1^M(n_1)\hdots v_k^M(n_k)\sigma(M)(i)\mathbf{1}_{\overline{M}} = \mathbf{1}_{\overline{M}}\ \ 
(v_j^M{\in}\mathcal{R}_M)
\end{eqnarray*}
We can find global sections $\tau_1^M, \hdots, \tau_k^M\in\Gamma(X, \mathcal{R})$ such that 
$\tau_j^M(M)=v_j^M\ (1{\leq} j{\leq} k)$.\ Then $\sum \tau_1^M(n_1)\hdots \tau_k^M(n_k)\sigma(i)\mathbf{1}$ 
takes the value $\mathbf{1}_M$ at $M$, and it follows that $\sum \tau_1^M(n_1)\hdots \tau_k^M(n_k)\sigma(i)\mathbf{1}$ 
and $\mathbf{1}$ agree on an open neighborhood $N_M$ of $M$.

\medskip
As $M$ ranges over $supp(\sigma)$ we get an open cover $\{N_M\}\cup supp(\sigma)'$ of $X$, and by the
partition property we can find a partition of $X$ into clopen sets,  each of which is contained in some $N_M$ or in $supp(\sigma)'$.\ It follows that there is a continuous section $\nu$ of the form
$\nu = \sum \tau_1(n_1)\hdots \tau_{\ell}(n_{\ell})\sigma(i)\mathbf{1}$ with the property that
$\nu(L)=\mathbf{1}_{\overline{L}}$ for all $L{\in} supp(\sigma)$; and if $L{\notin} supp(\sigma)$ then
$\sigma(L){=}0$, so that also $\nu(L){=}0$.\ This shows that $\nu$ is idempotent.\ Finally,
to check that $\nu(-1)\sigma{=}\sigma$, we only have to check it locally.\ But this is clear,
because $(\nu(-1)\sigma)(L){=}\nu(L)(-1)\sigma(L)$ and this is $\sigma(L)$ or $0$
according to whether $L{\in} supp(\sigma)$ or not.\ The proof of
Theorem \ref{thmsimplestalk} is complete.
\end{proof}

\section{Appendix}
If $m, n\in\ZZ$ we define
\begin{eqnarray*}
{m\choose n} = \left \{\begin{array}{ccc}
                                m(m-1)...(m-n+1)/n!&\mbox{if}\ n \geq 1\\
                                1&\ \mbox{if}\ n=0,\\
                                0&\mbox{if}\ n<0 \end{array} \right.
\end{eqnarray*}
We have ${m\choose n}\in\ZZ$, and the following identities hold for all $m, n, r$.
\begin{eqnarray}
{m\choose n} &=& (-1)^n{n-m-1\choose n}\label{bi1}, \\
{m\choose n} &=& {m-1\choose n}{+}{m-1\choose n-1} \label{bi2},\\
 {m\choose n}&=&\sum_{i=0}^n {r\choose i} {m-r\choose n-i} \label{bi3},\\
{m\choose n}&=& \sum_{i=0}^n (-1)^{i}{r\choose i}{r+m-i\choose n-i} \label{bi4}.
\end{eqnarray}
Care is warranted when dealing with these binomial coefficients.\ For example, the
`familiar' identity ${m\choose n} = {m\choose m-n}$ is not universally true.

\medskip
We use the following binomial expansion for all $m\in\ZZ$:
\begin{eqnarray}\label{binexp}
(z+w)^m=\sum_{n\geq 0}{m\choose n}z^{m-n}w^n.
\end{eqnarray}
That is, we always expand binomial expressions in nonnegative powers of the \emph{second}
variable.

\bibliographystyle{amsplain}

\end{document}